\documentclass[a4paper,reqno]{amsart}

\usepackage{amsmath,amsthm,amssymb,amsfonts,mathrsfs,hyperref,pifont,txfonts,geometry,mathptmx,graphicx}
\usepackage{esint}
\usepackage[usenames]{color}
\usepackage{slashbox}

\theoremstyle{plain}
\newtheorem{theorem}{Theorem}[section]
\newtheorem{corollary}[theorem]{Corollary}

\newtheorem{lemma}[theorem]{Lemma}
\newtheorem{proposition}[theorem]{Proposition}
\theoremstyle{definition}
\newtheorem{definition}{Definition}[section]
\theoremstyle{remark}
\newtheorem{remark}{Remark}[section]
\numberwithin{equation}{section}

\begin{document}

\title[Quasiconvexity at the boundary and the nucleation of austenite]
{Quasiconvexity at the boundary and the nucleation of austenite}
\author[J.~M.~Ball]{John M.~Ball}
\address{Oxford Centre for Nonlinear PDE\\
Mathematical Institute\\
University of Oxford\\
Andrew Wiles Building\\
Radcliffe Observatory Quarter\\
Woodstock Road\\
Oxford, OX2 6GG, UK
}
\email{ball@maths.ox.ac.uk}

\author[K.~Koumatos]{Konstantinos Koumatos}
\address{Gran Sasso Science Institute,\\
Viale Francesco Crispi 7\\
L'Aquila\\
67100\\
Italy
}
\email{koumatos@maths.ox.ac.uk}

\thanks{The research of both authors was supported by the EPSRC Science and Innovation award to the Oxford Centre for Nonlinear PDE (EP/E035027/1) and the European Research Council under 
the European Union's Seventh Framework Programme (FP7/2007-2013) / ERC grant agreement ${\rm n^o}$ 291053. The research of JMB was also supported by a Royal Society Wolfson Research Merit Award.
Both authors would like to thank H.~Seiner and R.~D.~James for helpful discussions, as well as the referees for their careful reading of the manuscript.}

\begin{abstract} 
Motivated by experimental observations of H.~Seiner \textit{et al.}, we 
study the nucleation of austenite in a single crystal of a CuAlNi shape-memory alloy stabilized as a single variant of martensite. In the experiments the nucleation 
process was induced by localized heating and it was observed that, regardless of where the localized heating was applied, the nucleation points were always 
located at one of the corners of the sample - a rectangular parallelepiped in the austenite.
Using a simplified nonlinear elasticity model, we propose an explanation for the location of the nucleation points by showing that the martensite is a 
local minimizer of the energy with respect to localized variations in the interior, on faces and edges of the sample, but not at some corners, where a 
localized microstructure, involving austenite and a simple laminate of martensite, can lower the energy. 
The result for the interior, faces and edges is established by showing that the free-energy function 
satisfies a set of quasiconvexity conditions at the stabilized variant in the interior, faces and edges, respectively, provided the specimen is suitably 
cut.
\end{abstract}

\maketitle



\section{Introduction}
\label{intro}

The purpose of this paper is to provide the mathematical analysis, and proposed explanation, for a remarkable experimental observation of Seiner \textit{et al} on a 
single crystal of CuAlNi; see \cite{icomat11} for details on the experimental procedure, observations and mathematical results.

In the experiment, the specimen was a parallelepiped of dimensions 12$\times$3$\times$3mm$^3$ in its high temperature phase, the austenite, with edges 
approximately along the cubic axes $(1,0,0)^T$, $(0,1,0)^T$, $(0,0,1)^T$ (see \cite{PhaseTran} for a detailed description). By applying a uniaxial 
compression along its longest edge, the specimen was transformed into a single variant of its low temperature phase, the martensite. However, due to an 
effect known as mechanical stabilization of martensite, the critical temperature for the transition back to austenite was significantly increased and the reverse 
transition did not occur during unloading.

The specimen was then locally heated by touching its surface with a heated iron tip with temperature electronically controlled at 200$^\circ$C 
(control accuracy $\sim \pm 5^\circ$C). We note that the temperature $\theta_c$ required for the transition back to austenite by homogeneous heating 
was approximately 60$^\circ$C, i.e.~significantly lower than the temperature of the iron tip.

The localized heating was applied in three different ways: (i) with the tip touching one of the corners surrounding the upper face; (ii) with the tip 
touching one of the edges, approximately in the middle between two corners; (iii) with the tip touching approximately at the centre of the upper face. 
These experiments were repeated with various faces chosen to be the upper (observed) ones.

When heating was applied at a corner, the nucleation was always induced exactly at that corner and occurred nearly immediately after touching the 
specimen with the tip. When heating was applied either on an edge or at the centre of the upper face, the nucleation occurred at one of the corners as well, i.e.~the 
localized heating did not result in formation of the nucleus under the tip. Moreover, the nucleus was only observable after 30-60 s, which was enough 
time for the corner to reach the increased critical temperature by thermal conduction. In different tests the nuclei were observed at different corners. 
After the nucleation, the transition front formed and propagated through the specimen.

In Fig.~\ref{fig3_HS}, snapshots from the observations are seen. The transition fronts have morphologies of the interfacial microstructures described in 
\cite{PhaseTran} ($X$- and $\lambda$-interfaces), in which the mechanically stabilized martensite is separated from austenite by a twinned region 
ensuring kinematical compatibility.

\begin{figure*}[ht]
\centering
\includegraphics[width=1\textwidth]{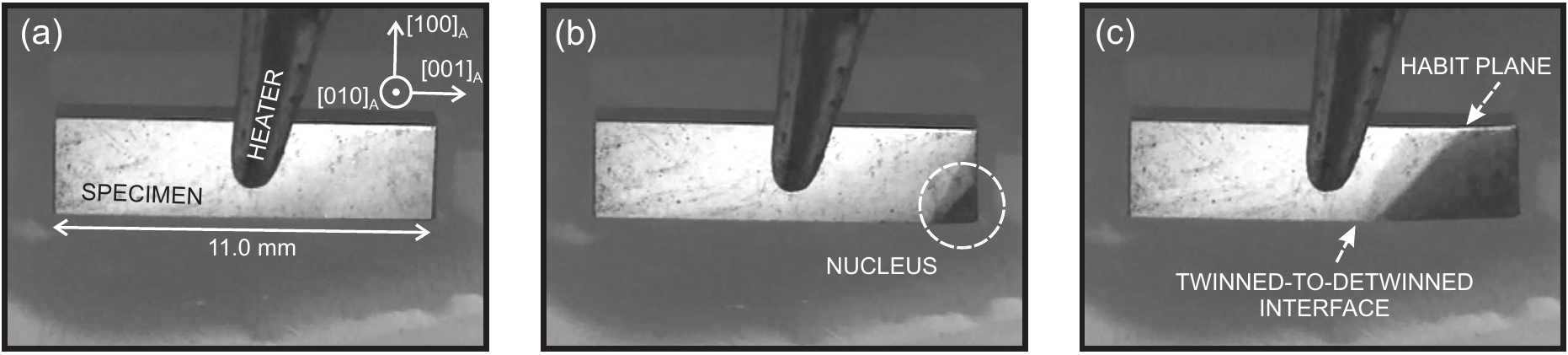}
\caption{Snapshots of the recorded video taken during the optical observations of the nucleation process (courtesy of H.~Seiner). (a) the initial state with the length and 
crystallographic orientation of the specimen given in the coordinate system of the austenitic lattice (indicated by the subscript A); (b) 
formation of the nucleus at a corner (the first frame of the recorded video in which the nucleus was clearly visible); (c) the fully formed transition 
front propagating through the specimen. The morphology of the interfacial microstructure is outlined by the arrows indicating the austenite-to-twinned 
martensite interface (the habit plane) and the twinned-to-detwinned interface between the laminate and the stabilized martensite.}
\label{fig3_HS}
\end{figure*}

It is only natural to treat this problem as one concerning local minimizers, and using a simplified nonlinear elasticity model we propose an 
explanation for the location of the nucleation points based on ideas of the modern calculus of variations. To be more specific, it is shown that the 
free-energy density of our simplified model must be quasiconvex in the interior, as well as quasiconvex at faces and edges of any suitably oriented 
convex polyhedral domain, Seiner's specimen belonging to this class of domains. In this analysis we assume that the initial state is a pure variant of martensite, as in Seiner's experiment. For some comments on the general case of nucleation of austenite from a microstructure of martensite see Section \ref{sec:4.2}.

The problem is expressed in terms of Young measures, and our quasiconvexity conditions in the interior and at faces are classical in the sense that they 
amount to the quasiconvexity condition of Morrey~\cite{morrey52} and the quasiconvexity at the boundary condition of Ball \& Marsden~\cite{ballmarsden} 
expressed in terms of Young measures. As for quasiconvexity at edges and corners, that is at the non-smooth parts of the boundary, these conditions can be defined 
analogously to quasiconvexity at the boundary. The quasiconvexity 
conditions for the stabilized variant in the interior, faces and edges imply that the stabilized variant is a minimizer of the energy with respect to 
localized variations in the interior, faces and edges respectively and nucleation of austenite cannot occur there. As for the corners, an explicit construction shows 
that, for the crystallographic directions of Seiner's specimen and at some corners, a specific microstructure containing austenite can lower the energy.

The paper is structured as follows: though in a simplified setting, our quasiconvexity-based approach is largely inspired by the remarkable work of 
Grabovsky \& Mengesha~\cite{grabovsky2009} on sufficient conditions for strong local minimizers, where an answer to the 
conjecture in~\cite{ballconjecture} on quasiconvexity-based sufficient conditions is given. In Section~\ref{sec:1}, we recall the 
classical definitions of quasiconvexity and briefly review the results in~\cite{grabovsky2009}. Further remarks on possible connections with our work 
and~\cite{grabovsky2009} are made in the concluding section of this paper. We also give a brief account of gradient Young measures along with standard 
results used in the paper and we present the general nonlinear elasticity model used to analyze microstructure formation.

In Section~\ref{sec:1a}, we take a small detour into the invertibility of Sobolev mappings and the interpenetration of matter. In particular, in our model, we will 
require invertibility for the maps underlying our admissible measures, as well as a certain regularity for the inverse map. This is achieved by employing the condition of 
Ciarlet \& Ne{\v{c}}as~\cite{ciarletnecas} and exploiting the theory of mappings with bounded distortion, all of which are introduced in this section along with 
auxiliary results used in the subsequent analysis. Our simplified elasticity model is introduced in Section~\ref{sec:2}, where $\Gamma$-convergence is 
employed in order to derive an appropriate energy functional and the set of admissible gradient Young measures for the problem is introduced.

Section~\ref{sec:3} is devoted to the proof of our main result, Theorem~\ref{theorem:main}, on the location of the nucleation points; our set of quasiconvexity conditions 
is also introduced and auxiliary results are established. Quasiconvexity at faces and edges is naturally dependent on the 
orientation of the convex polyhedral domain, i.e.~the direction of edges and face normals. The defining property of 
these directions is independent of the change of symmetry of the crystal lattice from austenite to martensite, but the directions themselves are 
dependent on the transformation; as such, for our main result to become applicable to Seiner's specimen these directions need to be identified for the 
cubic-to-orthorhombic transition of CuAlNi. Explicit formulae for these are given in Section~\ref{sec:4.1}. We warn the reader that the calculations are lengthy and, upon first reading, Section~\ref{sec:4.1} can be ignored. Instead, the reader can directly proceed to the concluding remarks in Section~\ref{sec:4.2}.

\section{Preliminaries}
\label{sec:1}

\subsection{The Weierstrass problem}
\label{sec:1.1}

Consider a general variational problem of the form
\[I(y)=\int_{\Omega}W(x,y(x),Dy(x))\,dx\]
over the class of admissible deformations
\[\mathcal{A}=\left\{y:\Omega\rightarrow\mathbb{R}^{N}:\,y\in\,X\:\:\mbox{and}\:\: y_{\vert\partial\Omega_{1}}=\bar{y}\right\}\]
where $\Omega\subset\mathbb{R}^{d}$ is a bounded Lipschitz domain, $\partial\Omega_1$ is a relatively open subset of $\partial\Omega$, 
$W:\overline{\Omega}\times\mathbb{R}^{N}\times\mathbb{R}^{N\times d}\rightarrow\mathbb{R}$ 
satisfies certain conditions, $X$ is an appropriate function space and $\bar{y}\in X$ is a specified mapping.

\begin{definition}
We say that $y_{0}\in\mathcal{A}$ is a \textit{strong local minimizer} of $I$ if there exists an $\epsilon>0$ such that 
$I(y_{0})\leq\,I(y)$ for all $y\in\mathcal{A}$ with $\| y_{0}-y\|_{\infty}<\epsilon$.\\
Similarly, $y_{0}\in\mathcal{A}$ is a \textit{weak local minimizer} of $I$ if there exists an $\epsilon>0$ such that 
$I(y_{0})\leq\,I(y)$ for all $y\in\mathcal{A}$ with $\| y_{0}-y\|_{1,\infty}<\epsilon$.
Here, $\|\cdot\|_{\infty}$ and $\|\cdot\|_{1,\infty}$ denote the norms in the spaces $L^{\infty}(\Omega,\mathbb{R}^N)$ and $W^{1,\infty}(\Omega,\mathbb{R}^N)$ 
respectively.
\end{definition}

The Weierstrass problem consists in finding necessary and sufficient conditions for a deformation $y_{0}\in\mathcal{A}$ to be a strong local minimizer of 
$I$. Of course, the function space $X$ and the conditions on the stored energy function $W$ are themselves part of the problem, in particular so that $I(y)$ is well defined for $y\in\mathcal A$. 
This is an old and long-standing problem in the calculus of variations which is fairly well understood in the so-called scalar cases of $d=1$ 
(Weierstrass) or $N=1$ (Hestenes \cite{hestenessufficiency}). However, in the vectorial case of $d\geq 2$, $N\geq 2$, which is of interest to us, the problem 
remained largely open until recently. For the convenience of the reader let us recall the definition of quasiconvexity~\cite{morrey52}:
\begin{definition}
We say that a function $f:\mathbb{R}^{N\times d}\rightarrow\mathbb{R}$ is quasiconvex at $F\in\mathbb{R}^{N\times d}$ if for $B=B(0,1)$ - the unit ball in $\mathbb{R}^d$ - 
and all $\psi\in W^{1,\infty}_{0}(B,\mathbb{R}^N)$,
\begin{equation}
f(F)\leq\frac{1}{\mathcal{L}^d(B)}\int_{B}f(F+D\psi(x))\,dx,
\end{equation}
where $\mathcal{L}^d$ denotes $d$-dimensional Lebesgue measure, whenever the integral exists. We say that $f$ is quasiconvex if it is quasiconvex at every $F\in\mathbb{R}^{N\times d}$.
\label{def:fqc}
\end{definition}

\begin{remark}
Under the above definition, quasiconvex functions are rank-one convex and hence continuous, see e.g. \cite[Lemma 4.3]{mullernotes}.
\end{remark}

With this definition at hand, let $X=C^{1}(\overline{\Omega},\mathbb{R}^{N})$, the space of functions $y\in C^{1}(\Omega,\mathbb{R}^N)$ which can be extended 
to a continuously differentiable function in an open set containing $\overline{\Omega}$. Also, for $y_0$ a strong local minimizer of $I$ in $\mathcal{A}$, let 
$\mathcal{R}=\left\{(y_0(x),Dy_0(x)):\,x\in\overline{\Omega}\right\}$. Assume that $W(x,y,F)$ is 
continuous and that its partial derivatives of first and second order in $(y,F)$ exist and are continuous on $\overline{\Omega}\times\mathcal{O}$ where 
$\mathcal{O}$ is a bounded and open neighbourhood of $\mathcal{R}$ in $\mathbb{R}^N\times \mathbb{R}^{N\times d}$. Letting
\[\rm{Var}(\mathcal{A})=\left\{\varphi\in\,C^{1}(\overline{\Omega},\mathbb{R}^{N}):\varphi_{\vert\partial\Omega_{1}}=0\right\},\]
known necessary conditions for the map $y_{0}$ to be a strong local minimizer of $I$ are the following:
\begin{itemize}
\item [(i)] Satisfaction of the weak form of the Euler-Lagrange equations, i.e.
\[\int_{\Omega}\left[W_{y}(x,y_{0}(x),Dy_{0}(x))\cdot\varphi(x)+W_{F}(x,y_{0}(x),Dy_{0}(x))\cdot D\varphi(x)\right]dx=0\]
for all $\varphi\in \rm{Var}(\mathcal{A})$, where $W_{y}$ and $W_{F}$ denote the derivatives of $W$ with respect to $y$ and $F=Dy$ respectively.
\item [(ii)] Positivity of the second variation, i.e.
\[\delta^{2}I(y_{0})=\frac{d^{2}}{d\epsilon^{2}}I(y_{0}+\epsilon\varphi)\vert_{\epsilon=0}\geq 0\]
for all $\varphi\in \rm{Var}(\mathcal{A})$.
\item [(iii)] Quasiconvexity in the interior, i.e.~for all $x_{0}\in\Omega$
\[\int_{B}W(x_{0},y_0(x_{0}),Dy_0(x_{0})+D\varphi(x))\,dx\geq\int_{B}W(x_{0},y_0(x_{0}),Dy_0(x_{0}))\,dx\]
for all $\varphi\in W^{1,\infty}_{0}(B,\mathbb{R}^{N})$ where $B=B(0,1)$ denotes the unit ball in $\mathbb{R}^{d}$; that is, $W(x_0,y_0,\cdot)$ is quasiconvex at 
$Dy_0(x_0)$.
\item [(iv)] Quasiconvexity at the boundary, i.e.~for all $x_{0}\in\partial\Omega\setminus\overline{\partial\Omega_{1}}$ - the free boundary - in the neighbourhood of which $\partial\Omega$ is $C^1$
\[\int_{B^{-}_{n(x_{0})}}W(x_{0},y_0(x_{0}),Dy_0(x_{0})+D\varphi(x))\,dx\geq\int_{B^{-}_{n(x_{0})}}W(x_{0},y_0(x_{0}),Dy_0(x_{0}))\,dx\]
for all $\varphi\in V_{n(x_{0})}$ where $n(x_{0})$ is the outward unit normal to $\Omega$ at $x_{0}$ and
\[V_{n(x_{0})}=\left\{\varphi\in W^{1,\infty}(B^{-}_{n(x_{0})},\mathbb{R}^{N})\,:\,\varphi\equiv 0\,\,
\mbox{on}\,\,\partial B\cap\partial{B^{-}_{n(x_{0})}}\right\}.\]
Here, $B^{-}_{n(x_{0})}=\left\{x\in B\,:\,x\cdot n(x_{0})<0\right\}$.
\end{itemize}
Conditions (i) and (ii) are classical. The necessity of condition (iii) is due to Meyers~\cite{meyers65} 
(see also~\cite{31}) and can be thought of as a multi-dimensional 
version of the Weierstrass positivity condition; in fact, it reduces to it when $d=1$. Condition (iv) was introduced by Ball \& 
Marsden~\cite{ballmarsden} and is a genuinely new condition valid for the vectorial case. For further references regarding the notion of quasiconvexity at the boundary, the reader is referred to \cite{kruzikqcboundary,mielkeqcboundary,silhavybook}.

For such $C^{1}$ maps $y_0$, in their seminal paper~\cite{grabovsky2009}, Grabovsky and Mengesha showed that a slightly strengthened version of the above set of 
necessary conditions is in fact sufficient for $y_{0}$ to be a strong local minimizer of $I$; see also Grabovsky \& Mengesha~\cite{grabovsky2007} and 
Grabovsky~\cite{grabovsky2008} for investigations on the problems of sufficient conditions for $W^{1,\infty}$ weak$\ast$ local minimizers and necessary conditions for 
$W^{1,\infty}$ strong local minimizers, respectively.

Grabovsky and Mengesha base their sufficiency proof on a decomposition lemma (see also \cite{fonseca1998analysis,kristensen1994finite}) 
which splits arbitrary variations of the dependent variable into a strong and a weak part. The core of their proof lies in showing that these two parts 
act on the functional independently. In particular, they show that the action of the weak part can be described in terms of the second variation, 
whereas, the action of the strong part is `localized', in the sense that it can be described as a superposition of `Weierstrass needles'. Then, it is 
the (uniform) positivity of the second variation and the (uniform) quasiconvexity conditions, respectively, that prevent the weak and the strong part 
from decreasing the functional.

Although it may be possible to extend the work in~\cite{grabovsky2009} to domains with edges and corners, it is formulated for smooth domains and everywhere defined and continuous $W$, and is not directly applicable to 
our case. However in this paper we restrict attention to localized variations only, corresponding to localized nucleation, and so it is the 
quasiconvexity conditions in the interior and at faces and edges of the parallelepiped domain which will be seen to prevent these localized variations from lowering the 
energy.

\subsection{Gradient Young measures}
\label{sec:1.2}

Young measures, introduced by L.C. Young~\cite{35}, are families of probability measures carrying the minimal information about a sequence 
$\left\{z^{k}\right\}$ that is necessary (under suitable hypotheses) to compute the weak limit of $f(z^{k})$ for continuous functions $f$. 

Let $\Omega\subset\mathbb{R}^d$ be a bounded Lipschitz domain and $C_{0}(\mathbb{R}^{N\times d})$ be the closure under the supremum norm of compactly supported, continuous functions on 
$\mathbb{R}^{N\times d}$, i.e.~the set of continuous functions from $\Omega$ to $\mathbb{R}^{N\times d}$ vanishing at infinity. By $\mathcal{M}(\mathbb{R}^{N\times d})$ 
we denote the dual space of $C_{0}(\mathbb{R}^{N\times d})$ consisting of signed Radon measures with finite mass equipped with the dual norm 
of total variation.

We say that a map $\mu:\Omega\rightarrow\mathcal{M}(\mathbb{R}^{N\times d})$ is weak$\ast$ measurable if the functions $x\mapsto\langle\mu(x),\psi
\rangle$ are measurable for all $\psi\in C_{0}(\mathbb{R}^{N\times d})$, where $\langle\mu(x),\psi\rangle=\int_{\mathbb{R}^{N\times d}}\psi\,d\mu(x)$.

Henceforth we denote the space of essentially bounded weak$\ast$ measurable functions from $\Omega$ to $\mathcal{M}(\mathbb{R}^{N\times d})$ by 
$L^{\infty}_{w^{\ast}}(\Omega, \mathcal{M}(\mathbb{R}^{N\times d}))$. Since $C_{0}(\mathbb{R}^{N\times d})$ is separable, 
$$L^{\infty}_{w^{\ast}}(\Omega, \mathcal{M}(\mathbb{R}^{N\times d}))= L^{1}(\Omega, C_{0}(\mathbb{R}^{N\times d}))^{\ast}$$ 
and is a Banach space when equipped with the norm $\|\mu\|=\mathrm{ess\,sup}_{x\in\Omega}\,\|\mu_{x}\|_{\mathcal{M}(\mathbb{R}^{N\times d})}$. 

A Young measure $\nu=(\nu_{x})_{x\in\Omega}$ is a map in $L^{\infty}_{w^{\ast}}(\Omega,\mathcal{M}(\mathbb{R}^{N\times d}))$ 
taking values in the space of probability measures. A classical result concerning Young measures (e.g.~\cite{ballym}) is that, for every 
sequence $\left\{z^k\right\}$ uniformly bounded in $L^{p}(\Omega,\mathbb{R}^{N\times d})$, $p>1$, there exists a subsequence (not relabelled) and a Young 
measure $\nu=(\nu_{x})_{x\in\Omega}$ such that, for all continuous functions $f:\mathbb{R}^{N\times d}\rightarrow\mathbb{R}$ satisfying 
$\vert f(F)\vert\leq c(1+\vert F\vert^q)$ with $1\leq q<p$,
\begin{equation}
\label{eq:ymlimit}
 f(z^k)\rightharpoonup\langle\nu_x,f\rangle:=\int_{\mathbb{R}^N}f(A)\,d\nu_x(A)\;\mbox{in}\;L^{p/q}(\Omega).
\end{equation}
In particular, since $z^k$ is bounded in $L^{p}(\Omega,\mathbb{R}^{N\times d})$,
\begin{equation}
\label{eq:ymlimitfns}
 z^k\rightharpoonup\langle\nu_x,\mathrm{id}\rangle=\int_{\mathbb{R}^N}A\,d\nu_x(A)\;\mbox{in}\;L^{p}(\Omega,\mathbb{R}^{N\times d}).
\end{equation}
Note that for $p=\infty$, (\ref{eq:ymlimit}) holds weak$\ast$ in $L^{\infty}(\Omega)$ and for all continuous $f:\mathbb{R}^{N\times d}\rightarrow\mathbb{R}$; similarly for (\ref{eq:ymlimitfns}).

We say that the Young measure $\nu=(\nu_{x})_{x\in\Omega}$ is generated by the sequence $z^k$ bounded in $L^{p}(\Omega,\mathbb{R}^{N\times d})$ if (\ref{eq:ymlimit}) 
holds. An important observation in the context of microstructure formation is that, for a compact set $K\subset\mathbb{R}^{N\times d}$,
\begin{equation}
 \mathrm{dist}(z^k,K)\rightarrow0\;\;\mbox{in measure}\,\Leftrightarrow\,\mathrm{supp}\,\nu_x\subset K\;\;\mbox{a.e.~in $\Omega$.}
\end{equation}

We note that there are more general versions of the above result (see e.g.~\cite{ballym}); however, we shall only be interested in Young measures 
generated by sequences bounded in $L^{p}(\Omega,\mathbb{R}^{N\times d})$ (in fact, mostly $L^{\infty}(\Omega,\mathbb{R}^{N\times d})$) and we do not elaborate further.

\begin{remark}
Given functions $\xi\in\,L^{1}(\Omega)$ and $f\in\,C_{0}(\mathbb{R}^{N\times d})$, the tensor product $\xi\otimes f$ denotes the element of 
$L^{1}(\Omega, C_{0}(\mathbb{R}^{N\times d}))$ given by $x\mapsto\xi(x)f$. We note that the span of such tensor products is dense 
in $L^{1}(\Omega, C_{0}(\mathbb{R}^{N\times d}))$. Then $\nu^{k}\in\,L^{\infty}_{w^{\ast}}(\Omega,\mathcal{M}(\mathbb{R}^{N\times d}))$ 
converges weak$\ast$ to $\nu$ in $L^{\infty}_{w^{\ast}}(\Omega, \mathcal{M}(\mathbb{R}^{N\times d}))$ if and only if $(\nu_k)$ 
is norm bounded in $L^{\infty}_{w^{\ast}}(\Omega, \mathcal{M}(\mathbb{R}^{N\times d}))$ and
\begin{equation}
\int_{\Omega}\xi(x)\langle\nu^{k},f\rangle\,dx\rightarrow\int_{\Omega}\xi(x)\langle\nu,f\rangle\,dx\nonumber
\end{equation}
for all $\xi\in\,L^{1}(\Omega)$, $f\in\,C_{0}(\mathbb{R}^{N\times d})$.
\end{remark}

We denote by $\mathcal{G}^{p}(\Omega,\mathbb{R}^{N\times d})$ the set of $W^{1,p}$ \textit{gradient Young measures}, i.e.~those Young measures generated by a 
sequence of gradients $Dy^k$ such that the sequence $y^k$ is uniformly bounded in $W^{1,p}(\Omega,\mathbb{R}^N)$. The following theorem due to Kinderlehrer \& 
Pedregal~\cite{36,37} provides a full characterization of $W^{1,p}$ gradient Young measures.

\begin{theorem}
\label{KPGYM}
A family $(\nu_{x})_{x\in\Omega}$ of probability measures on $\mathbb{R}^{N\times d}$, depending measurably on $x$, belongs to the space 
$\mathcal{G}^{p}(\Omega,\mathbb{R}^{N\times d})$ if and only if
\begin{itemize}
\item[(i)] $\int_{\Omega}\langle\nu_x,\vert\cdot\vert^p\,dx<\infty$ for $1 < p <\infty$ or $\mathrm{supp}\,\nu_{x}\subset K$ a.e.~for some compact set 
$K\subset\mathbb{R}^{N\times d}$ when $p=\infty$;
\item[(ii)] $\bar{\nu}_{x}=\langle\nu_{x},\mathrm{id}\rangle=Dy(x)$ a.e.~for some $y\in W^{1,p}(\Omega,\mathbb{R}^N)$ 
(referred to as the map underlying $\nu$);
\item[(iii)] $\langle\nu_{x},f\rangle\geq f(\bar{\nu}_{x})$ a.e.~for all quasiconvex $f:\mathbb{R}^{N\times d}\rightarrow\mathbb{R}$ satisfying 
$\vert f(A)\vert\leq c(1+\vert A\vert^p)$ if $1 < p <\infty$; for $p=\infty$ no growth condition is required.
\end{itemize}
\end{theorem}
\noindent In particular, we note that every $W^{1,p}$ gradient Young measure , $p>1$, satisfies the \textit{minors relations:}
\begin{equation}
\label{eq:minors0}
\langle\nu_{x},J\rangle=J(\bar{\nu_{x}})
\end{equation}
for all subdeterminants $J=J(A)$ of order $s\leq p$. This follows from Theorem~\ref{KPGYM} and the fact that minors are quasiaffine, 
i.e.~both $\pm J$ are quasiconvex.

We end our discussion on Young measures by giving a few remarks concerning $W^{1,\infty}$ gradient Young measures which will be of importance to us: in~\cite{36}, 
Kinderlehrer and Pedregal showed that if $(\nu_{x})_{x\in\Omega}$ is a $W^{1,\infty}$ gradient Young measure, then for a.e.~$\alpha\in\Omega$, the probability measure $\nu=\nu_{\alpha}$ is a homogeneous 
(i.e.~$x$-independent) $W^{1,\infty}$ gradient Young measure. Moreover, they showed that whenever $\nu=(\nu_{x})_{x\in\Omega}$ is a $W^{1,\infty}$ gradient Young measure 
with underlying map $y$ such that $y-Fx\in W^{1,\infty}_0(\Omega,\mathbb{R}^N)$, $F\in\,\mathbb{R}^{N\times d}$, then the measure $\mathrm{Av}\,\nu$ defined for all continuous 
$\psi:\mathbb{R}^{N\times d}\rightarrow\mathbb{R}$ by
\begin{equation*}
\label{eq:averaging}
 \langle\mathrm{Av}\,\nu,\psi\rangle=\frac{1}{\mathcal{L}^d(\Omega)}\int_{\Omega}\langle\nu_{x},\psi\rangle\,dx
\end{equation*}
is a homogeneous $W^{1,\infty}$ gradient Young measure with $\overline{\mathrm{Av}\,\nu}=F$. This is referred to later as the \textit{averaging} of Young measures.

Lastly, we recall a technical result due to Zhang~\cite{38} known as \textit{Zhang's Lemma}. This says that if a sequence 
$Dy^{k}$ is bounded in $L^{p}$, $p>1$, and generates a Young measure $(\nu_{x})_{x\in\Omega}$ with $\mathrm{supp}\,\nu_{x}\subset K$ a.e.~where 
$K\subset \mathbb{R}^{N\times d}$ is compact, then there exists a sequence $z^{k}$ with $Dz^{k}$ bounded in $L^{\infty}$ which generates the same 
Young measure $(\nu_{x})_{x\in\Omega}$. In other words, a $W^{1,p}$ gradient Young measure with compact support is a $W^{1,\infty}$ gradient Young 
measure. In the context of martensitic transformations, this will enable us to restrict attention to $W^{1,\infty}$ gradient Young measures and, respectively, 
underlying deformations in $W^{1,\infty}(\Omega,\mathbb{R}^{N})$.

\subsection{General elasticity model for microstructures}
\label{sec:1.3}

The general nonlinear elasticity model for martensitic transformations~\cite{1,4}, which neglects interfacial energy, leads to the prediction of infinitely fine 
microstructures which are identified with limits of infimizing sequences $y^{k}$, $k=1,2,\ldots$, for a total free energy
\begin{equation}
\label{eq:functionalin1.3}
E_{\theta}(y)=\int_{\Omega}\varphi(Dy(x),\theta)\,dx.
\end{equation}
Here, $\Omega\subset\mathbb{R}^3$ is a bounded Lipschitz domain representing the region occupied in the reference configuration by the undistorted austenite at the critical temperature $\theta_{c}$ and 
$y(x)\in\mathbb{R}^3$ denotes the deformed position of the particle $x\in\Omega$. The free-energy density $\varphi(F,\theta)$ depends on the 
deformation gradient $F\in\mathbb{R}^{3\times 3}$ and the temperature $\theta$. By frame indifference, $\varphi(RF,\theta)=\varphi(F,\theta)$ for all 
$F$, $\theta$ and for all $R\in\,SO(3)=\left\{R\in\mathbb{R}^{3\times 3}:R^{T}R=\mathbf{1},\,\det{R}=1\right\}$. Let
\begin{equation}
K_{\theta}=\lbrace F:\varphi(G,\theta)\geq\varphi(F,\theta)\;\mbox{for all $G\in\mathbb{R}^{3\times 3}$}\rbrace\nonumber
\end{equation}
denote the set of energy-minimizing deformation gradients. Then we assume that
\[K_{\theta}=\left\{\begin{array}{ll}
\alpha(\theta)SO(3)\mbox{ - austenite}&\,\theta>\theta_{c}\\
SO(3)\cup\bigcup^{N}_{i=1}SO(3)U_{i}(\theta_c)&\,\theta=\theta_{c}\\
\bigcup^{N}_{i=1}SO(3)U_{i}(\theta)\mbox{ - martensite}&\,\theta<\theta_{c},
\end{array}\right.\]
where the positive definite, symmetric matrices $U_{i}(\theta)$ correspond to the $N$ distinct variants of martensite and $\alpha(\theta)$ is 
the thermal expansion coefficient of the austenite with $\alpha(\theta_{c})=1$. We note that the matrices $U_i$ are symmetry related in the sense that
$$\left\{U_1,\ldots, U_N\right\}=\left\{Q^TU_1Q: Q\in\mathcal{P}^a\right\},$$
where $\mathcal{P}^a$ denotes the symmetry group of the austenite. In our case of cubic austenite $\mathcal{P}^a=\mathcal{P}^{24}$ - the subgroup of $SO(3)$ consisting of 
the 24 rotations mapping the unit cube to itself.

The weak$\ast$ limit $Dy$ of the gradients $Dy^k$ of an infimizing sequence corresponds to the macroscopic deformation gradient. However, information is lost in taking 
this limit and a more complete way to describe microstructure is via the use of gradient Young measures, described above. Then we seek to minimize
\[I_{\theta}(\nu)=\int_{\Omega}\langle\nu_{x},\varphi\rangle\,dx=\int_{\Omega}\int_{\mathbb{R}^{3\times3}}\varphi(A)\,d\nu_{x}(A)dx\]
over the space of gradient Young measures. In this case, the underlying (macroscopic) deformation gradient $Dy(x)$ corresponds to the 
centre of mass of $\nu$, i.e.~$Dy(x)=\bar{\nu}_{x}$ (see \cite{4}). 

As an example of the use of Young measures in the context of microstructures, consider the homogeneous measure 
$\nu=\lambda\delta_F+(1-\lambda)\delta_G$, for some $\lambda\in(0,1)$, supported on two rank-one connected matrices $F$ and $G=F+a\otimes n$ 
where $a$, $n$ are vectors and $\delta_{\cdot}$ denotes a Dirac mass. This Young measure is generated by gradients $Dy^k$, uniformly bounded in 
$L^{\infty}$, consisting of simple laminates formed from alternating layers with normal $n$ of width $\lambda k^{-1}$ and $(1-\lambda)k^{-1}$ in 
which $Dy^k$ takes the respective values $F$ and $G$. 
At each $x$, $\nu_x$ gives the limiting probabilities $\lambda$, 
$1-\lambda$ as $k\rightarrow\infty$ of finding the matrices $F$ and $G$, respectively, in an infinitesimal neighbourhood of $x$ (see~\cite{ballym} for 
a precise statement of this probabilistic interpretation). In this case, the macroscopic gradient is 
$Dy(x)=\bar{\nu}_{x}=\lambda F+(1-\lambda)G$.

\begin{remark}
We remark that under appropriate coercivity and growth assumptions on the energy density $\varphi$ the problem of minimizing $E_{\theta}$ over a set of 
admissible maps $\mathcal{A} = \bar{y} + W^{1,p}_0(\Omega,\mathbb{R}^3)$, $\bar{y}\in W^{1,p}(\Omega,\mathbb{R}^3)$ relaxes to the problem of minimizing $I_{\theta}$ over the class of Young measures 
$\nu=(\nu_x)_{x\in\Omega}$ generated by gradients of functions in $W^{1,p}(\Omega,\mathbb{R}^3)$ under the compatibility condition that 
$\bar{\nu}_x\in\mathcal{A}$; that is, any Young measure minimizer of $I_{\theta}$ is generated by a sequence of gradients infimizing $E_{\theta}$ and 
vice versa - see e.g.~\cite{mullernotes,pedregalbook} for relaxation results involving Young measures.
\end{remark}

\section{The Ciarlet-Ne\v{c}as constraint}
\label{sec:1a}

Consider the problem of minimizing the total free energy in \eqref{eq:functionalin1.3} over the set of admissible maps 
$\mathcal A=\left\{y\in\,W^{1,p}(\Omega,\mathbb{R}^{3}):\,y_{\vert\partial\Omega_{1}}=\bar{y},\,\det Dy>0\,\mbox{a.e.}\right\}$, where as before $\Omega$ is a Lipschitz domain, $\partial\Omega_1$ is a relatively open subset of $\partial\Omega$ 
and $\bar{y}\in\mathcal A$. For $y\in \mathcal A$ with $p>3$ Ciarlet and Ne\v{c}as~\cite{ciarletnecas} suggested that a mathematical model of frictionless self-contact without interpenetration could be obtained by requiring 
that admissible deformations $y$ also satisfy the constraint
\begin{equation}
\int_{\Omega}\det Dy(x)\;dx\leq\mathcal{L}^3(y(\Omega)).\tag{C-N}\nonumber
\end{equation}
They showed that under this constraint any admissible deformation $y$ is injective a.e.~in $\overline{\Omega}$ in the sense that
\[\mathrm{card}\,y^{-1}(x')=1\quad\mbox{for almost all $x'\in\,y(\overline{\Omega})$}.\]
They also showed that condition (C-N) is closed under weak convergence in $W^{1,p}(\Omega,\mathbb{R}^3)$ for $p>3$ (weak$\ast$ if $p=\infty$) so that minimizers of the 
associated problem remain a.e.~injective. 

\begin{remark}
The interpenetration of matter and the invertibility of Sobolev (and BV) mappings has been considered by many authors, 
e.g.~Ball~\cite{ballinvertibility}, \v{S}ver\'{a}k~\cite{sverak1988regularity}, Ciarlet \& Destuynder~\cite{ciarletdestuynder}, Cs\"ornyei 
\textit{et al.}~\cite{malyhomeomorphisms}, Hencl \textit{et al.}~\cite{spaceregularity,bvregularity}, Henao \& Mora-Corral~\cite{carlosduvanlusin,carlosduvanbv} as well as 
the book of Fonseca \& Gangbo~\cite{fonsecagangbo}.
\end{remark}

Our model is expressed in terms of gradient Young measures and we employ condition (C-N) in the obvious way, i.e.~we require that the underlying 
deformation of any admissible gradient Young measure satisfies this constraint. In our context, the (C-N) constraint results in deformations (underlying 
admissible measures) which are homeomorphic in $\Omega$ rather than simply a.e.~injective. This is because our admissible deformations turn out to be mappings of bounded distortion, as defined below. 

\begin{definition}
Let $\Omega\subset\mathbb{R}^{d}$ be open. A continuous map $y:\Omega\rightarrow\mathbb{R}^{d}$ is called a \textit{mapping of bounded distortion} if:
\begin{itemize}
\item[(i)] $y\in W^{1,d}_{\mathrm{loc}}(\Omega;\mathbb{R}^{d})$,
\item[(ii)] $\det Dy(x)\geq 0$ for a.e. $x\in \Omega$, or $\det Dy(x)\leq 0$ for a.e. $x\in \Omega$, and
\item[(iii)] there exists a number $M\geq1$ such that
\[\| Dy(x)\|^{d}\leq M\vert\det Dy(x)\vert\]
for almost all $x\in\Omega$.
\end{itemize}
Here, $\| Dy(x)\|=\sup_{\vert z\vert\leq1}\vert Dy(x)z\vert=\sigma_{\max}(Dy(x))$ where $\sigma_{\max}(A)$ denotes the maximum singular value of a matrix $A\in\mathbb{R}^{3\times 3}$.
\end{definition}

Mappings of bounded distortion enjoy remarkable properties. In particular, the following result of Reshetnyak~\cite{reshetnyak} will be crucial for our analysis.

\begin{lemma}
\label{lemmareshetnyak}
Every mapping $y$ of bounded distortion from an open domain $\Omega$ to the space $\mathbb{R}^{d}$ which is not identically constant is an open mapping.
\end{lemma}

We end this section by proving an auxiliary result which will be used to establish the regularity of our deformations under the (C-N) constraint and that of being 
mappings of bounded distortion.

\begin{lemma}
\label{lemma:regularity}
Let $\Omega\subset\mathbb{R}^3$ be a bounded Lipschitz domain. Suppose that $y\in W^{1,\infty}(\Omega,\mathbb{R}^3)$ is a mapping of bounded distortion, satisfying \rm{(C-N)} 
and $\det Dy(x)\geq r >0$ a.e.~in $\Omega$. Then, $y$ is a homeomorphism between $\Omega$ and $y(\Omega)$, its inverse $y^{-1}$ belongs to 
$W^{1,\infty}(y(\Omega),\mathbb{R}^3)$ and
\begin{equation*}
Dy^{-1}(x') = \left[Dy(y^{-1}(x'))\right]^{-1}
\end{equation*}
for a.e.~$x'\in y(\Omega)$.
\end{lemma}

\begin{proof}
We first show how the (C-N) constraint results in $y$ being a.e.~injective. For $\Omega$ a bounded and open subset of $\mathbb{R}^{3}$ and 
$y\in\,W^{1,p}(\Omega,\mathbb{R}^{3})$, $p>3$, a theorem of Marcus and Mizel~\cite{marcusmizel} shows that
\begin{equation}
\int_{\Omega}\vert\det\,Dy(x)\vert\,dx=\int_{y(\Omega)}\mathrm{card}\,y^{-1}(x')\,dx'
\label{eq:meaningful}
\end{equation}
whenever one of the two integrals is meaningful. In our case, $y\in W^{1,\infty}(\Omega,\mathbb{R}^3)$ and the integral on the left-hand side of (\ref{eq:meaningful}) is 
bounded. Then, using (\ref{eq:meaningful}) and the (C-N) constraint we infer that
\begin{equation}
\mathcal{L}^3(y(\Omega)) =\int_{y(\Omega)}\,dx'\leq\int_{y(\Omega)}\mathrm{card}\,y^{-1}(x')\,dx'=
\int_{\Omega}\det\,Dy(x)\,dx\leq\mathcal{L}^3( y(\Omega))\nonumber
\end{equation}
from which we obtain the required a.e.~injectivity, i.e.
\[\mathrm{card}\,y^{-1}(x') =1\quad\mbox{for a.e.~$x'\in\,y(\Omega)$}.\]

Injectivity everywhere in $\Omega$ now follows since $y$ has bounded distortion. Indeed, since constant maps cannot be a.e. injective, 
Lemma~\ref{lemmareshetnyak} implies that $y$ is an open mapping. Suppose for contradiction that there exist distinct $x_{1}, x_{2}\in\Omega$ 
such that $y(x_{1})=y(x_{2})=y_{0}\in\mathbb{R}^{3}$. Note that $y_{0}\in y(\Omega)$ which is open since $\Omega$ is open and $y$ 
maps open sets to open sets; hence, there exists $\epsilon>0$ such that $B(y_0,\epsilon)\subset\,y(\Omega)$.

By the continuity of $y$ the inverse image $y^{-1}(B(y_0,\epsilon))$ of $B(y_0,\epsilon)$ is an open set. However, 
$x_{1}$, $x_{2}\in y^{-1}(B(y_0,\epsilon))$ and we can thus find open neighbourhoods $U$, $V$ of $x_{1}$, $x_{2}$ respectively such that
\[U,V\subset y^{-1}(B(y_0,\epsilon)),\;x_{1}\in U,\;x_{2}\in V\;\;\mbox{and}\;\;U\cap V=\emptyset.\]
As $y$ is an open mapping, $y(U)$ and $y(V)$ are open sets. Furthermore, $y(U)\cap y(V)$ is open and non-empty, since $y_{0}\in y(U)\cap y(V)$. Thus, 
$\mathcal{L}^3(y(U)\cap y(V))>0$ and it must be the case that
\[\mbox{card}\;y^{-1}(x')\geq 2\;\;\mbox{for all}\;\;x'\in y(U)\cap y(V),\]
contradicting a.e.~injectivity. Then, $y$ being injective and open, it is a homeomorphism from $\Omega$ to $y(\Omega)$.

As for the regularity of $y^{-1}$, since $\det Dy(x)\geq r >0$ for a.e.~$x\in\Omega$ it follows that $Dy(y^{-1}(\cdot))^{-1}\in L^{\infty}(y(\Omega),\mathbb{R}^{3\times 3})$. 
To conclude the proof, we show that $Dy^{-1}(x')=Dy(y^{-1}(x'))^{-1}$ in the 
sense of distributions; then $y^{-1}$ being itself bounded in $y(\Omega)$, it is an element of $W^{1,\infty}(y(\Omega),\mathbb{R}^{3})$. 
Setting $w=y^{-1}$ and for any $\psi\in C_0^\infty(\Omega)$ we have that
\begin{eqnarray*}
\int_{y(\Omega)} w_i(x')\frac{\partial\psi}{\partial x_j'}(x')\,dx'&=&\int_{\Omega}x_i\frac{\partial\psi}{\partial x_j'}(y(x))\det Dy(x)\,dx\\
&=&\int_{\Omega}x_i\frac{\partial\psi}{\partial x_k}(y(x))Dy(x)^{-1}_{kj}\det Dy(x)\,dx,
\end{eqnarray*}
by the chain rule, the fact that $y\in W^{1,\infty}$ is differentiable a.e. and $Dy(\cdot)^{-1}\in L^{\infty}(\Omega,\mathbb{R}^{3\times3})$. Using Piola's identity and 
integrating by parts,
\begin{eqnarray*}
\int_{y(\Omega)} w_i(x')\frac{\partial\psi}{\partial x_j'}(x')dx'&=&-\int_{\Omega}\delta_{ik}\psi(y(x))(\mathrm{cof}\,Dy(x))_{jk}\,dx\\
&=&-\int_{\Omega}\frac{(\mathrm{cof}\,Dy(x))_{ji}}{\det Dy(x)}\psi(y(x))\det Dy(x)\,dx\\
&=&-\int_{y(\Omega)}\frac{(\mathrm{cof}\,Dy(w(x')))_{ji}}{\det Dy(w(x'))}\psi(x')\,dx',
\end{eqnarray*}
completing the proof.\qed
\end{proof}

\begin{remark}
We remark that since $y\in\,W^{1,\infty}(\Omega,\mathbb{R}^{3})$ is a homeomorphism with bounded distortion, a theorem from \cite{malyhomeomorphisms}
\footnote{In fact, in dimension 3, it only requires that $y\in\,W^{1,2}_{\mathrm{loc}}(\Omega,\mathbb{R}^{3})$ and is a homeomorphism of finite distortion; 
a notion weaker than that of bounded distortion.} asserts that its inverse already lies in the space $W^{1,1}_{\mathrm{loc}}(y(\Omega),\mathbb{R}^{3})$. However, 
for general homeomorphisms $y\in W^{1,\infty}(\Omega,\mathbb{R}^3)$ one cannot expect that $y^{-1}\in\,W^{1,1}_{\mathrm{loc}}(y(\Omega),\mathbb{R}^{3})$. As a typical example (see e.g.~\cite{malyhomeomorphisms}), 
consider the map $f(x)=x+u(x)$, $x\in\mathbb{R}$, where $u$ is the usual Cantor ternary function and let $g\equiv f^{-1}$. Then $g$ is 1-Lipschitz, and hence in $W^{1,\infty}(0,1)$, but $g^{-1}$ fails to be absolutely 
continuous. By setting $h(x)=(g(x_{1}),x_{2},x_{3})$, one obtains a Lipschitz homeomorphism whose inverse is not in $W^{1,1}_{\mathrm{loc}}(y(\Omega),\mathbb{R}^{3})$.
In~\cite{malyhomeomorphisms} the additional condition imposed on $y$ is that it has finite distortion, in which case $y^{-1}$ is also of bounded distortion. 
Without this assumption, $y^{-1}$ can only be expected to be in $BV_{loc}(y(\Omega),\mathbb{R}^{3})$. For applications 
of the theory of mappings of finite distortion and quasiconformal maps to the regularity of inverses of Sobolev (and BV) maps, the reader is referred 
to~\cite{planarregularity,spaceregularity,bvregularity,onninenregularity} and references therein.
\end{remark}

\section{The simplified model}
\label{sec:2}

To set up our model, let $\Omega\subset\mathbb{R}^{3}$ be a bounded convex polyhedral domain describing the undistorted austenite at the critical temperature 
$\theta_c$, that is $\Omega$ is a bounded domain that is the intersection of a finite number of open half-spaces. (The assumption on the form of $\Omega$ is made 
for ease of exposition, and it is not difficult to extend many of the results to a much wider class of domains having curved faces and edges.) Given any $p\in\mathbb{R}^3$, the set 
$E_p=\{ z\in\bar\Omega:p\cdot(z-x)\geq 0 \mbox{ for all }x\in\bar\Omega\}$ is either a closed polygon (whose relative interior we call a {\it face}), a closed line segment 
(whose relative interior we call an {\it edge}), or a point (which we call a {\it corner}). 

Let \[K:=SO(3)\cup\bigcup^{N}_{i=1}SO(3)U_{i}\] denote the union of the energy wells of the austenite and the 
martensite; for the cubic-to-orthorhombic transition of CuAlNi, $N=6$ and the martensitic variants $U_{i}$, $i=1,\ldots ,6$ are given by
\begin{equation}
\label{eq:orthorhombicvariants}
\begin{array}{cc}
U_{1}=\left(\begin{array}{ccc}
\beta & 0 & 0\\0 & \frac{\alpha+\gamma}{2} & \frac{\alpha-\gamma}{2}\\0 & \frac{\alpha-\gamma}{2} & \frac{\alpha+\gamma}{2}\end{array}\right) & 
U_{2}=\left(\begin{array}{ccc}
\beta & 0 & 0\\0 & \frac{\alpha+\gamma}{2} & \frac{\gamma-\alpha}{2}\\0 & \frac{\gamma-\alpha}{2} & \frac{\alpha+\gamma}{2}\end{array}\right)\\
\, & \, \\
U_{3}=\left(\begin{array}{ccc}
\frac{\alpha+\gamma}{2} & 0 & \frac{\alpha-\gamma}{2}\\0 & \beta & 0\\\frac{\alpha-\gamma}{2} & 0 & \frac{\alpha+\gamma}{2}\end{array}\right) & 
U_{4}=\left(\begin{array}{ccc}
\frac{\alpha+\gamma}{2} & 0 & \frac{\gamma-\alpha}{2}\\0 & \beta & 0\\\frac{\gamma-\alpha}{2} & 0 & \frac{\alpha+\gamma}{2}\end{array}\right)\\
\, & \, \\
U_{5}=\left(\begin{array}{ccc}
\frac{\alpha+\gamma}{2} & \frac{\alpha-\gamma}{2} & 0\\\frac{\alpha-\gamma}{2} & \frac{\alpha+\gamma}{2} & 0\\0 & 0 & \beta\end{array}\right) &	
U_{6}=\left(\begin{array}{ccc}
\frac{\alpha+\gamma}{2} & \frac{\gamma-\alpha}{2} & 0\\\frac{\gamma-\alpha}{2} & \frac{\alpha+\gamma}{2} & 0\\0 & 0 & \beta\end{array}\right).
\end{array}
\end{equation}

Henceforth, we denote the mechanically stabilized variant of martensite by $U_{s}$ where $s\in\left\{1,\ldots,N\right\}$; then the homogeneous gradient Young measure 
$\delta_{U_{s}}$ corresponds to a pure phase of the variant $U_{s}$. 

Also, we make the standing assumption on the lattice parameters that $\det U_{s}\leq 1$ and $\lambda_{\max}(\mathrm{cof}\,U_s)\geq1$, where 
$\lambda_{\max}(A)$ corresponds to the maximum eigenvalue of the matrix $A$; in particular, the same holds for all martensitic variants as they are symmetry related. 
This is a technical assumption and it is consistent with the lattice parameters of the CuAlNi specimen of the experiment.

As regards the regularity and growth of our energy density $\varphi:\mathbb{R}^{3\times 3}\rightarrow\overline{\mathbb{R}}=\mathbb{R}\cup\left\{+\infty\right\}$ we only assume that $\varphi$ is lower semicontinuous. 
Moreover, as described in the introduction, the temperature of the iron tip exceeded the critical temperature $\theta_c$ significantly and it was 
only after the lapse of sufficient time for the temperature at the corner to reach $\theta_c$ that the transformation initiated; thus, we assume that $\theta>\theta_{c}$, i.e.~austenite is 
energetically preferable to martensite. Specifically, on the energy wells $K$, we assume that
\begin{equation}
\varphi(F)=\left\{\begin{array}{ccc}
-\delta,&\quad &F\in SO(3)\\
0,&\quad &F\in\bigcup^{N}_{i=1}SO(3)U_{i}
\end{array}\right.
\label{eq:varphidefinition}
\end{equation}
where $\delta>0$ is such that $\min_{\mathbb{R}^{3\times 3}}\varphi=-\delta$; in particular, $\varphi$ is bounded 
below.

As noted above, we do not impose any coercivity or growth conditions for our minimization problem. However, the question of local minimizers 
is dependent on the growth of $\varphi$ off the energy wells and in order to make the problem more tractable, we wish to work with an energy functional 
which captures the essential behaviour of $\varphi$ but becomes infinite off the set $K$. This allows one to disregard the growth of $\varphi$ and 
instead concentrate only on microstructures supported on the wells; moreover, any configuration that lowers the energy will necessarily be partly 
supported on the austenitic well, $SO(3)$, so that austenite has nucleated. To derive this functional we invoke the topology of $\Gamma$-convergence. This type 
of convergence is only used to deduce our model rigorously and we do not make use of other results of $\Gamma$-convergence theory. As such we do not review $\Gamma$-convergence here and 
refer the reader to Dal Maso~\cite{dalmaso} for an extensive account.

\paragraph{\textbf{Derivation of the energy functional:}}

We note that, due to Zhang's lemma, it is natural to restrict attention to $W^{1,\infty}$ gradient Young measures since, once we consider the blown up version of 
our functional, any gradient Young measure with finite energy will be supported entirely within the compact set $K$. Of course, our admissible measures will come with 
further restrictions which, for the purposes of deriving our functional via 
$\Gamma$-convergence, we ignore. Instead, let us define our functionals on the broader class of measures
\begin{equation}
\label{eq:YMgamma}
\mathcal{A}:=\left\{\nu\in\mathcal{G}^{\infty}(\Omega,\mathbb{R}^{3\times3}) : \bar{\nu}_x=Dy(x)\mbox{ a.e., $\det Dy(x)>0$ a.e.~and $y$ satisfies (C-N)}\right\}.
\end{equation}

\begin{remark}
We remark that we are interested in the $\Gamma$-limit with respect to weak$\ast$ convergence in the space 
$L^{\infty}_{w^{\ast}}(\Omega,\mathcal{M}(\mathbb{R}^{3\times 3}))$. In principle, $\Gamma$-convergence can be defined for a general topological space, with the 
convergence being naturally induced by the topology; nevertheless, one needs to check that the sequential characterization of $\Gamma$-convergence is valid in this context.

However, the space $L^{\infty}_{w^{\ast}}(\Omega,\mathcal{M}(\mathbb{R}^{3\times 3}))$ is the dual of $L^{1}(\Omega,C_{0}(\mathbb{R}^{3\times 3}))$, a separable 
normed vector space, and hence the closed unit ball in $L^{\infty}_{w^{\ast}}(\Omega,\mathcal{M}(\mathbb{R}^{3\times 3}))$ is weak$\ast$ 
compact and when equipped with the topology induced by the weak$\ast$ convergence it becomes a metrizable space with metric given by
\begin{equation}
\mathrm{d}\left(\nu,\,\mu\right)=\displaystyle\sum^{\infty}_{i,j=1}2^{-i-j}\vert (\nu-\mu, h_{j}\otimes f_{i})\vert =
\displaystyle\sum^{\infty}_{i,j=1}2^{-i-j}\vert\int_{\Omega}h_j(x)\langle\nu_{x}-\mu_{x},f_{i}\rangle\,dx\vert ,
\label{eq:ymdistance}
\end{equation}
where $\left\{f_{i}\right\}$ and $\left\{h_{j}\right\}$ are countable dense subsets in the unit balls of $C_{0}(\mathbb{R}^{3\times 3})$ and 
$L^{1}(\Omega)$ respectively. Henceforth, we will freely use the term $\Gamma$-convergence with respect to the weak$\ast$ topology in the space of Young 
measures without referring to the above argument about its metrization.
\end{remark}

Let $\psi:\mathbb{R}^{3\times 3}\rightarrow\mathbb{R}$ be such that
\begin{equation}
\label{eq:psigrowth}
\begin{array}{ll}
\psi(F)\geq 0,& \psi(F)\geq -d+c|F|^p, \mbox{ for all } F,\\
\psi(F)=0& \mbox{ if and only if } F\in K
\end{array}
\end{equation}
for some $d$, $c>0$ and $p>3$. For each $k\in\mathbb{N}$, let $W^{k}:=k\psi+\varphi$ and define the functional 
$I^{k}:\mathcal{A}\rightarrow\overline{\mathbb{R}}$ by
\begin{equation}
I^{k}(\nu):=\int_{\Omega}\langle\nu_{x},W^{k}\rangle\;dx.
\label{eq:I^{k}}
\end{equation}
The function $\psi$ vanishes on the wells, whereas it is positive and growing away from them. Hence, in the limit $k\rightarrow\infty$, 
one expects that the functional will behave like $\varphi$ for measures supported a.e.~in $K$, while it should blow up whenever part of the support lies 
outside $K$.

\begin{proposition}
Let $\varphi$, $\psi:\mathbb{R}^{3\times 3}\rightarrow\overline{\mathbb{R}}$ be lower semicontinuous functions satisfying (\ref{eq:varphidefinition}) and 
(\ref{eq:psigrowth}) respectively and for $k\in\mathbb{N}$ write $W^{k}=k\psi +\varphi$. Let 
$I^{k}:\mathcal{A}\rightarrow\mathbb{R}$ be as in (\ref{eq:I^{k}}). Then, 
$I=\Gamma\mbox{-}\lim_{k\rightarrow\infty}I^{k}$ with respect to the weak$\ast$ convergence in $L^{\infty}_{w^{\ast}}(\Omega, \mathcal{M}(\mathbb{R}^{N\times d}))$, where
\begin{equation}
I(\nu)=\left\{\begin{array}{cl}
-\delta\int_{\Omega}\nu_{x}(SO(3))\;dx,&\:\:\mathrm{supp}\;\nu_x\subset K\,\,\mbox{a.e.}\\
+\infty,&\:\:\mathrm{otherwise}
\end{array}\right.
\end{equation}
and $\nu_{x}(SO(3))=\int_{SO(3)}\;d\nu_{x}(A)$.
\label{lemmagammaconvergence}
\end{proposition}

\begin{remark}
In terms of microstructures, the expression $\nu_{x}(SO(3))$ represents the volume fraction of austenite at $x\in\Omega$.
\end{remark}

\begin{proof}
Note that since $\psi\geq 0$, the sequence $\varphi^{k}$ is increasing and therefore so is $I^{k}$, i.e.~the $\Gamma$-limit 
exists and is given by the supremum of the lower semicontinuous envelopes of the functionals $I^k$, see e.g.~\cite{dalmaso}. To establish the 
$\liminf$-inequality, let $\nu^k\in\mathcal{A}$ be such that
\[\nu^{k}\overset{\ast}{\rightharpoonup}\nu\quad\mathrm{in}\quad L^{\infty}_{w^{\ast}}\left(\Omega,\,\mathcal{M}\left(\mathbb{R}^{3\times 3}\right)\right).\]
We first show that $\nu\in\mathcal{A}$. Clearly, we may assume that $\liminf_{k}I^k\left(\nu^{k}\right)<\infty$. In particular, this implies that 
$\mathrm{supp}\,\nu_{x}\subset\,K$ a.e.~in $\Omega$; otherwise, $\nu_x$ has part of its support on a set of positive measure outside $K$. Then, letting $f\in C_{0}(\mathbb{R}^{3\times 3})$ such that 
$0\leq f\leq\psi$ and $f\left(A\right)=0$ if and only if $A\in K$, we deduce by the weak$\ast$ convergence of $\nu^{k}$ that
\[\int_{\Omega}\langle\nu^{k}_{x},f\rangle\,dx\rightarrow\int_{\Omega}\langle\nu_{x},f\rangle\,dx=C>0.\]
Hence, for $k$ large enough,
\begin{eqnarray}
I^{k}(\nu^{k})&=&k\int_{\Omega}\langle\nu^{k}_{x},\psi\rangle\;dx+\int_{\Omega}\langle\nu^{k}_{x},\varphi\rangle\;dx\nonumber\\
&\geq &k\int_{\Omega}\langle\nu^{k}_{x},f\rangle\,dx-\delta\vert\Omega\vert\nonumber\\
&\geq &k\frac{C}{2}-\delta\vert\Omega\vert,\nonumber
\end{eqnarray}
contradicting $\liminf_{k}I^k\left(\nu^{k}\right)<\infty$.

Also, since $\liminf_{k}I^k\left(\nu^{k}\right)<\infty$, up to a subsequence we may also assume that $\sup_k I^k(\nu^k)<\infty$ and, by (\ref{eq:psigrowth}), we 
deduce that for some constant $C$,
\begin{equation}
\label{eq:uniformpbound}
\sup_{k}\int_{\Omega}\int_{\mathbb{R}^{3\times3}}\vert A\vert^p\,d\nu^{k}_{x}(A)dx\leq C.
\end{equation}
In particular, this implies that $\nu_x$ is a probability measure for a.e.~$x$ (e.g.~\cite[Theorem 3.6]{sychev}) and that (e.g.~\cite[Theorem 3.7]{sychev})
\begin{equation}
\label{eq:convergencecentreofmass}
\bar{\nu}^{k}\rightharpoonup\bar{\nu}\;\;\mbox{in $L^{p}(\Omega,\mathbb{R}^{3\times3})$}.
\end{equation}
Also, since $F\mapsto\vert F\vert^p$ is quasiconvex, non-negative and satisfies a $p$-growth condition, Theorem~\ref{KPGYM} gives that
\begin{equation}
\label{eq:uniformpbound2}
\sup_{k}\int_{\Omega}\vert\bar{\nu}^{k}_{x}\vert^p\,dx\leq C,
\end{equation}
i.e.~the mappings $y^k$ underlying the measures $\nu^k$ are uniformly bounded in $W^{1,p}(\Omega,\mathbb{R}^3)$ and in view of (\ref{eq:convergencecentreofmass}), 
$y^k\rightharpoonup y$ in $W^{1,p}(\Omega,\mathbb{R}^3)$ where $Dy(x)=\bar{\nu}_{x}$ a.e.~in $\Omega$; note, that {\it a priori} $y^k\in W^{1,\infty}(\Omega,\mathbb{R}^3)$ for 
each $k$ but the bound is not uniform.

By \cite[Proposition 4.6]{sychev}, we infer that $\nu$ is a $W^{1,p}$ gradient Young measure and, since $\mathrm{supp}\,\nu_x\subset K$ a.e., Zhang's lemma says that 
$\nu\in\mathcal{G}^{\infty}(\Omega,\mathbb{R}^{3\times3})$. The remaining conditions now follow: the determinant constraint follows from the minors relations 
(\ref{eq:minors0}) and the fact that $\mathrm{supp}\,\nu_x\subset K$ a.e.~in $\Omega$, since
\begin{equation}
\label{eq:detweakcts}
0<\min (\det U_s,1)=\min_{F\in K}\det F\leq\langle\nu_{x},\det\rangle=\det \bar{\nu}_{x}=\det Dy(x)\leq\max_{F\in K}\det F=\max (\det U_s,1)
\end{equation}
(where the last inequality is not needed her but will be used below), (C-N) follows by the fact that each $y^k$ satisfies (C-N) and $y^k\rightharpoonup y$ in $W^{1,p}(\Omega,\mathbb{R}^3)$, $p>3$ (see~\cite{ciarletnecas}).

Having established that $\nu\in\mathcal{A}$, let us prove the $\liminf$-inequality, that is
\[
I(\nu) \leq \liminf_{k\to\infty} I^k(\nu^k).
\]
Note that any non-negative lower semicontinuous function $f$ on 
$\mathbb{R}^{3\times3}$ (generally, on a metric space) is the pointwise supremum of a non-decreasing sequence of functions in $C_0\left(\mathbb{R}^{3\times 3}\right)$, 
e.g.~given $f$ not identically $+\infty$, its Lipschitz regularizations
$${\displaystyle g_n(F):=\inf_{A\in\mathbb{R}^{3\times3}}(f(A)+n\vert F-A\vert)}$$
satisfy $g_1\leq g_2\leq\ldots\leq f$ and $g_n(F)\rightarrow f(F)$ for all $F\in\mathbb{R}^{3\times3}$. Multiply each $g_n$ by a suitable decaying and continuous function $h_n$, e.g.
\[
h_n(F)=\left\{\begin{array}{lc}
1,& |F|\leq n,\\
1+n-|F|,& n\leq |F| \leq n+1\\
0,& |F|\geq n+1.
\end{array}\right.
\]
Then, $f_n=g_nh_n$ is the desired sequence in $C_0(\mathbb{R}^{3\times 3})$. Next, consider the map $\varphi +\delta$; this is 
non-negative, lower semicontinuous and hence there exists a non-decreasing sequence of functions, say $\varphi_{n}\in C_{0}\left(\mathbb{R}^{3\times 3}\right)$, with 
$0\leq\varphi_{n}\leq\varphi +\delta$ converging pointwise to $\varphi +\delta$. Then, by weak$\ast$ convergence and the fact that $\varphi_{n}\leq\varphi +\delta$, for 
$n\in\mathbb{N}$
\begin{eqnarray}
\int_{\Omega}\langle\nu_{x},\varphi_{n}\rangle\,dx&=&\displaystyle\liminf_{k\rightarrow\infty}\int_{\Omega}\langle\nu^{k}_{x},\varphi_{n}\rangle\,dx\nonumber\\
&\leq &\displaystyle\liminf_{k\rightarrow\infty}\int_{\Omega}\langle\nu^{k}_{x},\varphi +\delta\rangle\,dx\nonumber\\
&\leq &\displaystyle\liminf_{k\rightarrow\infty}\int_{\Omega}\langle\nu^{k}_{x},k\psi +\varphi +\delta\rangle\,dx =\displaystyle\liminf_{k\rightarrow\infty}I^{k}(\nu^{k})+\delta\vert\Omega\vert\nonumber
\end{eqnarray}
since $\psi\geq 0$ and $\nu^{k}_{x}$ is a probability measure for a.e.~$x$. But $\varphi_n\geq 0$ for all $n$, and letting $n\to\infty$ we obtain the 
$\liminf$ inequality by monotone convergence and the fact that $\nu_{x}$ is a probability measure a.e.~in $\Omega$.

For the recovery sequence, let $\nu\in\mathcal{G}^{\infty}(\Omega,\mathbb{R}^{3\times3})$. The recovery sequence is simply given by the constant sequence $\nu^{k}=\nu$, as then trivially 
$\nu^k\in\mathcal{G}^{\infty}(\Omega,\mathbb{R}^{3\times3})$, $\nu^{k}\overset{\ast}{\rightharpoonup}\nu$ and
\begin{equation}
I^{k}(\nu^{k})=I^{k}(\nu)=k\int_{\Omega}\langle\nu_{x},\psi\rangle\,dx+\int_{\Omega}\langle\nu_{x},\varphi\rangle\,dx,\nonumber
\end{equation}
so that if $\mbox{supp}\,\nu_x\subset K$ a.e. then $I^k(v)=I(v)<\infty$, while otherwise $I^k(v)=I(v)=\infty$.\qed
\end{proof}

We may alternatively write the energy functional $I(\nu)$ as
\begin{equation}
I(\nu)=\int_{\Omega}\langle\nu_{x},W\rangle\,dx,\nonumber
\end{equation}
where the energy density $W:\mathbb{R}^{3\times 3}\rightarrow\overline{\mathbb{R}}$ is given by
\begin{equation}
W(F)=\left\{\begin{array}{rl}
-\delta, &\:\: F\in SO(3)\\
0, &\:\: F\in\bigcup^{N}_{i=1}SO(3)U_{i}\\
+\infty, &\:\:\mathrm{otherwise}.
\end{array}\right.
\label{eq:W}
\end{equation}
For our simplified model, we use $I$ as our functional and we proceed to define the set of admissible measures.

\begin{remark}
Note that due to the singular form of our density, if $I(\nu)<\infty$ then $\mathrm{supp}\,\nu_{x}\subset K$ a.e.~in $\Omega$, implying that the 
underlying gradient $Dy(x)=\nu_x$ satisfies
\begin{equation}
\label{eq:dyinkqc}
Dy(x)\in K^{qc}\quad\mbox{for a.e.~$x\in\Omega$},
\end{equation}
where the quasiconvexification $K^{qc}$ of an arbitrary compact set $K\subset\mathbb{R}^{N\times d}$ is defined by (see \cite{sveraktwowells})
\[
K^{qc}:=\{F\in\mathbb{R}^{N\times d}\,:\,f(F)\leq \max_{K}f,\mbox{ for all $f:\mathbb{R}^{N\times d}\to\mathbb{R}$ quasiconvex}\}.
\]
For some general results regarding quasiconvex hulls of sets, we also refer the reader to \cite{newBJ,mullernotes}.
\end{remark}

\paragraph{\textbf{Admissible measures:}}
Crucially, we assume that the localized heating necessarily leads to a localized nucleation of austenite and, in our minimization problem, we only consider variations of 
$\delta_{U_{s}}$ which are localized in the interior, on faces, edges and at corners. In particular, let $x_0\in\bar\Omega$ and depending on whether $x_0$ is an interior point, 
belongs to a face or edge, or is a corner, let $S(x_0)\subset\Omega$ be as illustrated in Fig.~\ref{fig:admissible}.

\begin{figure}[ht]
	\centering
	\def\svgwidth{0.9\columnwidth}
	\begingroup
    \setlength{\unitlength}{\svgwidth}
  \begin{picture}(1,0.55363752)%
    \put(0,0){\includegraphics[width=\unitlength]{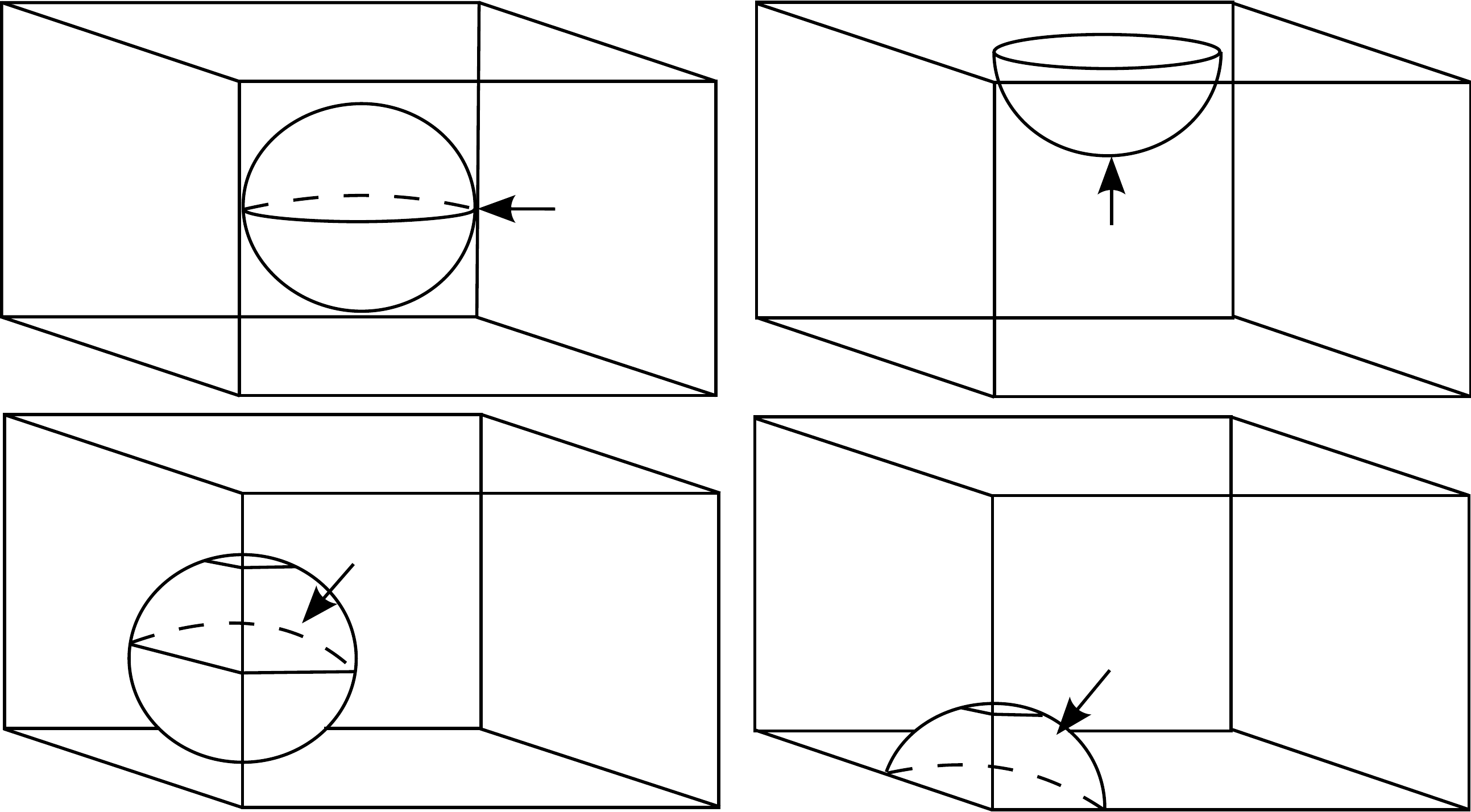}}%
    \put(0.38,0.394){\color[rgb]{0,0,0}\makebox(0,0)[lb]{\smash{$S(x_0)$}}}%
    \put(0.725,0.365){\color[rgb]{0,0,0}\makebox(0,0)[lb]{\smash{$S(x_0)$}}}%
    \put(0.24,0.17){\color[rgb]{0,0,0}\makebox(0,0)[lb]{\smash{$S(x_0)$}}}%
    \put(0.74,0.1){\color[rgb]{0,0,0}\makebox(0,0)[lb]{\smash{$S(x_0)$}}}%
    \put(0.014,0.46){\color[rgb]{0,0,0}\makebox(0,0)[lb]{\smash{interior}}}%
    \put(0.545,0.46){\color[rgb]{0,0,0}\makebox(0,0)[lb]{\smash{face}}}%
    \put(0.017,0.18){\color[rgb]{0,0,0}\makebox(0,0)[lb]{\smash{edge}}}%
    \put(0.54,0.18){\color[rgb]{0,0,0}\makebox(0,0)[lb]{\smash{corner}}}%
  \end{picture}%
\endgroup
	\caption{Subsets $S(x_0)$ of $\Omega$ (taken to be a rectangular parallelepiped) used for testing whether nucleation of austenite can occur in the interior, on a face, an edge and at a corner; 
these are given respectively by the intersection of $\Omega$ with a small ball centred at a point in the interior, on a face, an edge or a corner.}
	\label{fig:admissible}
\end{figure}
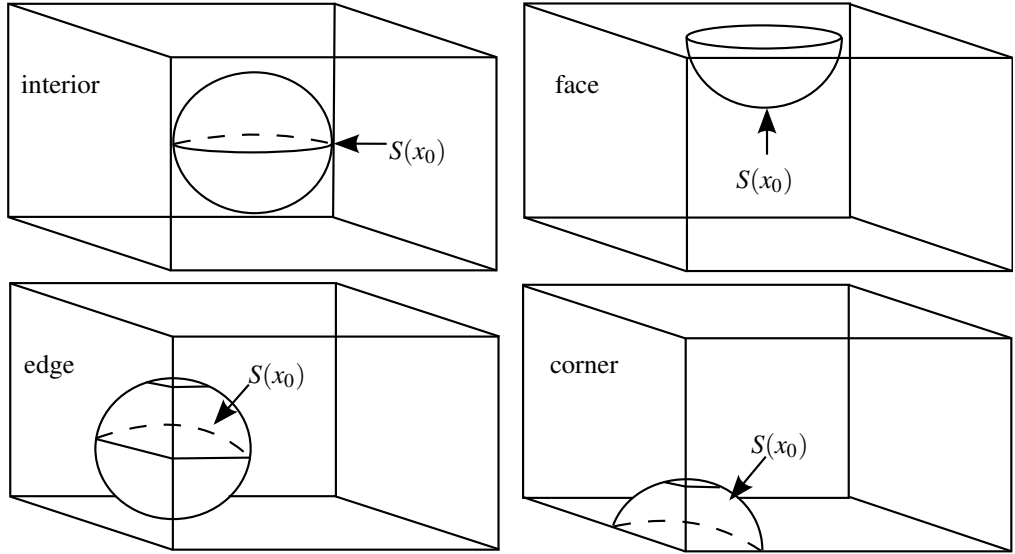

More precisely, let $B(x_0,r)$ be the ball of radius $r>0$ centred at $x_0$.
\begin{itemize}
\item[(i)] (interior) If $x_0\in\Omega$ then $S(x_0)=B(x_0,r)$ for some $r>0$ such that $\overline{S(x_0)}\subset\Omega$.
\item[(ii)] (face) If $x_0$ belongs to a face $F$ contained in the plane $\left\{x\cdot n=k\right\}$, where $n\in S^{2}$ is outward pointing and $k\in\mathbb{R}$, 
then
\[
 S(x_0)=\left\{x\in B(x_0,r):x\cdot n<k\right\}
\]
for some $r>0$ such that $\overline{S(x_0)}\setminus\left\{x\cdot n=k\right\}\subset\Omega$. 
\item[(iii)] (edge) If $x_0$ belongs to an edge $E$ that is the intersection of two faces contained in the planes $\left\{x\cdot n_{1}=k_{1}\right\}$ and 
$\left\{x\cdot n_{2}=k_{2}\right\}$ where $n_{i}\in S^{2}$ are outward pointing and $k_{i}\in\mathbb{R}$, $i=1,2$, then 
\[
 S(x_0)=\left\{x\in B(x_0,r)\,:\,x\cdot n_{i}<k_i,\,\,i=1,2\right\}
\]
for some $r>0$ such that $\overline{S(x_0)}\setminus\bigcup^{2}_{i=1}\left\{x\cdot n_i=k_i\right\}\subset\Omega$.
\item[(iv)] (corner) If $x_0$ is a corner $C$ that is the intersection of $N\geq 3$ faces contained in the planes $\left\{x\cdot n_{i}=k_{i}\right\}$ where 
$n_{i}\in S^{2}$ are outward pointing and $k_{i}\in\mathbb{R}$, $i=1,\ldots ,N$, then 
\[
 S(x_0)=\left\{x\in B(x_0,r)\,:\,x\cdot n_{i}<k_i,\,\,i=1,\ldots, N\right\}
\]
for some $r>0$ such that $\overline{S(x_0)}\setminus\bigcup^{N}_{i=1}\left\{x\cdot n_i=k_i\right\}\subset\Omega$.
\end{itemize}

\begin{remark}
Note that the sets $S(x_0)$ are chosen to be (parts of) balls for simplicity; clearly, no matter what the shape of a `nucleation 
region' is, it can always be embedded in (part of) such a ball.

Moreover, we can write $\partial S(x_0)$ as the disjoint union of 
$\partial S(x_0)\cap\Omega$ and $\partial S(x_0)\cap\partial\Omega$, which are relatively open and closed subsets of $\partial S(x_0)$ respectively.
\end{remark}
 
\begin{definition}
A measure $\nu=(\nu_x)_{x\in\Omega}$ is admissible if there exists $x_0\in\overline{\Omega}$, $r>0$ and $S(x_0)$ as above such that $\nu\in{\mathcal A}(x_0)$, where
\begin{equation}
\mathcal{A}(x_0)=\{\nu\in\mathcal{G}^{\infty}(\Omega,\mathbb{R}^{3\times3}) : \nu_{x}=\delta_{U_s}\mbox{ a.e. }x\notin S(x_0), y\mbox{ satisfies (C-N) }, \det\,Dy>0\mbox{ a.e., } y_{|\partial S(x_0)\cap\Omega}=U_{s}x\}.\nonumber
\end{equation}
and $y\in W^{1,\infty}(\Omega,\mathbb{R}^3)$ is the deformation underlying $\nu$.
\label{def:admissible+c-n}
\end{definition}

For faces, edges and corners $\partial S(x_0)\cap\partial\Omega$ act as free 
boundaries; these are comprised of part of the given face, parts of the two faces that meet at the given edge and parts of the faces that meet at the 
given corner, respectively.

We end this section by showing how condition (C-N) leads to finite-energy deformations which are homeomorphic in $\Omega$.

\begin{theorem}
Let $x_0\in\overline{\Omega}$ and $\nu\in{\mathcal A}(x_0)$. Suppose that $\mathrm{supp}\,\nu_x\subset K$ a.e.~in $\Omega$ and let 
$y\in W^{1,\infty}(\Omega,\mathbb{R}^3)$ be the underlying deformation of $\nu$. Then $y$ is a homeomorphism between $\Omega$ and $y(\Omega)$, its 
inverse $y^{-1}$ belongs to $W^{1,\infty}(y(\Omega),\mathbb{R}^3)$ and
$$Dy^{-1}(x')=[Dy(y^{-1}(x'))]^{-1}$$
for a.e.~$x'\in y(\Omega)$.
\label{theorem:homeomorphic}
\end{theorem}

\begin{proof}
In view of Lemma~\ref{lemma:regularity}, we need only prove that $y$ is a mapping of bounded distortion and that $\det Dy(x)\geq r > 0$ a.e.~in 
$\Omega$. The determinant constraint follows by the minors relations and the fact that $\mathrm{supp}\,\nu_x\subset K$ a.e.~in $\Omega$ as in (\ref{eq:detweakcts}). 
To show that $y$ has bounded distortion, note that $\Omega$ being Lipschitz, $y\in W^{1,\infty}(\Omega,\mathbb{R}^{3})$ is continuous. Moreover, 
$\mathrm{supp}\,\nu_{x}\subset K$ a.e.~and the norm on $\mathbb{R}^{3\times3}$, $\|F\|=\sigma_{\max}(F)$ is quasiconvex, i.e.
\[\|\bar{\nu}_x\|=\|Dy(x)\|\leq\max_{F\in K}\|F\|\:\:\mbox{for a.e.}\;x\in\Omega.\]
Also, by (\ref{eq:detweakcts}), $\det\,Dy(x)$ is bounded above and below a.e.~and does not change sign in $\Omega$. Then, for a.e.~$x\in\Omega$,
\[\frac{\| Dy(x)\|^{3}}{\vert\det Dy(x)\vert}\leq\frac{\max_{F\in K}\|F\|^{3}}{\min_{F\in K} \det F}\]
which is clearly bounded, so that $y$ is a mapping of bounded distortion.\qed
\end{proof}

\section{Proposed explanation of the nucleation points}
\label{sec:3}

Our main result implies that nucleation is only possible at a corner. This is however dependent on the directions along which the specimen is cut and the initial variant $U_s$. We define below a corresponding class of admissible domains for which this result will be proved to hold. We first note that for any admissible measure $\nu$ with finite energy, $\mathrm{supp}\,\nu_{x}\subset\,K=SO(3)\cup\bigcup^{N}_{i=1}SO(3)U_{i}$ a.e., which implies that 
$\bar{\nu}_{x}\in\,K^{qc}$ a.e.~in $\Omega$. In particular, for any quasiconvex function $f:\mathbb{R}^{3\times 3}\rightarrow\mathbb{R}$
\[f(\bar{\nu}_{x})\leq\displaystyle\max_{F\in\,K}f(F)\:\:\:\mbox{for a.e.~$x\in\Omega$.}\]
This motivates the following definition.

\begin{definition}
Let $K=SO\left(3\right)\cup\bigcup^{N}_{i=1}SO\left(3\right)U_{i}$ and $s\in\left\{1,\dots ,N\right\}$. We say that a vector $e\in S^{2}$ - the unit sphere in $\mathbb{R}^{3}$ - is a maximal direction for $U_{s}$ if
\begin{equation}
\vert U_{s}e\vert =\displaystyle\max_{F\in\,K}\vert Fe\vert=\displaystyle\max_{i\in\left\{1,\dots ,N\right\}}\left\lbrace\vert U_{i}e\vert ,1\right\rbrace,\nonumber
\end{equation}
where we use the shorthand $\max_{i\in\{1,\ldots,N\}}\{|U_ie|,1\}:=\max \{\max_{i\in\{1,\ldots,N\}}\{|U_ie|\},1\}$.
Similarly, we say that a vector $e\in\,S^{2}$ is a maximal direction for $U^{-1}_{s}$ if either
\begin{eqnarray}
\label{eq:maximalitycofactors}
\vert(\mathrm{cof}\,U_{s})e\vert&>&\displaystyle\max_{F\in\,K\setminus{SO(3)U_s}}\vert(\mathrm{cof}\,F)e\vert=\displaystyle\max_{i\in\left\{1,\dots ,N\right\}\setminus\left\{s\right\}}\lbrace\vert(\mathrm{cof}\,U_{i})e\vert ,1\rbrace\\
\mbox{or}&& e=e_{\max}(\mathrm{cof}\,U_s),\nonumber
\end{eqnarray}
where $e_{\max}(\mathrm{cof}\,U_s)$ denotes the eigenvector of $\mathrm{cof}\,U_s$ corresponding to its largest eigenvalue\footnote{Note that, since we  assume that $\lambda_{\rm max}({\rm cof}\,U_s)\geq 1$, 
$e_{\max}(\mathrm{cof}\,U_s)$ satisfies \eqref{eq:maximalitycofactors} with nonstrict inequality but may or may not do so with strict inequality.}.
We denote the set of maximal directions for $U_{s}$ and $U^{-1}_{s}$ by $\mathcal{M}_{s}$ and $\mathcal{M}^{-1}_{s}$ respectively.
\label{definitionmaximal}
\end{definition}

\begin{definition}
Let $\Omega\subset\mathbb{R}^{3}$ be a convex polyhedral domain. We say that an edge of $\Omega$ is admissible for $U_{s}$ if it is in the direction of a 
vector in $\mathcal{M}_{s}\cup\,U^{-2}_{s}\mathcal{M}^{-1}_{s}$. Similarly, a face of $\Omega$ is admissible for $U_{s}$ if the normal to the face is 
perpendicular to a vector in $\mathcal{M}_{s}\cup\,U^{-2}_{s}\mathcal{M}^{-1}_{s}$. The domain $\Omega$ is admissible for $U_{s}$ if all 
of its edges are admissible.
\end{definition}

\begin{remark}
We view the set $U^{-2}_{s}\mathcal{M}^{-1}_{s}$ as a subset of $S^2$; that is $e\in U^{-2}_{s}\mathcal{M}^{-1}_{s}$ if there 
exists $f\in\mathcal{M}^{-1}_{s}$ such that $e=U^{-2}_{s}f/\vert U^{-2}_{s}f\vert$. Moreover, note that if an edge is admissible it follows that the faces 
intersecting at that edge are also admissible since the normals are necessarily perpendicular to that edge. Therefore, for $\Omega$ to be admissible we 
need not require that its faces are admissible too.
\end{remark}

We are now in a position to state and prove our main result:

\begin{theorem}
Let $\Omega\subset\mathbb{R}^{3}$ be a convex polyhedral domain that is admissible for $U_{s}$, $s\in\{1,\ldots,N\}$, and assume that $\det\,U_{s}\leq 1$ as well as $\lambda_{\max}(\mathrm{cof}\,U_s)\geq1$. 
Let $x_0\in\overline{\Omega}$ and $\nu\in{\mathcal A}(x_0)$ be such that $I(\nu)<I(\delta_{U_s})$. 
Then, $x_0$ is a corner.\footnote{We note that in \cite{icomat11} the result was erroneously stated with nonstrict inequalities in \eqref{eq:maximalitycofactors} 
corresponding to the definition of the maximal directions for $U_s^{-1}$ and $U_s^{-1}{\mathcal M}_s^{-1}$ in place of $U_s^{-2}{\mathcal M}_s^{-1}$.}

Furthermore, let $\Omega\subset\mathbb{R}^3$ be a rectangular parallelepiped with edges along the vectors $e_1$, $e_2$, $e_3$ and define a coordinate system such that the axes are parallel to the edges 
and each corner of $\Omega$ belongs to a different octant $O_i$, $i=1,\ldots,8$. Fix $s\in\{1,\ldots,N\}$ and suppose that there exists $l\in\{1,\ldots,N\}$, $l\neq s$, such that
\begin{align*}
QU_l-U_s &=a\otimes n\\
U_s+\lambda a\otimes n &=R+b\otimes m,
\end{align*}
for some $\lambda\in(0,1)$, $R$, $Q\in SO(3)$, $a$, $b\in\mathbb{R}^3$ and $n$, $m\in S^2$ such that $n$, $m$ belong to the same octant, say $O_k$, they are not perpendicular to the vectors $e_i$, $i=1,2,3$ and $U^{-1}_{s}b\cdot n<0$. 
Then, for each of the corners $x_0$ belonging to the octants $O_k$ and $-O_k$, there exists an admissible measure $\nu\in\mathcal{A}(x_0)$ such that $I(\nu)<I(\delta_{U_s})$.

In particular, for the {\rm CuAlNi} specimen of the experiment undergoing a cubic-to-orthorhombic transformation with $N=6$, $e_i$, $i=1,2,3$, the standard basis of $\mathbb{R}^3$ and lattice parameters $\alpha=1.06372$, $\beta=0.91542$, $\gamma=1.02368$, for each $s\in\{1,\ldots,6\}$, 
there exist precisely four such corners.
\label{theorem:main}
\end{theorem}

\begin{remark}\label{rem:aftermain}
For the simplified model constructed above, if $\delta_{U_{s}}$ is a local minimizer with respect to variations in ${\mathcal A}(x_0)$ for $x_0$ not a corner then 
no nucleation can occur in the neighbourhood of $x_0$. This is clear from the form of our energy as $I(\delta_{U_s})=0$ and
\[
I(\nu)<0\:\mbox{ if and only if }\:\int_{\Omega}\nu_x(SO(3)) dx=\int_{\Omega}\int_{SO(3)}\,\,d\nu_{x}(A)>0.
\]
On the other hand, if there exists $\nu\in\mathcal{A}(x_0)$ such that $I(\nu)<I(\delta_{U_s})$, we infer that $\nu$ must be partly supported on $SO(3)$; in 
particular, austenite has nucleated.
Thus the first part of Theorem~\ref{theorem:main} says that $\delta_{U_s}$ is a local minimizer with respect to localized variations in the interior, on faces 
and at edges of $\Omega$ and no nucleation can occur there. The second part says that, for Seiner's specimen, $\delta_{U_s}$ is not a local minimizer with 
respect to localized variations at some corner and hence the austenite can indeed nucleate at that corner. Therefore, Theorem~\ref{theorem:main} states 
that austenite must and indeed can nucleate at a corner.

We should point out that the requirement that $\Omega$ is admissible for $U_{s}$ as well as the conditions on the lattice parameters and the constraint 
(C-N) are only relevant for faces and edges; for corners and the interior no such conditions are required.

The intuition behind the definition of maximal directions is the following: suppose that $y$ is a deformation underlying a finite-energy admissible measure and consider a line segment along a maximal direction for $U_s$ joining two points on the prescribed boundary of $S(x_0)$; that is the boundary which is deformed according to $U_s x$. Note that such a line segment always exists if $\Omega$ is admissible. Since $Dy\in K^{qc}$ a.e., the definition of a maximal direction implies that the length of the deformed line segment cannot exceed the length of the same segment when transformed by $U_s x$, i.e. of the straight line joining the two points on the prescribed boundary in the deformed configuration; but, since this is the shortest distance, the line segment must have deformed by $U_s x$. The case of maximal directions for $U_s^{-1}$ is similar (but less intuitive) with the argument applied to the inverse deformation provided by the constraint (C-N). This type of rigidity is made clear through Lemma 8, 9 and 11.
\end{remark}

\paragraph{\textbf{Quasiconvexity conditions:}}
At this stage we introduce a set of quasiconvexity conditions in the interior, at faces, edges and corners; we use these to prove the first part of 
Theorem~\ref{theorem:main}. As discussed in \cite[p.~9]{ballopenprobs}, some care is needed when defining quasiconvexity conditions for integrands $W$ taking the value 
$+\infty$, as in our case, and we take the route of defining the conditions in terms of gradient Young measures. 
\begin{definition}
Let  $W:\mathbb{R}^{3\times 3}\rightarrow\mathbb{R}\cup\{+\infty\}$ be bounded below and Borel measurable; let $\Omega\subset\mathbb{R}^3$ be a bounded convex polyhedral domain and 
let $B=B(0,1)=\left\{x\in\mathbb{R}^{3}\,:\,\vert x\vert<1\right\}$, the unit ball in $\mathbb{R}^{3}$.
\begin{itemize}
\item[(i)] (interior) We say that $W$ is quasiconvex at $F\in\mathbb{R}^{3\times 3}$ in the interior of $\Omega$ if
\[\langle\mu,W\rangle\geq\,W(F)\]
for all $\mu\in\mathcal{B}_{i}$ where
\[\mathcal{B}_{i}=\left\{\mbox{$\mu\in\mathcal{G}^{\infty}(B,\mathbb{R}^{3\times3})$: $\mu$ homogeneous, $\bar{\mu}=F$}\right\}.\]
\item[(ii)] (face) Let a face of $\Omega$ be contained in the plane $\left\{x\cdot n=k\right\}$ for some outward pointing normal $n\in S^{2}$ and 
$k\in\mathbb{R}$; let $B_{f}=\left\{x\in B\,:\,x\cdot n<0\right\}$. We say that $W$ is quasiconvex at $F\in \mathbb{R}^{3\times 3}$ on that face if
\[\int_{B_{f}}\langle\mu_{x},W\rangle\,dx\geq\int_{B_{f}}W(F)\,dx\]
for all $\mu\in\mathcal{B}_{f}$ where
\[\mathcal{B}_{f}=\{\mu\in\mathcal{G}^{\infty}(B_f,\mathbb{R}^{3\times3}):\,\;\bar{\mu}_{x}=Dz(x)
\mbox{ a.e.~in $B_{f}$, $z$ satisfies (C-N) and}\,\,z(x)=Fx\,\,\mbox{on $\partial B\cap\partial{B_f}$}\}.\]
\item[(iii)] (edge) Let an edge of $\Omega$ be the intersection of two faces contained in the planes $\left\{x\cdot n_{i}=k_{i}\right\}$ for some outward pointing 
normals $n_{i}\in S^{2}$ and $k_{i}\in\mathbb{R}$, $i=1,2$; let $B_{e}=\left\{x\in B\,:\,x\cdot n_{i}<0,\,i=1,2\right\}$. We say that $W$ is 
quasiconvex at $F\in \mathbb{R}^{3\times 3}$ at that edge if
\[\int_{B_{e}}\langle\mu_{x},W\rangle\,dx\geq\int_{B_{e}}W(F)\,dx\]
for all $\mu\in\mathcal{B}_{e}$ where
\[\mathcal{B}_{e}=\{\mu\in\mathcal{G}^{\infty}(B_e,\mathbb{R}^{3\times3}):\,\;\bar{\mu}_{x}=Dz(x)
\mbox{ a.e.~in $B_{e}$, $z$ satisfies (C-N) and}\,\,z(x)=Fx\,\,\mbox{on $\partial B\cap\partial{B_e}$}\}.\]
\item[(iv)] (corner) Let a corner of $\Omega$ be the intersection of three faces contained in the planes $\left\{x\cdot n_{i}=k_{i}\right\}$ for some outward 
pointing normals $n_{i}\in S^{2}$ and $k_{i}\in\mathbb{R}$, $i=1,2,3$; let $B_{c}=\left\{x\in B\,:\,x\cdot n_{i}<0,\,i=1,2,3\right\}$. We say 
that $W$ is quasiconvex at $F\in \mathbb{R}^{3\times 3}$ at that corner if
\[\int_{B_{c}}\langle\mu_{x},W\rangle\,dx\geq\int_{B_{c}}W(F)\,dx\]
for all $\mu\in\mathcal{B}_{c}$ where
\[\mathcal{B}_{c}=\{\mu\in\mathcal{G}^{\infty}(B_c,\mathbb{R}^{3\times3}):\,\;\bar{\mu}_{x}=Dz(x)
\mbox{ a.e.~in $B_{c}$, $z$ satisfies (C-N) and}\,\,z(x)=Fx\,\,\mbox{on $\partial B\cap\partial{B_c}$}\}.\]
\end{itemize}
\label{def:quasiconvexity}
\end{definition}

\begin{remark}
The above definition of quasiconvexity in the interior at $F$ is equivalent to the existence of a nondecreasing sequence of everywhere finite (and thus continuous) 
quasiconvex functions $W^{(j)}$ such that $W^{(j)}(F)\to W(F)$ as $j\to\infty$ (see \cite[Remark 4]{newBJ}). It is natural to expect that similar equivalences hold for 
the other quasiconvexity conditions, but this lies outside the scope of this paper.
\end{remark}

\begin{remark}
We note that whenever a measure $\mu\in\mathcal{B}_{\omega}$, ${\omega}\in\{f,e,c\}$, satisfies $\mathrm{supp}\,\mu_{x}\subset K$, its underlying deformation $z$ becomes a mapping 
of bounded distortion and thus by Lemma~\ref{lemma:regularity} inherits the property of being homeomorphic with an inverse in the space $W^{1,\infty}(z(B_{\omega}),\mathbb{R}^{3})$.

Moreover, if we suppose that the map $W$ is finite and continuous everywhere and ignore the determinant and (C-N) constraints, the quasiconvexity conditions in the 
interior and on a face are essentially the standard quasiconvexity conditions in the interior and at the boundary but phrased in terms of Young measures. 
To see this, suppose that $W$ is quasiconvex in the interior at $F$ in the sense of Definition~\ref{def:quasiconvexity} and let 
$z\in Fx+W^{1,\infty}_{0}(\Omega,\mathbb{R}^{3})$; consider the measure $\nu=\delta_{Dz(\cdot)}$ which is clearly a $W^{1,\infty}$ gradient Young measure 
and define $\mu=\mathrm{Av}\,\nu$, the average of $\nu$, through its action on a continuous function $f$, by
\[
\langle\mu,f\rangle=\frac{1}{\vert\Omega\vert}\int_{\Omega}\langle\nu_{x},f\rangle\,dx.
\]
As remarked in Section~\ref{sec:1}, $\mu$ is a homogeneous gradient Young measure and satisfies $\bar{\mu}=F$. By the quasiconvexity of $W$ we get that
\begin{align}
\langle\mu, W\rangle\geq W(F)&\Rightarrow\frac{1}{\vert\Omega\vert}\int_{\Omega}\langle\nu_{x},W\rangle\,dx\geq W(F)\nonumber\\
&\Rightarrow\frac{1}{\vert\Omega\vert}\int_{\Omega}W(Dz(x))\,dx\geq W(F).
\label{laquila}
\end{align}
Conversely, suppose that \eqref{laquila} holds.  Let $\mu$ be a homogeneous $W^{1,\infty}$ gradient Young measure with $\bar{\mu}=F$ 
and consider its generating sequence $z^{k}$. We may assume that this lies in $W^{1,\infty}(\Omega,\mathbb{R}^{3})$ and, by a standard modification 
(e.g.~\cite{321}), we may also assume that $z^{k}\in Fx+W^{1,\infty}_{0}(\Omega,\mathbb{R}^{3})$. Using \eqref{laquila} for $z^k$ we deduce that
\begin{eqnarray*} W(F)&\leq &\lim_{k\to\infty}\frac{1}{|\Omega|}\int_\Omega W(Dz^k)\,dx\\
&=&\frac{1}{|\Omega|}\int_\Omega \langle\mu,W\rangle\,dx\\
&=&\langle\mu,W\rangle
\end{eqnarray*}
since $\mu$ is homogeneous. The case of a face is similar.

In regards to the other two conditions, these are natural extensions of the quasiconvexity conditions in the interior and at a face. The existence of such conditions at boundary points 
having conical singularities, such as at edges and corners, was discussed in \cite[p.~259 Remark 2]{ballmarsden}, but to the authors' knowledge this idea has not previously been applied.
\end{remark}

Next we prove that the above quasiconvexity conditions in the interior, at a face, an edge or a corner are sufficient for the measure $\delta_{U_{s}}$ 
to be a minimizer with respect to localized variations in the interior, at a face, an edge or a corner respectively.

\begin{lemma}
Let $\Omega\subset\mathbb{R}^{3}$ be a bounded convex polyhedral domain and suppose that $W:\mathbb{R}^{3\times 3}\rightarrow\overline{\mathbb{R}}$ is quasiconvex at 
$U_{s}$ in the interior (resp.~on faces, edges or at corners) of $\Omega$ in the sense of Definition~\ref{def:quasiconvexity}. Let $x_0$ belong to the interior (resp. a face, edge or corner). Then 
$I(\delta_{U_{s}})\leq I(\nu)$ for any ${\mathcal A}(x_0)$.
\label{lemma:qciffmin}
\end{lemma}

\begin{proof}
The proof for faces, edges and corners is similar and we only treat the case of a face; the case of the interior differs from the rest and we treat it last.

Suppose that $W$ is quasiconvex at $U_s$ at a face contained in $\left\{x\cdot n=k\right\}$ where $n\in S^2$ is outward pointing and $k\in\mathbb{R}$; 
let $x_0$ belong to the face and let $\nu=(\nu_x)_{x\in\Omega}\in\mathcal{A}(x_0)$ with underlying deformation $y\in W^{1,\infty}(\Omega,\mathbb{R}^3)$. We wish 
to deduce that $I(\nu)\geq I(\delta_{U_s})$. Since $W=+\infty$ outside $K$, we may also assume that
\[\mathrm{supp}\,\nu_x\subset K\:\:\:\mbox{a.e.}\]
as, otherwise, $I(\nu)>I(\delta_{U_s})$ and there is nothing to prove.

The set $S(x_0)\subset\Omega$ where $\nu$ is allowed to differ from $\delta_{U_{s}}$ is of the form $\left\{x\in B(x_0,r):x\cdot n<k\right\}$ for $r>0$ sufficiently small; clearly $k=x_0\cdot n$ and
\[
 S(x_0)=x_0+rB_f,
\]
where $B_f=\left\{x\in B(0,1):x\cdot n<0\right\}$. Then also $x_0+r(\partial B_{f}\cap\partial\Omega)=\partial S(x_0)\cap\partial\Omega$ and 
$x_0+r(\partial B_{f}\cap\Omega)=\partial S(x_0)\cap\Omega$.

We note that, since $\nu$ is a $W^{1,\infty}$ gradient Young measure, the parametrized measure $(\nu_x)_{x\in S(x_0)}$ is a $W^{1,\infty}$ gradient Young 
measure with underlying deformation $y\vert_{S(x_0)}$; this follows directly from Theorem~\ref{KPGYM}. Also, $y\vert_{S(x_0)}$ satisfies the (C-N) constraint 
since $y$, and thus $y\vert_{S(x_0)}$, is injective. Define $\mu=(\mu_{x})_{x\in B_{f}}$ by
\begin{equation}
\mu_{x}=\nu_{x_0+rx}.\nonumber
\end{equation}
We claim that $\mu$ is a $W^{1,\infty}$ gradient Young measure. Suppose that $y^k\in W^{1,\infty}(S(x_0),\mathbb{R}^3)$ generates the measure $(\nu_x)_{x\in S(x_0)}$; in 
particular, we may assume that $y^k\overset{\ast}{\rightharpoonup}y\vert_{S(x_0)}$ in $W^{1,\infty}(S(x_0),\mathbb{R}^3)$. For $x\in B_f$, define 
\[
 z^k(x)=\frac{1}{r}[y^k(x_0+rx)-U_sx_0].
\]
This is a sequence uniformly bounded in $ W^{1,\infty}(B_f,\mathbb{R}^3)$ and for any $\xi\in L^{1}(B_f)$, 
$\psi\in C_{0}(\mathbb{R}^{3\times 3})$,
\begin{eqnarray*}
\lim_{k\to\infty}\int_{B_f}\xi(x)\psi(Dz^k(x))\,dx&=&\lim_{k\to\infty}\frac{1}{r^3}\int_{S_f}\xi(\frac{x'-x_0}{r})\psi(Dy^k(x'))\,dx'\\
&=&\frac{1}{r^3}\int_{S_f}\xi(\frac{x'-x_0}{r})\langle\nu_{x'},\psi\rangle\,dx'\\
&=&\int_{B_f}\xi(x)\langle\mu_x,\psi\rangle\,dx,
\end{eqnarray*}
where we have used the change of variables $x'=x_0+rx$; i.e. the sequence $z^k$ generates the $W^{1,\infty}$ gradient Young measure $\mu$. Note also that 
$z(x)=\frac{1}{r}[y(x_0+rx)-U_sx_0]$ is the underlying deformation of $\mu$ and, for $x\in\partial B_f\cap\Omega$, $z(x)=U_sx$. Also, $z$ satisfies the 
(C-N) constraint as
\begin{eqnarray*}
\int_{B_f}\det Dz(x)\,dx&=&\frac{1}{r^3}\int_{S(x_0)}\det Dy(x')\,dx'\\
&\leq &\frac{1}{r^3}\mathcal{L}^3(y(S(x_0)))=\mathcal{L}^3(z(B_f)),
\end{eqnarray*}
since $y\vert_{S(x_0)}$ satisfies (C-N). In particular, $\mu\in\mathcal{B}_{f}$ and by the quasiconvexity assumption,
\begin{equation}
\int_{B_{f}}\langle\mu_{x},W\rangle\,dx\geq\int_{B_{f}}W(U_{s})dx.\nonumber
\end{equation}
Noting that $\mu_{x}=\nu_{x_0+rx}$, changing variables to $x'=x_0+rx$ and multiplying by $r^3$, we deduce that
\begin{equation*}
\int_{S(x_0)}\langle\nu_{z},W\rangle\,dz\geq\int_{S(x_0)}W(U_{s})dz.
\end{equation*} 
This is precisely what we need to show since then
\begin{eqnarray}
I(\nu)&=&\int_{S(x_0)}\langle\nu_{x},W\rangle\,dx+\int_{\Omega\setminus S(x_0)}W(U_{s})\,dx\nonumber\\
&\geq &\int_{S(x_0)}W(U_{s})\,dx+\int_{\Omega\setminus S(x_0)}W(U_{s})\,dx\nonumber\\
&=&I(\delta_{U_s}).\nonumber
\end{eqnarray}

For the interior, let $x_0\in\Omega$ and $\nu=(\nu_{x})_{x\in\Omega}\in\mathcal{A}(x_0)$; as before, we may assume that for a.e.~$x\in\Omega$, 
$\mathrm{supp}\,\nu_x\subset K$. Define $\mu:=\mathrm{Av}\,\nu$, so that $\mu$ is a homogeneous $W^{1,\infty}$ gradient Young measure with 
$\mathrm{supp}\,\mu\subset K$ and $\bar{\mu}=U_{s}$. Then, $\mu\in\mathcal{B}_{i}$ and by the quasiconvexity assumption we deduce that
$\langle\mu,W\rangle\geq W(U_{s})$ and hence $I(\mu)\geq I(\delta_{U_{s}})$.
To finish the proof, it suffices to show that $I(\nu)=I(\mu)$ so that $I(\nu)\geq I(\delta_{U_s})$. Since the measures $\mu$ and $\nu_x$, for a.e.~$x$, are 
supported in $K$, we may replace $W$ by any continuous function agreeing with $W$ on $K$ (not relabelled) and then
\begin{equation*}
I(\nu)=\int_{\Omega}\langle\nu_{x},W\rangle\;dx=\vert\Omega\vert\langle\mu,W\rangle=\int_{\Omega}\langle\mu,W\rangle\;dx=I(\mu).
\end{equation*}\qed
\end{proof}

Thus to prove the first part of Theorem~\ref{theorem:main} it will suffice to show that $W$ is quasiconvex at $U_s$ in the interior and at faces and edges.

\begin{remark}
We note that for an interior point $x_0$ quasiconvexity of $W$ at $U_s$ in the interior is 
necessary for $\delta_{U_s}$ to be a minimizer with respect to localized variations in ${\mathcal A}(x_0)$. This is because by 
Lemma~\ref{lemma:regularity} the underlying deformation of any $\mu\in {\mathcal B}_i$ is a homeomorphism, so that $\mu$ can be rescaled to the ball 
$B(x_0,r)$ and extended by $\delta_{U_s}$ in the rest of $\Omega$, so that the resulting measure belongs to ${\mathcal A}(x_0)$. However, for faces, edges and corners, 
necessity does not obviously follow since measures in ${\mathcal A}(x_0)$ are required to satisfy the (C-N) constraint. 
For example, assume that $x_0$ belongs to a face contained in $\left\{x\cdot n=k\right\}$, $\delta_{U_s}$ is a 
minimizer with respect to variations in $\mathcal{A}(x_0)$ and let $\mu\in\mathcal{B}_{f}$. Taking $r>0$ such that $S(x_0)=x_0+rB_{f}$ and $\overline{S(x_0)}\setminus\{x\cdot n=k\}\subset\Omega$, 
define $\nu=(\nu_x)_{x\in\Omega}$ by
\begin{equation}
\nu_{x}=\left\{\begin{array}{ccc}
\mu_{\frac{x-x_0}{r}},&\quad &x\in x_0+rB_{f}\\
\delta_{U_{s}},&\quad &x\in\Omega\setminus (x_0+rB_f).
\end{array}\right.\nonumber
\end{equation}
As in the above proof, we may assume that $\mathrm{supp}\,\nu_x\subset K$ a.e.~so that $\nu$ is a $W^{1,\infty}$ gradient Young measure and its 
underlying deformation preserves orientation and (up to a constant) satisfies the boundary condition, i.e.~if $\nu$  satisfied the (C-N) constraint 
we would have that $\nu\in{\mathcal A}(x_0)$. We could then use the fact that $\delta_{U_s}$ is a minimizer to get that $I(\nu)\geq I(\delta_{U_s})$. But this implies that
\begin{equation}
\int_{x_0+rB_{f}}\langle\mu_{\frac{x-x_0}{r}},W\rangle\,dx\geq\int_{x_0+rB_{f}}W(U_{s})\,dx.\nonumber
\end{equation}
Then, making the change of variables $x'=(x-x_0)/r$ and dividing by $r^{3}$,
\[\int_{B_{f}}\langle\mu_{x'},W\rangle\,dx'\geq\int_{B_{f}}W(U_{s})\,dx',\]
proving the quasiconvexity of $W$ at $U_s$ on that face. However, it is entirely possible that the underlying deformation $z$ of $\mu$ satisfies (C-N) but that 
the underlying deformation, say $y$, of $\nu$ is not a.e.~injective in $\Omega$ and thus does not satisfy (C-N). This can happen whenever the image of the 
free boundary of $B_f$ `goes around a corner' and comes close to, or even into contact with, the image of the prescribed part of the boundary 
$\partial B_f\cap\Omega$; see Fig.~\ref{fig:cnfails} for an example.

\begin{figure}[ht]
	\centering
	\def\svgwidth{0.8\columnwidth}
	\begingroup
    \setlength{\unitlength}{\svgwidth}
  \begin{picture}(1,0.4)%
    \put(0,0){\includegraphics[width=\unitlength]{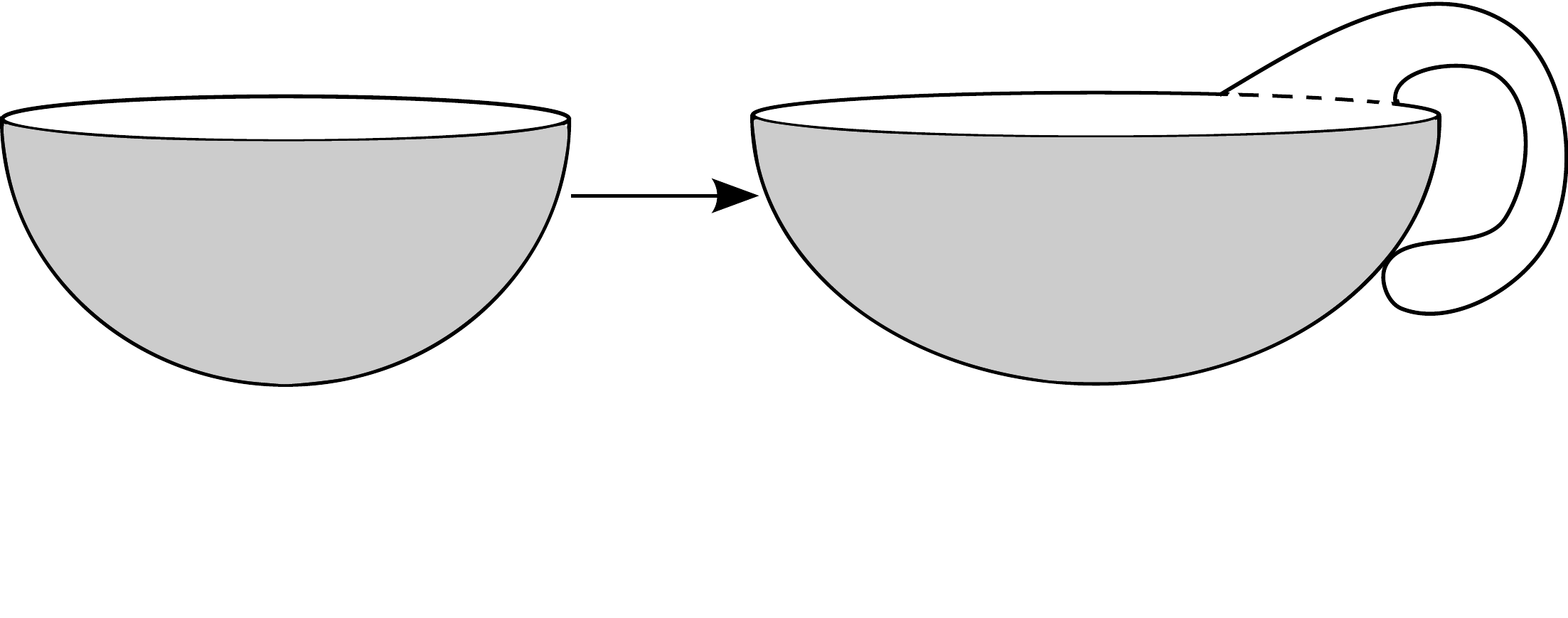}}%
    \put(0.41,0.3){\color[rgb]{0,0,0}\makebox(0,0)[lb]{\smash{$z$}}}%
    \put(0.16,0.225){\color[rgb]{0,0,0}\makebox(0,0)[lb]{\smash{$B_f$}}}%
    \put(0.64,0.225){\color[rgb]{0,0,0}\makebox(0,0)[lb]{\smash{$z(B_f)$}}}%
\end{picture}
\endgroup
	\caption{Depiction of a possible deformation $z$ underlying a measure $\mu\in\mathcal{B}_f$ which may satisfy the (C-N) constraint but the 
underlying deformation of $\nu$ cannot; in particular, $\nu\notin\mathcal{A}(x_0)$.}
	\label{fig:cnfails}
\end{figure}
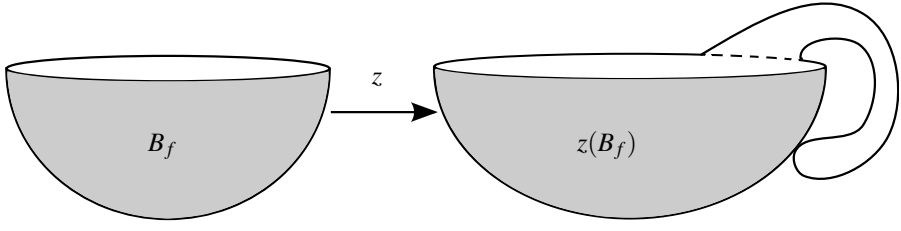

Nevertheless, we mention that by altering the definition of the measures in $\mathcal{B}_{\omega}$, 
$\omega\in\left\{f,e,c\right\}$, it is possible to retain the necessity of the quasiconvexity conditions at $U_s$ on a face, an edge or at a corner for the 
measure $\delta_{U_s}$ to be a minimizer amongst the respective localized variations.
One possibility is to add a \textit{confinement condition} on the deformations underlying measures in $\mathcal{B}_{\omega}$ in the spirit of the 
confinement condition of Ciarlet \& Ne\v{c}as in~\cite{ciarletnecas}, so that the localized problem respects the nature of the original one.

To motivate this in our context, let $x_0$ be a typical point in the interior of an edge of $\Omega$ contained in the intersection of the planes 
$\left\{x\cdot n_i=k_i\right\}$, $S(x_0)$ an appropriate set of the form $B(x_0,r)\cap\Omega$ and $y$ a deformation underlying a measure in $\mathcal{A}(x_0)$ so 
that $y(x)=U_sx$ outside $S(x_0)$. Writing $S(x_0)=x_0+rB_e$ we see that the deformation $z$ defined for $x\in B_e$ by $z(x)=\frac{1}{r}[y(x_0+rx)-U_sx_0]$ is 
injective not only in $B_e$ but also in the larger set $D_r\subset\mathbb{R}^3$ such that $\Omega=x_0+rD_r$; 
then, $z$ can be extended linearly by $U_sx$ to $D_r$ while satisfying (C-N). By blowing up around the point $x_0$, i.e. by letting $r\rightarrow0$, we see 
that a natural confinement condition would be to require that any deformation $z$ underlying a measure in $\mathcal{B}_{e}$ can be extended linearly by 
$U_sx$ to the wedge 
\[D:=\bigcup^{2}_{i=1}\left\{x\cdot n_i<0\right\}\]
and that the extension satisfies (C-N) on any bounded open subset of the wedge. Similarly, for faces 
or corners, $D$ would be the half-space defined by the face or the `octant' defined by the three faces meeting at the corner respectively.

If we add the confinement condition in the definition of ${\mathcal B}_{\omega}$ as above, the quasiconvexity conditions at $U_s$ become trivially necessary for 
$\delta_{U_s}$ to be a minimizer in the respective class of admissible measures, but the conditions are no longer obviously sufficient. 
Restricting attention to a corner $x_0$ say, it is possible for a map $y$, underlying a measure 
in $\mathcal{A}(x_0)$, to map part of $S(x_0)$ into the region $U_sD$, where $D$ is the corresponding `octant', so that the map 
$z(x)=\frac{1}{r}[y(x_0+rx)-U_sx_0]$ is not injective in $D$. Hence, employing (C-N) in this way weakens our main result and it is thus not preferred.
Nevertheless, since finite-energy measures are almost everywhere supported in $K$, one may use the resulting uniform Lipschitz condition on the maps 
underlying admissible measures in $\mathcal{A}(x_0)$ to show that, if $x_0$ is the corner in question, there exists a neighbourhood $B(x_0,r)$ such that 
$\delta_{U_s}$ remains a minimizer amongst measures in $\mathcal{A}(x_0)$ provided $S(x_0)\subset B(x_0,r)$, i.e.~austenite cannot nucleate in 
$B(x_0,r)\cap\Omega$; similar statements hold for faces and edges.

Clearly, one could also assume the same confinement condition for the maps underlying our admissible measures and get both sufficiency and necessity as 
well as our main result. However, this is much too strong a condition and does not seem natural.
\end{remark}

We now proceed to the proof of our main result where we distinguish between three cases: corners, interior, and faces and edges.

\subsection{Corners}
\label{subsec:corners}

To resolve the second part of Theorem~\ref{theorem:main}, we construct an explicit Young measure in $\mathcal{A}(x_0)$ that lowers the energy. We note that the 
construction depends heavily on the orientation of $\Omega$ and the lattice parameters of the material.

\begin{lemma}
Assume that $\Omega\subset\mathbb{R}^3$ is a rectangular parallelepiped with edges along the vectors $e_1$, $e_2$, $e_3$ and 
define a coordinate system such that the axes are parallel to the edges and each corner of $\Omega$ belongs to a different octant $O_i$, $i=1,\ldots,8$. 
Fix $s\in\{1,\ldots,N\}$ and suppose that there exists $l\in\{1,\ldots,N\}$, $l\neq s$, such that
\begin{align}
\label{eq:compatibilitycorner1}
QU_l-U_s &=a\otimes n\\
\label{eq:compatibilitycorner2}
U_s+\lambda a\otimes n &=R+b\otimes m,
\end{align}
for some $\lambda\in(0,1)$, $R$, $Q\in SO(3)$, $a$, $b\in\mathbb{R}^3$ and $n$, $m\in S^2$ such that $n$, $m$ belong to the same octant, say $O_k$, they are not perpendicular to the vectors $e_i$, $i=1,2,3$ and $U^{-1}_{s}b\cdot n<0$. 
Then, for each of the corners $x_0$ belonging to the octants $O_k$ and $-O_k$, there exists an admissible measure $\nu\in\mathcal{A}(x_0)$ such that $I(\nu)<I(\delta_{U_s})$.

In particular, for the CuAlNi specimen of the experiment with $e_i$, $i=1,2,3$, the standard basis of $\mathbb{R}^3$ and lattice parameters $\alpha=1.06372$, $\beta=0.91542$, $\gamma=1.02368$, for each $s\in\{1,\ldots,6\}$, there exist precisely four such corners.
\label{prop:corner}
\end{lemma}

\begin{proof}
Under the assumptions of the lemma, consider the corner $x_0$ of $\Omega$ lying in the octant $O_k$ and let 
$S(x_0)\subset\Omega$ be an appropriate set of the form $B(x_0,r)\cap\Omega$. Define
\begin{eqnarray*}
 S^1&:=&\left\{x\in S(x_0): x\cdot m>k_m\right\},\\
 S^2&:=&\left\{x\in S(x_0): x\cdot m<k_m\:\:\mbox{and}\:\:x\cdot n>k_n\right\},\\
 S^3&:=&\left\{x\in S(x_0): x\cdot n<k_n\right\}
\end{eqnarray*}
where $k_m$, $k_n\in\mathbb{R}$ are such that
\begin{equation}
\label{eq:cornerconstruction0}
\left\{x\in \Omega: x\cdot m\geq k_m\right\}\cap\left\{x\in \Omega: x\cdot n\leq k_n\right\}=\emptyset.
\end{equation}
Note that, under the above choice of normals, it is always possible to choose $k_m$, $k_n$ verifying (\ref{eq:cornerconstruction0}) 
and such that the sets $S^i$, $i=1,2,3$, are nonempty.

Define a parametrized measure $\nu=(\nu_x)_{x\in\Omega}$, as in Figure~\ref{fig:min_corner}, by
\begin{equation}
\nu_{x}=\left\{\begin{array}{lcl}
\delta_{R},&\quad &x\in S^{1}\\
(1-\lambda)\delta_{U_{s}}+\lambda\delta_{QU_{l}},&\quad &x\in S^{2}\\
\delta_{U_{s}},&\quad &x\in S^3\cup(\Omega\setminus S(x_0)).
\end{array}\right.
\end{equation}
If this measure belongs to $\mathcal{A}(x_0)$, the proof is complete as then
$$I(\nu)=-\delta\mathcal{L}^3(S^{1})<0=I(\delta_{U_s}).$$
Let us first verify that $\nu$ defines a $W^{1,\infty}$ gradient Young measure. Note that in each of the regions $S^{1}$, $S^{2}$ and 
$S^3\cup(\Omega\setminus S(x_0))$, $\nu_{x}$ is a homogeneous $W^{1,\infty}$ gradient Young measure with underlying deformation gradient given by
\begin{equation}
\label{eq:5.1proof1}
Dy(x)=\left\{\begin{array}{lcl}
R,&\quad &x\in S^{1}\\
(1-\lambda)U_{s}+\lambda\,QU_{l},&\quad &x\in S^{2}\\
U_{s},&\quad &x\in S^3\cup(\Omega\setminus S(x_0)).
\end{array}\right.
\end{equation}
This is trivially the case for the regions $S^{1}$ and $S^3\cup(\Omega\setminus S(x_0))$; as for the simple laminate in $S^{2}$, it can be constructed as the 
weak$\ast$ limit of a sequence $y^k$ uniformly bounded in $W^{1,\infty}(S^2,\mathbb{R}^{3})$ such that 
$$\mathrm{dist}(Dy^k,\left\{U_s,QU_l\right\})\rightarrow 0$$ in measure, i.e.~the associated measure is a $W^{1,\infty}$ gradient Young measure supported on 
these two matrices; see e.g.~\cite{1}.

To verify that $\nu=(\nu_{x})_{x\in\Omega}$ is itself a $W^{1,\infty}$ gradient Young measure, by~\cite{36} we simply need to check that there is a $y$ belonging to 
$W^{1,\infty}(\Omega,\mathbb{R}^{3})$ satisfying \eqref{eq:5.1proof1}. But this reduces to verifying Hadamard's jump condition across the interfaces 
$\left\{x\cdot m=k_m\right\}$ and $\left\{x\cdot n=k_n\right\}$. This is immediate from \eqref{eq:compatibilitycorner2} for 
the interface $\left\{x\cdot m=k_m\right\}$, while for the interface $\{x\cdot n=k_n\}$ we have that 
\[\left[(1-\lambda)U_{s}+\lambda\,QU_{l}\right]-U_{s}=\lambda(QU_{l}-U_{s})=\lambda\,a\otimes n,\]
as required.

As $\nu$ is supported in $K$ a.e.~in $\Omega$, $\det Dy(x)>0$ a.e.~and it remains to verify the boundary condition on $\partial S(x_0)\cap\Omega$ and the 
(C-N) constraint. The underlying deformation $y$ is given up to a constant by
\begin{equation}
y(x)=\left\{\begin{array}{lcl}
Rx+k_mb-\lambda k_na,&\quad &x\in S^{1}\\
(U_s+\lambda a\otimes n)x-\lambda k_na,&\quad &x\in S^{2}\\
U_{s}x,&\quad &x\in S^3\cup(\Omega\setminus S(x_0)).
\end{array}\right.
\end{equation} 
Clearly, $\partial S(x_0)\cap\Omega\subset S^3\cup(\Omega\setminus S(x_0))$ and then $y\vert_{\partial S(x_0)\cap\Omega}=U_sx$. As for the (C-N) constraint, it 
suffices to show that $y$ is injective. To reach a contradiction, suppose that $x_i\neq x_j$ but $y(x_i)=y(x_j)$; there are three non-trivial cases to 
consider:
\begin{itemize}
 \item[(a)] $x_1\in \overline{S^1}$ and $x_2\in \overline{S^2}$;
 \item[(b)] $x_2\in \overline{S^2}$ and $x_3\in \overline{S^3\cup(\Omega\setminus S(x_0))}$;
 \item[(c)] $x_1\in \overline{S^1}$ and $x_3\in \overline{S^3\cup(\Omega\setminus S(x_0))}$.
\end{itemize}
Let us first treat case (b); $y(x_2)=y(x_3)$ implies that
\begin{equation}
\label{eq:cornerconstruction1}
 U_s(x_2-x_3)=\lambda(k_n-x_2\cdot n)a.
\end{equation}
But $U_s+a\otimes n=QU_l$ so that, by taking determinants on both sides, we infer that $(\det U_s)(1+U^{-1}_{s}a\cdot n)=\det U_l$; but 
$\det U_s = \det U_l\neq 0$ and hence
\begin{equation}
\label{eq:cornerconstruction2}
U^{-1}_{s}a\cdot n=0.
\end{equation}
Now multiplying \eqref{eq:cornerconstruction1} to the left by $U^{-1}_{s}$, taking the dot product with $n$ and using (\ref{eq:cornerconstruction2}), we 
find that $(x_2-x_3)\cdot n=0$. Since $x_2\cdot n\geq k_n$, $x_3\cdot n\leq k_n$ it follows that $x_2\cdot n=x_3\cdot n=k_n$ and hence $U_sx_2=U_s x_3$, implying $x_2=x_3$, a contradiction.

Next we treat case (a). Now, $y(x_1)=y(x_2)$ implies that
\begin{equation}
\label{eq:cornerconstruction3}
R(x_1-x_2)=(x_2\cdot m-k_m)b,
\end{equation}
where we have made use of (\ref{eq:compatibilitycorner2}). Also, taking determinants in (\ref{eq:compatibilitycorner2}) and using 
(\ref{eq:cornerconstruction2}), we infer that $\det U_s=1+R^Tb\cdot m$ and hence
\begin{equation}
\label{eq:cornerconstruction4}
-1<R^Tb\cdot m\leq0,
\end{equation}
since we are also assuming that $0<\det U_{s}\leq1$. Multiplying \eqref{eq:cornerconstruction3} on the left by $R^T$, taking the dot product with $m$ and 
using (\ref{eq:cornerconstruction4}), we find that 
$$(x_1-x_2)\cdot m\leq -(x_2\cdot m-k_m),$$
since also $x_2\cdot m\leq k_m$. Thus $x_1\cdot m=x_2\cdot m$ and hence $Rx_1=Rx_2$, contradicting $x_1\neq x_2$. 
(In fact we do not need to invoke the hypothesis $\det U_s\leq 1$ here, since if $\det U_s>1$ then $R^Tb\cdot m>0$ and we get a similar contradiction from \eqref{eq:cornerconstruction3}.)

As for case (c), we repeat a similar argument which reduces the injectivity of $y$ to the condition $U^{-1}_{s}b\cdot n<0$. However, we are unable to check this sign 
for general lattice parameters and, instead, we verify this numerically for Seiner's specimen. Suppose then that $y(x_1)=y(x_3)$, i.e.
$$Rx_1+k_mb-\lambda k_na=U_sx_3.$$ 
By (\ref{eq:compatibilitycorner2}), we may write $R=U_s+\lambda a\otimes n - b\otimes m$ so that $y(x_1)=y(x_3)$ becomes
\begin{equation}
\label{eq:cornerconstruction5}
 U_s(x_1-x_3)=\lambda(k_n-x_1\cdot n)a+(x_1\cdot m-k_m)b.
\end{equation}
We may now multiply to the left by $U^{-1}_{s}$ and take the dot product with $n$ so that (\ref{eq:cornerconstruction5}) becomes
$$(x_1-x_3)\cdot n=(x_1\cdot m-k_m)U^{-1}_{s}b\cdot n.$$
However, $x_1\cdot n>k_n$ by (\ref{eq:cornerconstruction0}) and hence $(x_1-x_3)\cdot n>0$; also, $x_1\cdot m\geq k_m$ and we reach a contradiction provided 
that $U^{-1}_{s}b\cdot n<0$.

We note that taking $\tilde{a}=-a$, $\tilde{n}=-n$, $\tilde{b}=-b$ and $\tilde{m}=-m$ in (\ref{eq:compatibilitycorner1}) and 
(\ref{eq:compatibilitycorner2}) does not alter anything in the above argument. Thus, if the construction is possible for a corner in the octant 
$O_k$ say, it is also possible for the corner in the octant $-O_k$.

In particular, for Seiner's specimen and each $s\in\left\{1,\ldots,6\right\}$, we can choose two sets of solutions to the twin and habit 
plane equations (\ref{eq:compatibilitycorner1}) and (\ref{eq:compatibilitycorner2}) respectively, such that the twin and habit plane normals are not 
perpendicular to any edge and they lie in the same octant as a corner. Provided $U_s^{-1}b\cdot n<0$ each set of solutions allows the construction of a microstructure 
lowering the energy for either of a pair of opposite corners, the two pairs of corners being distinct. This is indeed possible and we refer the reader 
to Appendix~B for the details. There, in Tables B.8 and B.9, one may also find a summary of the results for all values of $s\in\left\{1,\ldots,6\right\}$.\qed 
\end{proof}

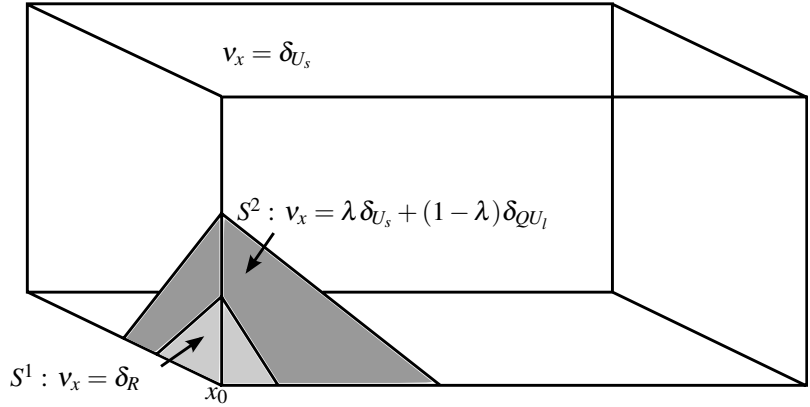
\begin{figure}[ht]
	\centering
	\def\svgwidth{0.7\columnwidth}	
	\begingroup
    \setlength{\unitlength}{\svgwidth}
  \begin{picture}(1,0.59959741)%
    \put(0,0){\includegraphics[width=\unitlength]{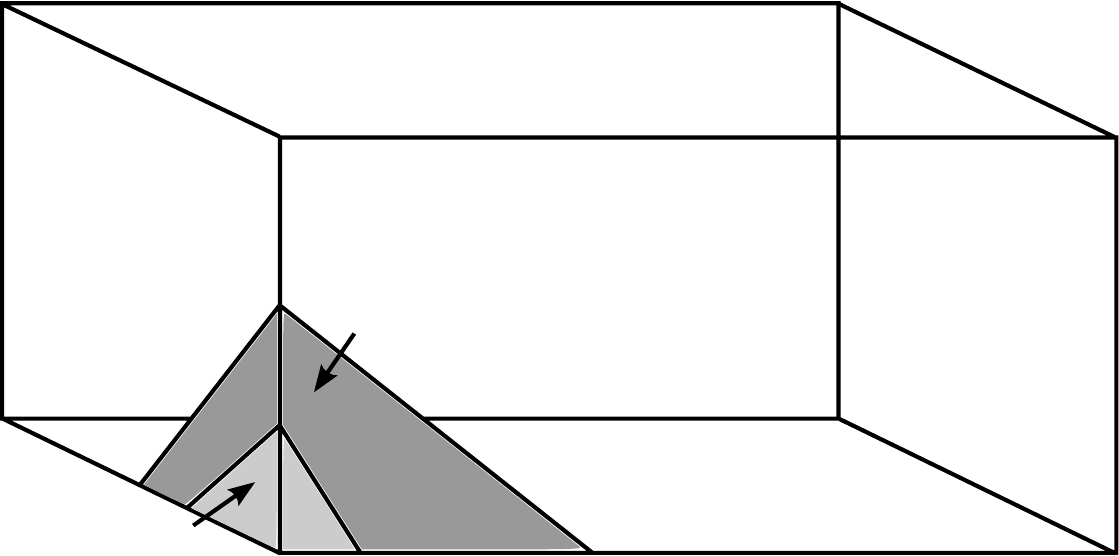}}%
    \put(0.25,0.42){\color[rgb]{0,0,0}\makebox(0,0)[lb]{\smash{$\nu_x=\delta_{U_{s}}$}}}%
    \put(-0.02,0.001){\color[rgb]{0,0,0}\makebox(0,0)[lb]{\smash{$S^{1}:\,\nu_x=\delta_R$}}}%
    \put(0.23,-0.015){\color[rgb]{0,0,0}\makebox(0,0)[lb]{\smash{$x_0$}}}%
    \put(0.27,0.215){\color[rgb]{0,0,0}\makebox(0,0)[lb]{\smash{$S^{2}:\,\nu_x=\lambda\delta_{U_{s}}+(1-\lambda)\delta_{QU_{l}}$}}}%
  \end{picture}%
\endgroup
	\caption{Depiction of a measure $\nu\in\mathcal{A}(x_0)$ such that $I(\nu)<I(\delta_{U_{s}})$ for a corner $x_0$. In the region $S^{1}$, $\nu_x=\delta_R$ for some $R\in\,SO(3)$ 
	so that austenite has nucleated at a corner; in the region $S^{2}$, $\nu_x=\lambda\delta_{U_{s}}+(1-\lambda)\delta_{QU_{l}}$ for some $Q\in\,SO(3)$ and $l\in\lbrace1,\ldots,6\rbrace$ such that the 
	matrices $R$ and $\lambda U_{s}+(1-\lambda)QU_{l}$ are rank-one connected, i.e.~$\nu_x$ corresponds to a simple laminate formed from the gradients $U_{s}$ and $QU_{l}$ there, forming a compatible interface with $R$. 
	Note that the normals to the interfaces between the austenite and the simple laminate (habit plane) and between the simple laminate and the pure phase of $U_{s}$ (twinned-to-detwinned interface) are different.}
	\label{fig:min_corner}
\end{figure}

\subsection{Interior}
\label{subsec:interior}

\quad

For the case of a point $x_0$ belonging to the interior we wish to deduce that the stabilized martensite $\delta_{U_{s}}$ is indeed a minimizer with respect to localized 
variations in $\mathcal{A}(x_0)$. In particular, we prove the following:

\begin{lemma}
The map $W:\mathbb{R}^{3\times 3}\rightarrow\overline{\mathbb{R}}$ given by (\ref{eq:W}) is quasiconvex at $U_{s}$ in the interior.
\label{lemma:quasiconvexW}
\end{lemma}

By the above lemma, along with Lemma~\ref{lemma:qciffmin}, we may immediately infer the required result which we provide for completeness in the form of a 
corollary.

\begin{corollary}
Assume that $\Omega\subset\mathbb{R}^{3}$ is a bounded convex polyhedral domain. Then, if $x_0\in \Omega$,
\[I(\nu)\geq I(\delta_{U_s})\mbox{ for all }\nu\in{\mathcal A}(x_0).\]
In particular, nucleation cannot occur in the interior.
\label{prop:interior}
\end{corollary}

\begin{proof}\hspace{-0.15cm} {\it of Lemma}~\ref{lemma:quasiconvexW}
Let $\mu\in\mathcal{B}_{i}$; that is $\mu$ is a homogeneous gradient Young measure satisfying $\bar{\mu}=U_{s}$. We need to show that $\langle \mu,W\rangle\geq W(U_s)$, and thus we may assume that 
$\mathrm{supp}\,\mu\subset K$. Since the map $F\mapsto\det F$ is quasiaffine,
\begin{eqnarray}
\det U_{s}&=&\langle\mu,\det\rangle\nonumber\\
&=&\int_{SO(3)}\det A\;d\mu(A)+\int_{\bigcup_{i}SO(3)U_{i}}\det A\;d\mu(A)\nonumber\\
&=&\int_{SO(3)}1\;d\mu(A)+\int_{\bigcup_{i}SO(3)U_{i}}\det U_{s}\;d\mu(A),
\label{eq:prooflemmabarnu=us1}
\end{eqnarray}
since $\det U_i = \det U_s$ for all $i=1,\ldots,N$, the variants being symmetry related. On the other hand, as $\mu$ is a probability measure,
\begin{equation}
\det U_{s}=\int_{SO(3)}\det U_{s}\;d\mu(A)+\int_{\bigcup_{i}SO(3)U_{i}}\det U_{s}\;d\mu(A).
\label{eq:prooflemmabarnu=us2}
\end{equation}
Subtracting equations (\ref{eq:prooflemmabarnu=us1}) and (\ref{eq:prooflemmabarnu=us2}), we obtain
\[\int_{SO(3)}(1-\det U_{s})\;d\mu(A)=0.\]
Therefore, $\mu(SO(3))=0$ or $\det\,U_{s}=1$; the former case is precisely what we need to prove. So, let $\det\,U_{s}=\alpha\beta\gamma=1$ where 
$\alpha$, $\beta$ and $\gamma$ are the lattice parameters and eigenvalues of $U_s$. By the AM-GM inequality
\begin{equation}
\frac{\vert U_{s}\vert^{2}}{3}=\frac{\alpha^{2}+\beta^{2}+\gamma^{2}}{3}\geq(\alpha^{2}\beta^{2}\gamma^{2})^{1/3}=1\nonumber
\end{equation}
and thus $\vert U_{s}\vert^{2}>3=\vert\mathbf{1}\vert^{2}$. Note that the inequality is strict as otherwise $\alpha=\beta=\gamma=1$ and 
$U_{i}=\mathbf{1}$ for all $i=1,\ldots\,,N$. Also, the map $F\mapsto \vert F\vert^{2}$ is convex and hence
\begin{eqnarray}
\vert U_{s}\vert^{2}&\leq &\langle\mu,\vert\cdot\vert^{2}\rangle\nonumber\\
&=&\int_{SO(3)}\vert A\vert^{2}\;d\mu(A)+\int_{\bigcup_{i}SO(3)U_{i}}\vert A\vert^{2}\;d\mu(A)\nonumber\\
&=&\int_{SO(3)}3\;d\mu(A)+\int_{\bigcup_{i}SO(3)U_{i}}\vert U_{s}\vert^{2}\;d\mu(A)
\label{eq:prooflemmabarnu=us3}
\end{eqnarray}
as the martensitic variants are symmetry related and the norm stays constant. Moreover, $\mu$ being a probability measure,
\begin{equation}
\label{eq:prooflemmabarnu=us4}
\vert U_{s}\vert^{2}=\int_{SO(3)}\vert U_{s}\vert^{2}\;d\mu(A)+\int_{\bigcup_{i}SO(3)U_{i}}\vert U_{s}\vert^{2}\;d\mu(A)
\end{equation}
and subtracting equations (\ref{eq:prooflemmabarnu=us3}), (\ref{eq:prooflemmabarnu=us4}), we infer that
\begin{equation}
\int_{SO(3)}(3-\vert U_{s}\vert^{2})\;d\mu(A)\geq 0.\nonumber
\end{equation}
However, $\vert U_{s}\vert^{2}>3$ and hence, $\mu(SO(3))=0$ which completes the proof.\qed
\end{proof}

\begin{remark}
\label{remark:homogeneous}
The proof of Lemma~\ref{lemma:quasiconvexW} only uses the fact that $\bar{\mu}=U_{s}$. For an inhomogeneous measure $\mu=(\mu_x)_{x\in\Omega}$ with 
$\mathrm{supp}\,\mu_x\subset K$ and $\bar{\mu}_x=U_{s}$ a.e.~in $\Omega$, the same argument applies to show that for a.e.~$x\in\Omega$, $\mu_x(SO(3))=0$. 
We keep this remark in mind as it will be central in proving quasiconvexity at faces and edges.
\end{remark}

\subsection{Faces and Edges}
\label{subsec:faces&edges}

We treat the cases of faces and edges simultaneously. Both for faces and edges, we wish to conclude that if $x_0$ belongs to a face or edge then 
$\delta_{U_s}$ is a minimizer of $I$ in $\mathcal{A}(x_0)$, which in turn implies that nucleation cannot occur there. 
To prove this we establish the quasiconvexity of $W$ at $U_{s}$ on faces and edges, from which the desired result follows from Lemma~\ref{lemma:qciffmin}.

For convenience, we shall use the subscript $f,e$ in e.g.~$\mathcal{B}_{f,e}$ or $B_{f,e}$ to mean `either $\mathcal{B}_{f}$ or $\mathcal{B}_{e}$' and `either $B_{f}$ or 
$B_{e}$'. Here ${\mathcal B}_{f,e}$ is defined as in Definition~\ref{def:quasiconvexity} with $F=U_s$.

\begin{lemma}
\label{prop:feqc}
Assume that $\Omega\subset\mathbb{R}^{3}$ is a bounded convex polyhedral domain which is admissible for $U_s$ and let 
$W:\mathbb{R}^{3\times 3}\rightarrow\overline{\mathbb{R}}$ be as in (\ref{eq:W}). Then $W$ is quasiconvex at $U_{s}$ on all faces and edges of $\Omega$.
\end{lemma}

\begin{corollary}
\label{prop:fe}
Assume that $\Omega\subset\mathbb{R}^{3}$ as above is admissible for $U_s$. Then, if $x_0$ belongs to a face or edge
\begin{equation}
I(\nu)\geq I(\delta_{U_s})\mbox{ for all }\nu\in{\mathcal A}(x_0).\nonumber
\end{equation}
In particular, nucleation cannot occur at any face or edge.
\end{corollary}

\begin{remark}
Note that the statement of Lemma~\ref{prop:feqc} requires that $\Omega$ be admissible for $U_s$ and it is here that we use this 
condition for the first time. The (C-N) constraint as well as the conditions on the lattice parameters enter here too. This will become apparent in the proofs 
that follow.
\end{remark}

By Remark~\ref{remark:homogeneous}, we have the following lemma to which the proof of Lemma~\ref{prop:feqc} will be reduced.

\begin{lemma}
Assume that for a face or an edge $\mu\in\mathcal{B}_{f,e}$ satisfies $\bar{\mu}_{x}=U_{s}$ for a.e.~$x\in B_{f,e}$. Then,
\[\int_{B_{f,e}}\langle\mu_{x}, W\rangle\,dx\geq\int_{B_{f,e}}W(U_{s})\,dx,\]
i.e.~$W$ is quasiconvex at $U_s$ on that face or edge. In particular, if $\mathrm{supp}\,\mu_{x}\subset K$ a.e.~then $\mu_x(SO(3))=0$ for 
a.e.~$x\in B_{f,e}$.
\label{lemmabarnu=us}
\end{lemma}

Thus, to prove Lemma~\ref{prop:feqc}, it suffices to show that all finite-energy measures in $\mathcal{B}_{f}$ and $\mathcal{B}_{e}$ satisfy 
$\bar{\mu}_{x}=U_{s}$. In the case of the interior we reduced the problem to (homogeneous) measures $\mu$ such that $\bar{\mu}=U_{s}$ due to the fact that 
the boundary condition on deformations underlying measures in $\mathcal{A}(x_0)$ was satisfied on the entire boundary of $S(x_0)$. However, on a face or at an edge, averaging the 
measures does not work. Therefore, we follow a different approach in order to establish that $\bar{\mu}_{x}=U_{s}$ which is dependent on the orientation of $\Omega$ or, 
equivalently, the directions in which the specimen is cut.

This is based on a rigidity argument, which uses the notions of maximal directions for $U_{s}$ and $U^{-1}_{s}$ from Definition~\ref{definitionmaximal}, to 
deduce that $\bar{\mu}_{x}=U_{s}$ a.e.~in $B_{f,e}$ for all measures in $\mathcal{B}_{f,e}$, provided that $\Omega$ is admissible itself. This argument is 
presented in Lemma~\ref{lemma:rigidity} and applied in Lemma~\ref{lemmamaximaldirectionsus} and Lemma~\ref{lemmamaximaldirectionsusinverse}.

\begin{lemma}
\label{lemma:rigidity}
Let $t^{-}$, $t^{+}\in\mathbb{R}$ with $t^{-}<t^{+}$ and suppose that $\sigma:[t^{-},t^{+}]\rightarrow\mathbb{R}^d$, $d\geq1$, is absolutely continuous, and satisfies
\[\|\frac{d}{dt}\sigma\|_{\infty}\leq\frac{\vert\sigma(t^{+})-\sigma(t^{-})\vert}{\vert t^{+}-t^{-}\vert}.\]
Then,
\[
\sigma(t)=\sigma(t^{-})+(t-t^{-})(\sigma(t^{+})-\sigma(t^{-}))/(t^{+}-t^{-})
\]
for all $t\in[t^-,t^+]$.
\end{lemma}

\begin{proof}
By the Fundamental Theorem of Calculus and the uniform bound on the derivative of $\sigma$,
\begin{eqnarray}
0&\leq&\int^{t^{+}}_{t^{-}}\left|\frac{d}{dt}\sigma(t)-\frac{\sigma(t^{+})-\sigma(t^{-})}{t^{+}-t^{-}}\right|^2\,dt\nonumber\\
&=&\int^{t^{+}}_{t^{-}}\left|\frac{d}{dt}\sigma(t)\right|^2+\left|\frac{\sigma(t^{+})-\sigma(t^{-})}{t^{+}-t^{-}}\right|^2-2\frac{d}{dt}\sigma(t)\cdot\frac{\sigma(t^{+})-
\sigma(t^{-})}{t^{+}-t^{-}}\,dt\nonumber\\
&=&\int^{t^{+}}_{t^{-}}\left|\frac{d}{dt}\sigma(t)\right|^2 - \left|\frac{\sigma(t^{+})-\sigma(t^{-})}{t^{+}-t^{-}}\right|^2\,dt\leq0.\nonumber
\end{eqnarray}
That is $\frac{d}{dt}\sigma(t)=(\sigma(t^{+})-\sigma(t^{-}))/(t^{+}-t^{-})$ and the result follows.\qed
\end{proof}

We may now apply Lemma~\ref{lemma:rigidity} to get the following result:

\begin{lemma}
Let $\mu\in\mathcal{B}_{f,e}$ with $\mathrm{supp}\,\mu_{x}\subset\,K$ a.e.~in $B_{f,e}$ and let $z\in\,W^{1,\infty}(B_{f,e},\mathbb{R}^{3})$ be such that 
$Dz(x)=\bar{\mu}_{x}$ a.e.~in $B_{f,e}$. In addition, let $e\in\mathcal{M}_{s}$ and suppose that $D\subset B_{f,e}$ is an open set with the property that 
for each $x\in D$ the line segment parametrized by $r_x(t)=x+te$, $t\in [t^{-}_{x},t^{+}_{x}]$, where
\begin{eqnarray*}
 t^{-}_{x}&=&\sup\left\{t<0: x+te\notin B_{f,e}\right\}\\
 t^{+}_{x}&=&\inf\left\{t>0: x+te\notin B_{f,e}\right\},
\end{eqnarray*}
lies in $D$ for all $t\in (t^{-}_{x},t^{+}_{x})$ and $x+t^{\pm}_{x}e\in\partial B\cap\partial B_{f,e}$, the prescribed part of the boundary of $B_{f,e}$. 
Then $z(x)=U_sx$ for a.e. $x\in D$.
\label{lemmamaximaldirectionsus}
\end{lemma}

\begin{proof}
For $x\in D$, let $\sigma_x(t)=z(r_x(t))$.
Since $z\in W^{1,\infty}$, by Morrey \cite[Theorem 3.1.2, Lemma 3.1.1]{morreybook},  $z(r_x(t))$ is, for a.e. $x\in D$, absolutely continuous in $t$ on 
$(t_x^-,t_x^+)$, $z(r_x(t))$ tends to limits as $t\to (t_x^-)^+, t\to (t_x^+)^-$ and, for a.e.~$t\in(t^{-}_{x},t^{+}_{x})$,
\begin{equation}
\label{eq:sigmaderivative}
\frac{d}{dt}\sigma_x(t)=\frac{d}{dt}z(r_x(t))= Dz(r_x(t))e.
\end{equation}
Note that for any fixed $e\in\mathbb{R}^{3}$ the function $F\mapsto \vert Fe\vert$ is a convex function of the matrix $F$. But 
$\bar{\mu}_{x}=Dz(x)\in\,K^{qc}$ a.e.~and, by Fubini's theorem, a.e.~point of a.e.~line segment parallel to the vector $e$ belongs to the subset of 
$B_{f,e}$ where $Dz(x)\in K^{qc}$. That is, for a.e.~$x\in\,D$ and a.e.~$t\in(t^{-}_{x},t^{+}_{x})$,
\begin{equation}
\vert Dz(r_x(t))e\vert\leq\max_{i}\left\lbrace \vert U_{i}e\vert, 1\right\rbrace =\vert U_{s}e\vert,\nonumber
\end{equation}
since $e\in S^{2}$ is a maximal direction for $U_{s}$. Combining with (\ref{eq:sigmaderivative}) we deduce that for a.e.~$x\in D$, and a.e. $t\in (t_x^-,t_x^+)$,
\begin{equation}
\left|\frac{d}{dt}\sigma_x(t)\right|=\vert Dz(r_x(t))e\vert\leq\vert U_{s}e\vert=\frac{\vert\sigma_x(t^{+}_x)-\sigma_x(t^{-}_x)\vert}{\vert t^{+}_x-t^{-}_x\vert}.
\end{equation}
Applying Lemma~\ref{lemma:rigidity}, we infer that for a.e.~$x\in D$ and all $t\in(t^{-}_{x},t^{+}_{x})$,
\[z(r_x(t))=z(r_x(t^{-}_x))+(t-t^{-}_x)\frac{z(r_x(t^{+}_x))-z(r_x(t^{-}_x))}{t^{+}_x-t^{-}_x}=U_sr_x(t),\]
so that passing to the limit $t\to 0$ we get $z(x)=U_sx$ as required.\qed
\end{proof}

The idea is now simple; let $e\in\mathcal{M}_{s}$ and suppose that we may take $D=B_{f,e}$, i.e.~the entire set $B_{f,e}$ can be covered by line segments 
in the direction of $e$ joining points on the prescribed boundary. Then, by Lemma~\ref{lemmamaximaldirectionsus}, we immediately deduce that for any 
$\mu\in\mathcal{B}_{f,e}$ such that $\mathrm{supp}\,\mu_x\subset K$ a.e.,
\[\bar{\mu}_{x}=U_{s}\quad\mbox{for a.e.~$x\in B_{f,e}$}\] 
and using Lemma~\ref{lemmabarnu=us} we infer the quasiconvexity of $W$ at $U_{s}$ on a face or an edge. In particular, we can prove the following:

\begin{lemma}
Let $\Omega\subset\mathbb{R}^{3}$ be a bounded convex polyhedral domain and suppose that $e\in\mathcal{M}_{s}$ is a vector parallel to an edge or perpendicular to a normal 
of a face of $\Omega$. Then $W$ is quasiconvex at $U_{s}$ on that edge or face. In particular, if $x_0$ belongs to the edge or face, then $I(\nu)\geq I(\delta_{U_s})$ for every 
$\nu\in\mathcal{A}(x_0)$.
\label{finalpropmaxus}
\end{lemma}

\begin{proof}
Let $\mu\in\mathcal{B}_{f,e}$ with $\mathrm{supp}\,\mu_x\subset\,K$ a.e.~in $B_{f,e}$. It suffices to show that if the normal to a face is perpendicular to, or an edge is 
parallel to, $e\in\mathcal{M}_{s}$ then any respective region $B_{f,e}$ can be covered by line segments in the direction of $e\in\mathcal{M}_{s}$, 
joining points on the prescribed boundary $\partial B\cap\partial B_{f,e}$.

Since it is open and convex, $B_{f,e}$ can be covered by lines parallel to any direction joining points on its boundary. Hence, it is a question of making sure that, 
if these are in the direction of $e$, they never intersect the free boundary $B\cap\partial B_{f,e}$.

For the case of a face, we have that $B_f=B\cap\{x:x\cdot n<0\}$. If  $x_0\in B_f$ and  $x_0+t_0e\in B\cap \partial B_f$ for some $t_0$, then $(x_0+t_0e)\cdot n=0$, so that $x_0\cdot n=0$, a contradiction. 
Similarly, for the case of an edge we have that $B_e=B\cap\{x:x\cdot n_1<0, x\cdot n_2<0\}$. So if $x_0\in B_e$ and $x_0+t_0e\in B\cap\partial B_e$ for some $t_0$, then for $i=1$ or $i=2$ 
we have that $(x_0+t_0e)\cdot n_i=0$, and hence $x_0\cdot n_i=0$, a contradiction.\qed
\end{proof}

Next, we employ the (C-N) constraint and the assumptions on the lattice parameters and we follow a similar method as for the maximal directions for $U_{s}$, this time 
basing our argument on the inverse deformation and maximal directions for $U^{-1}_{s}$. 
This method allows us to add more directions along which the underlying deformations of measures in $\mathcal{B}_{f,e}$ agree with $U_{s}x$; indeed, for 
the cubic-to-orthorhombic transition of CuAlNi, we will see in the following section that all added directions are in fact new.

In Theorem~\ref{theorem:homeomorphic} we established that for $x_0\in\bar\Omega$ measures in $\mathcal{A}(x_0)$ whose support is contained in $K$ have underlying deformations 
that are homeomorphic and remarked that this property is naturally inherited by measures in $\mathcal{B}_{f,e}$. We now restrict attention to the deformed 
configuration $z(B_{f,e})$ and the maps $z^{-1}:z(B_{f,e})\rightarrow B_{f,e}$. The underlying idea remains the rigidity argument of 
Lemma~\ref{lemma:rigidity} but this time applied in a more elaborate manner.

In applying our rigidity argument, we consider an open set in the deformed configuration, covered by line segments along directions in 
$U^{-2}_{s}\mathcal{M}^{-1}_{s}$ joining points on $z(\partial B\cap\partial B_{f,e})$. We claim that almost all of these lines necessarily deform linearly 
under the inverse map $z^{-1}$ and according to $U^{-1}_{s}$; then, returning to the forward deformation, we can establish that the pre-image of these 
segments deforms under $z$ as $U_{s}x$.

Nevertheless, as Fig.~\ref{fig:badsituation} below suggests for the case of a face (similarly for an edge), we need to make sure that such an open set can 
actually `fit' in the deformed region $z(B_{f,e})$; this is possible but, for now, we assume that this is indeed the case and we return to it after proving 
the result analogous to Lemma~\ref{lemmamaximaldirectionsus} concerning the maximal directions for $U^{-1}_{s}$.

\begin{figure}[ht]
	\centering
	\def\svgwidth{0.6\columnwidth}	
	\begingroup
    \setlength{\unitlength}{\svgwidth}
  \begin{picture}(1,0.3)%
    \put(0,0){\includegraphics[width=\unitlength]{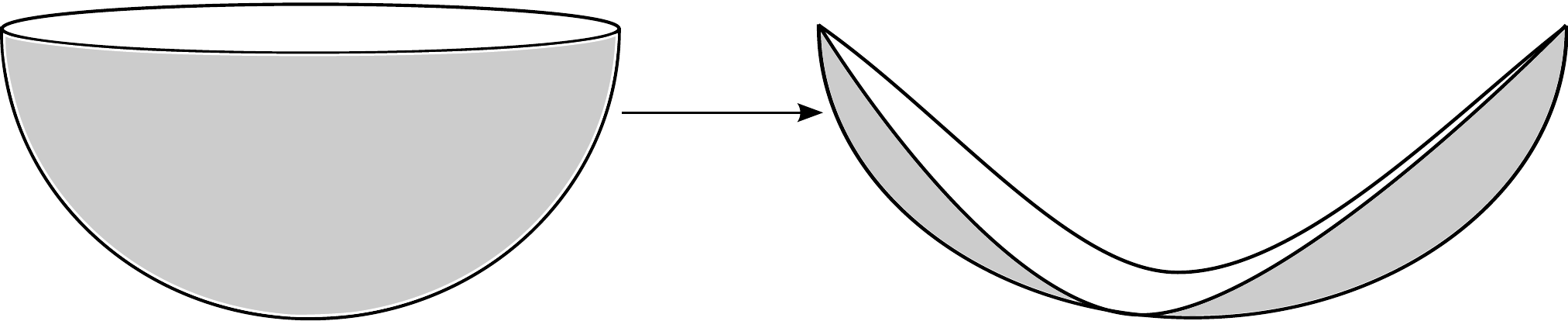}}%
    \put(0.2,0.22){\color[rgb]{0,0,0}\makebox(0,0)[lb]{\smash{$B_{f}$}}}%
    \put(0.68,0.22){\color[rgb]{0,0,0}\makebox(0,0)[lb]{\smash{$z(B_{f})$}}}%
    \put(0.44,0.18){\color[rgb]{0,0,0}\makebox(0,0)[lb]{\smash{$z$}}}%
  \end{picture}%
\endgroup
	\caption{Depiction of a deformation at a face; in this case, the method of maximal directions for $U^{-1}_{s}$ is not applicable.}
	\label{fig:badsituation}
\end{figure}
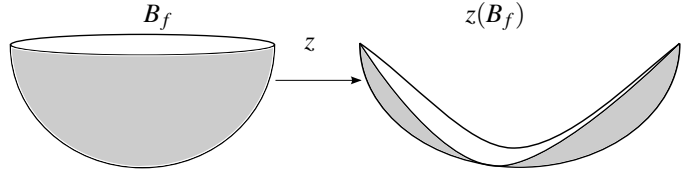

\begin{lemma}
Let $\mu\in\mathcal{B}_{f,e}$ with $\mathrm{supp}\,\mu_{x}\subset K$ a.e.~in $B_{f,e}$ and let $z\in\,W^{1,\infty}(B_{f,e},\mathbb{R}^{3})$ be such that 
$Dz(x)=\bar{\mu}_{x}$ a.e.~in $B_{f,e}$. In addition, let $e\in U^{-2}_{s}\mathcal{M}^{-1}_{s}$ and suppose that $D\subset B_{f,e}$ is an open set with 
the property that for each $x\in D$ the line segment parametrized by $r_x(t)=x+te$, $t\in [t^{-}_{x},t^{+}_{x}]$, where
\begin{eqnarray*}
 t^{-}_{x}&=&\sup\left\{t<0: x+te\notin B_{f,e}\right\}\\
 t^{+}_{x}&=&\inf\left\{t>0: x+te\notin B_{f,e}\right\},
\end{eqnarray*}
lies in $D$ for all $t\in (t^{-}_{x},t^{+}_{x})$ and $x+t^{\pm}_{x}e\in\partial B\cap\partial B_{f,e}$, the prescribed part of the boundary of $B_{f,e}$. 
Assume further that $U_sD\subset z(B_{f,e})$. Then $z(x)=U_sx$ for a.e. $x\in D$.
\label{lemmamaximaldirectionsusinverse}
\end{lemma}

\begin{proof}
For $x\in D$, write $\rho_x(t)=U_{s}r_x(t)=U_sx+tU_se$. 
Let $w:z(B_{f,e})\rightarrow B_{f,e}$ be the inverse of $z$. We have remarked that maps underlying measures in $\mathcal{B}_{f,e}$ inherit the properties of maps 
underlying measures in $\mathcal{A}(x_0)$ so that, by Theorem~\ref{theorem:homeomorphic}, $w\in W^{1,\infty}(z(B_{f,e}),\mathbb{R}^3)$ is differentiable a.e.~in 
$z(B_{f,e})$ and $$Dw(v)=\left[Dz(w(v))\right]^{-1}$$ for a.e.~$v\in z(B_{f,e})$.

Also, for every $x\in D$, $r_x(t)\in D$ for all $t\in(t^{-}_{x},t^{+}_{x})$ and therefore, for $t\in(t^{-}_{x},t^{+}_{x})$, 
$\rho_x(t)\in U_sD\subset z(B_{f,e})$ by assumption; also, for each $x\in D$, the line segments parametrized by $\rho_x$ are all in the direction of the 
vector $U_se$. Hence, as in Lemma~\ref{lemmamaximaldirectionsus}, for a.e.~$x\in D$, $w(\rho_x(t))$ is absolutely continuous on $(t_x^-,t_x^+)$, tends to limits as $t\to t_x^-$ and $t\to t_x^+$, 
and $\frac{d}{dt}w(\rho_x(t))=Dw(\rho_x(t))U_se$.

Lastly, we note that $\bar{\mu}_x=Dz\left(x\right)\in K^{qc}$ a.e.~and hence, for a.e.~$x\in D$ and a.e.~$t\in(t^{-}_{x},t^{+}_{x})$,
$$Dz(w(\rho_x(t)))\in K^{qc}.$$
To see this, note that the set $N\subset B_{f,e}$ such that $Dz(x)\notin K^{qc}$ satisfies $\mathcal{L}^{3}(N)=0$; since $z\in W^{1,\infty}(B_{f,e},\mathbb{R}^3)$ 
and $B_{f,e}\subset\mathbb{R}^3$ is open and bounded, $z$ maps sets of measure zero to sets of measure zero (see~\cite{marcusmizel}), 
i.e.~$\mathcal{L}^3(z(N))=0$. Then, a.e.~point of a.e.~line in the direction of $U_se$ in the deformed configuration does not belong 
to the set $z(N)$ and their pre-images do not belong to $N$.

Now assume that $e=e_{\max}(\mathrm{cof}\,U_s)$ and let
$$\sigma_x(t)=w(\rho_x(t)).$$
The function $F\mapsto\vert(\mathrm{cof}\,F)^{T}e\vert=\vert(\mathrm{adj}\,F)e\vert$ is polyconvex. Since all the $U_{i}$'s are symmetric, $U_se\parallel e$ and, for 
a.e.~$x\in D$ and a.e.~$t\in(t^{-}_{x},t^{+}_{x})$, $Dz(w(\rho_x(t)))\in\,K^{qc}$, we infer that
\begin{eqnarray}
\vert\frac{d}{dt}\sigma_x(t)\vert&=&\vert Dw(\rho_x(t))U_se\vert\nonumber\\
&=&\frac{1}{\det Dz(w(\rho_x(t))}\vert(\mathrm{adj}\;Dz(w(\rho_x(t)))U_se\vert\nonumber\\
&\leq&\frac{1}{\det U_s}\max_{i, R\in SO(3)}\lbrace\vert(\mathrm{cof}\;U_{i}R)U_se\vert, \vert U_se\vert\rbrace\nonumber\\
&=&\frac{1}{\det U_s}\lambda_{\max}(\mathrm{cof}\,U_s)\vert U_se\vert\nonumber\\
&=&\frac{1}{\det U_s}\vert(\mathrm{cof}\,U_s)U_se\vert=1=\frac{\vert\sigma_x(t^{+}_{x})-\sigma_x(t^{-}_{x})\vert}{\vert t^{+}_{x}-t^{-}_{x}\vert}.
\label{eq:lemmaauxilliary1}
\end{eqnarray}
Above we have used the fact that the map $F\mapsto\det F$ is quasiaffine and, since $\det U_s\leq1$, $\det U_{s}\leq\det Dz(w(\rho_x(t)))\leq1$. Now 
Lemma~\ref{lemma:rigidity} applies to give that $$w(\rho_x(t))=w(\rho_x(t^{-}_{x}))+(t-t^{-}_{x})\frac{w(\rho_x(t^{+}_{x}))-w(\rho_x(t^{-}_{x}))}{t^{+}_{x}-t^{-}_{x}}
=r_x(t),$$
which implies, setting $t=0$, that $w(U_sx)=x$ and hence $z(x)=U_sx$ for a.e. $x$.

Next, assume that $e\in U^{-2}_{s}\mathcal{M}^{-1}_{s}$, $e\neq e_{\max}(\mathrm{cof}\,U_s)$, and let
$$\tilde{\sigma}_x(t)=U^{2}_{s}e\cdot w(\rho_x(t)).$$
Then, for a.e.~$x\in D$ and a.e.~$t\in(t^{-}_{x},t^{+}_{x})$,
\begin{equation}
\label{eq:sigma2derivative}
\frac{d}{dt}\tilde{\sigma}_x(t)=\frac{d}{dt}U^{2}_{s}e\cdot w(\rho_x(t))=U^{2}_{s}e\cdot Dw(\rho_x(t))U_se=U^{2}_{s}e\cdot [Dz(w(\rho_x(t)))]^{-1}U_se.
\end{equation}
However, the function $F\mapsto\vert(\mathrm{cof}\,F)U^{2}_{s}e\cdot U_se\vert$ is polyconvex and since, for a.e.~$x\in D$ and a.e.~$t\in(t^{-}_{x},t^{+}_{x})$, 
$Dz(w(\rho_x(t)))\in\,K^{qc}$, we infer that
\begin{eqnarray}
\label{eq:arrayproof}
\vert(\mathrm{cof}\,Dz(w(\rho_x(t))))U^{2}_{s}e\cdot U_se\vert &\leq&\max_{i,R\in SO(3)}\left\{R(\mathrm{cof}\,U_i)U^{2}_{s}e\cdot U_se, RU^{2}_{s}e\cdot 
U_se\right\}\nonumber\\
&\leq&\max_{i}\left\{\vert(\mathrm{cof}\,U_i)U^{2}_{s}e\vert\vert U_se\vert, \vert U^{2}_{s}e\vert\vert U_se\vert\right\}\nonumber\\
&=&\vert(\mathrm{cof}\,U_s)U^{2}_{s}e\vert\vert U_se\vert= \det\,U_s \vert U_{s}e\vert^2
\end{eqnarray}
since $e\in U^{-2}_{s}\mathcal{M}^{-1}_{s}$ and $U_s$ is symmetric. As before, $F\mapsto\det F$ is quasiaffine and $\det U_s\leq1$, hence 
\[\vert U^{2}_{s}e\cdot[Dz(w(\rho_x(t)))]^{-1}U_se\vert\leq \vert U_{s}e\vert^2.\]
Combining with (\ref{eq:sigma2derivative}) we deduce that for a.e.~$x\in D$ and a.e.~$t\in(t^{-}_{x},t^{+}_{x})$
\begin{equation}
\left|\frac{d}{dt}\tilde{\sigma}_x(t)\right|\leq\vert U_{s}e\vert^2=\vert U^{2}_{s}e\cdot e\vert=\frac{\vert\tilde{\sigma}_x(t^{+}_{x})-\tilde{\sigma}_x(t^{-}_{x})\vert}{\vert t^{+}_{x}-t^{-}_{x}
\vert}
\end{equation}
and applying the rigidity argument of Lemma~\ref{lemma:rigidity}, we infer that for a.e.~$x\in D$ and all $t\in(t^{-}_{x},t^{+}_{x})$,
\begin{equation}
\label{eq:detDw=DetUs}
U^{2}_{s}e\cdot Dw(\rho_x(t))U_se=\frac{U^{2}_{s}e\cdot w(\rho_x(t^{+}_{x}))-U^{2}_{s}e\cdot w(\rho_x(t^{-}_{x}))}{t^{+}_{x}-t^{-}_{x}}=\vert U_se\vert^2.
\end{equation}
To finish the proof, note that by (\ref{eq:arrayproof})
\[\vert\mathrm{cof}\,Dz(w(\rho_x(t)))U^{2}_{s}e\cdot U_{s}e\vert\leq\det U_s\vert U_{s}e\vert^2.\]
However, (\ref{eq:detDw=DetUs}) says that
\[\vert\mathrm{cof}\,Dz(w(\rho_x(t)))U^{2}_{s}e\cdot U_{s}e\vert=\det Dz(w(\rho_x(t)))\vert U_{s}e\vert^2\geq\det U_s\vert U_{s}e\vert^2\]
by our assumption that $\det U_s\leq 1$. But then $\det Dz(w(\rho_x(t)))=\det U_s$ and
\[\vert(\mathrm{cof}\,Dz(w(\rho_x(t))))U^{2}_{s}e\cdot U_{s}e\vert=\det U_s\vert U_{s}e\vert^2=\vert(\mathrm{cof}\,U_s)U^{2}_{s}e\cdot U_{s}e\vert.\]
Letting $\psi:\mathbb{R}^{3\times3}\rightarrow\mathbb{R}$ be the polyconvex function $\psi(F)=\vert(\mathrm{cof}\,F)U^{2}_{s}e\cdot U_{s}e\vert$ and using 
the fact that the measure $\mu=(\mu_x)_{x\in B_{f,e}}$ underlying the deformation $z$ is a $W^{1,\infty}$ gradient Young measure, we deduce that for 
a.e.~$x\in D$ and $t\in(t^{-}_{x},t^{+}_{x})$,
\begin{eqnarray}
\vert(\mathrm{cof}\,U_s)U^{2}_{s}e\cdot U_{s}e\vert&=&\psi(Dz(w(\rho_x(t))))\nonumber\\
&\leq&\langle\mu_{w(\rho_x(t))},\psi\rangle\nonumber\\
&=&\int_{SO(3)}\vert AU^{2}_{s}e\cdot U_{s}e\vert\,d\mu_{w(\rho_x(t))}(A)\nonumber\\
&&+\sum_{i}\int_{SO(3)U_i}\vert(\mathrm{cof}\,A)U^{2}_{s}e\cdot U_{s}e\vert\,d\mu_{w(\rho_x(t))}(A)\nonumber\\
&\leq&\mu_{w(\rho_x(t))}(SO(3))\vert U^{2}_{s}e\vert\vert U_{s}e\vert+\nonumber\\
&&\sum_{i}\mu_{w(\rho_x(t))}(SO(3)U_i)\vert(\mathrm{cof}\,U_i)U^{2}_{s}e\vert\vert U_{s}e\vert.
\label{eq:revisedproof1}
\end{eqnarray}
On the other hand, since $(\mathrm{cof}\,U_s)U^{2}_{s}e$ is parallel to $U_{s}e$, we also deduce that
\begin{eqnarray}
 \vert(\mathrm{cof}\,U_s)U^{2}_{s}e\cdot U_{s}e\vert&=&\vert(\mathrm{cof}\,U_s)U^{2}_{s}e\vert\vert U_{s}e\vert\nonumber\\
&=&\mu_{w(\rho_x(t))}(SO(3))\vert (\mathrm{cof}\,U_s)U^{2}_{s}e\vert\vert U_{s}e\vert+\nonumber\\
&&\sum_{i}\mu_{w(\rho_x(t))}(SO(3)U_i)\vert(\mathrm{cof}\,U_s)U^{2}_{s}e\vert\vert U_{s}e\vert.
\label{eq:revisedproof2}
\end{eqnarray}
Subtracting equation (\ref{eq:revisedproof2}) from (\ref{eq:revisedproof1}), we obtain
\begin{eqnarray*}
 0&\leq&\mu_{w(\rho_x(t))}(SO(3))\left[\vert U^{2}_{s}e\vert\vert U_{s}e\vert-\vert(\mathrm{cof}\,U_s)U^{2}_{s}e\vert\vert U_{s}e\vert\right]+\\
&&\sum_{i\neq s}\mu_{w(\rho_x(t))}(SO(3)U_i)\left[\vert(\mathrm{cof}\,U_i)U^{2}_{s}e\vert\vert U_{s}e\vert-\vert(\mathrm{cof}\,U_s)U^{2}_{s}e\vert
\vert U_{s}e\vert\right].
\end{eqnarray*}
However, since $e\in U^{-2}_{s}\mathcal{M}^{-1}_{s}$, all terms in the brackets are strictly negative and hence, for a.e.~$x\in D$ and 
$t\in(t^{-}_{x},t^{+}_{x})$,
\[\mu_{w(\rho_x(t))}(SO(3))=\mu_{w(\rho_x(t))}(SO(3)U_i)=0\:\:\mbox{for all $i\neq s$},\]
implying that $\mathrm{supp}\,\mu_{w(\rho_x(t))}\subset SO(3)U_s$ for a.e.~$x\in D$ and $t\in(t^{-}_{x},t^{+}_{x})$. Given a connected component $\Delta$ of $D$, the set $w(U_s\Delta)$ is open and connected, 
and consists of the union of all the sets $\{w(\rho_x(t)):t\in(t_x^-,t_x^+)\}$ for $x\in \Delta$. The set of points $y\in U_s\Delta$ with $\mathrm{supp}\,\mu_{w(y)}\subset SO(3)U_s$ is of full measure, 
and therefore so is the set of points $p\in w(U_s\Delta)$ with $\mathrm{supp}\,\mu_p\subset SO(3)$, since $w$ maps sets of measure zero to sets of measure zero. By standard results (see e.g.~\cite{ballnorankone}) 
this implies that $\mu_p=\delta_{RU_s}$ for a.e. $p\in w(U_s\Delta)$, where $R\in SO(3)$ is constant (possibly depending on $\Delta$). Thus for a.e. $x\in \Delta$ and a.e. $t\in(t_x^-,t_x^+)$ we have that 
$Dz(w(\rho_x(t)))=RU_s$, and thus
\begin{eqnarray*}
w(\rho_x(t))&=&\int_{t_x^-}^tDw(\rho_x(s))U_se\,ds+w(\rho_x(t_x^-))\\
&=&U_s^{-1}R^TU_se(t-t_x^-)+x+t_x^-e.
\end{eqnarray*}
Setting $t=t_x^+$ we deduce that $U_s^{-1}R^TU_se=e$, so that setting $t=0$ we obtain $w(U_sx)=x$, and thus $z(x)=U_sx$. 
Applying this argument to each connected component of $D$ completes the proof.\qed
\end{proof}

The idea is now almost identical to that of maximal directions for $U_{s}$; if we are able to take $D=B_{f,e}$, i.e.~if we can cover $B_{f,e}$ by segments 
in the direction of $e\in U^{-2}_{s}\mathcal{M}^{-1}_{s}$ joining points on the prescribed boundary $\partial B\cap\partial B_{f,e}$, we can immediately deduce 
that $\bar{\mu}_{x}=U_{s}$ a.e.~and Lemma~\ref{lemmabarnu=us} will provide the required quasiconvexity condition. However, there is a crucial obstacle, that in 
the statement of Lemma~\ref{lemmamaximaldirectionsusinverse} we assumed that the image of the set $D$ under $U_{s}$ lies entirely in $z(B_{f,e})$. 

We proceed with two lemmata to show that an appropriate open set $D\subset B_{f,e}$ exists and that it can be chosen as the entire set $B_{f,e}$. Firstly, 
we wish to show that points in $B_{f,e}$, sufficiently close to the prescribed boundary $\partial B\cap\partial B_{f,e}$, have the property that their image 
under the linear map $U_sx$ lies in $z(B_{f,e})$ for any deformation $z$ underlying an admissible measure in $\mathcal{B}_{f,e}$ supported within $K$. 
This allow us to cover a small region in $B_{f,e}$, sufficiently close to $\partial B\cap\partial B_{f,e}$, by the appropriate line segments such that their 
image under $U_sx$ lies in $z(B_{f,e})$, i.e.~one appropriate set $D$ exists. This is described in Lemma~\ref{distancelemma2} below and, as its proof is 
applicable to both faces and edges, we treat them together. We note that Lemma~\ref{lemmamaximaldirectionsusinverse} will then say that this 
(possibly small) region necessarily transforms like $U_sx$ and the idea is that this process can be continued to exhaust $B_{f,e}$.

\begin{lemma}
Let $z$ be a map underlying a measure in $\mathcal{B}_{f,e}$ with support contained in $K$ a.e.~and suppose that $C\subset\partial B\cap\partial B_{f,e}$ is 
closed. There exists some $\epsilon>0$ such that whenever $x\in B_{f,e}$ satisfies $\mathrm{dist}(x,C)<\epsilon$ then $U_sx\in z(B_{f,e})$, where for any 
two sets $A,\,B\subset\mathbb{R}^{3}$, $\mbox{dist}(A,\,B)=\inf\lbrace\vert a-b\vert,\,a\in A,\,b\in B\rbrace$.

In particular, if $r:\left[t^{-},t^{+}\right]\rightarrow\overline{B_{f,e}}$ is the parametrization of a line segment such that $r(t^{-}),\,r(t^{+})\in\,C$ 
and for all $t\in(t^{-},t^{+})$, $r(t)\in B_{f,e}$ then
\[\left\{U_{s}r(t)\,:\,t\in(t^{-},t^{+})\right\}\subset z(B_{f,e})\]
whenever $\mathrm{dist}(r(t),C)<\epsilon$ for all $t\in\left[t^{-},t^{+}\right]$.
\label{distancelemma2}
\end{lemma}

\begin{proof}
Since $C\subset\partial B\cap\partial B_{f,e}$ is closed and the sets $\partial B\cap\partial B_{f,e}$, $B\cap\partial B_{f,e}$ are disjoint, there 
exists $\delta>0$ such that $\mathrm{dist}(C, B\cap\partial B_{f,e})=\delta$. Let $c_0\in C$; then 
$\mathrm{dist}(c_0,B\cap\partial B_{f,e})\geq\delta$. Restrict $z$ to $B(c_0,\delta)\cap B_{f,e}$ and extend it to 
$\tilde{z}:B(c_0,\delta)\rightarrow\mathbb{R}^3$ by $U_sx$, i.e.
$$\tilde{z}(x)=\left\{\begin{array}{rcl}z(x),&\:\:&x\in B(c_0,\delta)\cap B_{f,e}\\\:&\:\:&\:\: \\U_sx,&\:\:&x\in B(c_0,\delta)\setminus B_{f,e}.
\end{array}\right.$$
Note that $\tilde{z}$ is continuous, an element of $W^{1,\infty}(B(c_0,\delta),\mathbb{R}^3)$, and has bounded distortion. Thus, since $\tilde{z}$ is not 
identically equal to a constant, Lemma~\ref{lemmareshetnyak} says that $\tilde{z}$ is an open mapping.
Then, $\tilde{z}(B(c_0,\delta))$ is an open set and there exists $\epsilon_0$ such that $B(\tilde{z}(c_0),\epsilon_0)\subset\tilde{z}(B(c_0,\delta))$, i.e.
$$B(z(c_0),\epsilon_0)=B(U_sc_0,\epsilon_0)\subset\tilde{z}(B(c_0,\delta)).$$
We now claim that $U_s\left[B(c_0,\epsilon_0/\Vert U_s\Vert)\cap B_{f,e}\right]\subset z(B_{f,e})$ where $\Vert U_s\Vert=\sigma_{\max}(U_s)$; trivially, one obtains that 
$U_sB(c_0,\epsilon_0/\Vert U_s\Vert)\subset B(U_sc_0,\epsilon_0)$ and hence
$$U_s\left[B(c_0,\epsilon_0/\Vert U_s\Vert)\cap B_{f,e}\right]\subset\tilde{z}(B(c_0,\delta))=z(B(c_0,\delta)\cap B_{f,e})\cup U_s(B(c_0,\delta)\setminus 
B_{f,e}).$$
But $B(c_0,\epsilon_0/\Vert U_s\Vert)\cap B_{f,e}$ and $B(c_0,\delta)\setminus B_{f,e}$ are disjoint and, since $U_s$ is invertible, the linear map 
$U_sx$ is injective, i.e.
$$U_s\left[B(c_0,\epsilon_0/\Vert U_s\Vert)\cap B_{f,e}\right]\subset z(B(c_0,\delta)\cap B_{f,e})\subset z(B_{f,e}).$$

To conclude the proof, it suffices to find an $\epsilon$ such that $U_s\left[B(c,\epsilon)\cap B_{f,e}\right]\subset z(B_{f,e})$ for all $c\in C$; by the 
above argument, for any $c\in C$ there exists $\epsilon_c>0$ such that
$$U_s\left[B(c,\epsilon_c)\cap B_{f,e}\right]\subset z(B_{f,e})$$
and the family of sets $\left\{B(c,\epsilon_c)\right\}_{c\in C}$ is an open cover of $C$. Since $C$ is closed, it is also compact and there exist 
$c_i\in C$, $i=1,\ldots,N$ such that the sets $\left\{B(c_i,\epsilon_{c_i})\right\}_{i=1,\ldots,N}$ cover $C$; let
$${\displaystyle\epsilon=\mathrm{dist}([\bigcup^{N}_{i=1}B(c_i,\epsilon_{c_i})]^c,C)},$$
where for a set $A\subset\mathbb{R}^3$, $A^c$ denotes its complement in $\mathbb{R}^3$. Then $\epsilon>0$ and for any $\tilde{x}\in[\bigcup^{N}_{i=1}B(c_i,\epsilon_{c_i})]^c$ 
and $c\in C$ we have $\vert\tilde{x}-c\vert\geq\epsilon$. 
Hence, if $x\in B_{f,e}$ satisfies $\mathrm{dist}(x,C)<\epsilon$, $x\in(\bigcup^{N}_{i=1}B(c_i,\epsilon_{c_i}))\cap B_{f,e}$ and $U_sx\in z(B_{f,e})$.\qed
\end{proof}

\begin{remark}
\label{remark:selfcontact}
We note that the above result excludes a certain kind of `one-sided' contact between the images under $z$ of $\partial B\cap\partial B_{f,e}$ and $B\cap\partial B_{f,e}$, 
the prescribed and free parts of the boundary of $B_{f,e}$ respectively. In particular, it says that for any $x_0\in\partial B\cap\partial B_{f,e}$ there 
exists an $\epsilon_0>0$ such that the set $U_s\left[B(x_0,\epsilon_0)\cap B_{f,e}\right]$ lies in $z(B_{f,e})$ and hence does not intersect the image of 
the free boundary $z(B\cap\partial B_{f,e})$. By abusing terminology, $U_s\left[B(x_0,\epsilon_0)\cap B_{f,e}\right]$ can be thought of as a 
`neighbourhood' of $z(x_0)$ lying entirely within $z(B_{f,e})$.

Moreover let $r(t)$, $t\in[0,a]$, be the parametrization of a line segment in the direction of $e$ with endpoints on $\partial B\cap\partial B_{f,e}$ such that $r(t)\in B_{f,e}$ for 
any $t\in(0,a)$ and suppose that
\[
\left\{U_sr(t): t\in(0,a)\right\}\subset z(B_{f,e}).
\]
For any $t_0\in(0,a)$, $U_sr(t_0)\in z(B_{f,e})$ which is open, i.e.~there exists $\epsilon_0>0$ such that $B(U_sr(t_0),\epsilon_0)\subset z(B_{f.e})$ and 
hence
\[
U_s[B(r(t_0),\frac{\epsilon_0}{\Vert U_s\Vert})]\subset z(B_{f,e}).
\]
Treating endpoints, or generally boundary points, as in the proof of Lemma~\ref{distancelemma2} and repeating the argument, one can find $\epsilon>0$ such 
that whenever $x\in B_{f,e}$ and $\mathrm{dist}(x,r([0,a]))<\epsilon$, then $U_sx\in z(B_{f,e})$. In particular, whenever $\rho(t)$, $t\in[0,b]$ is another line segment in the direction of $e$ with endpoints 
on $\partial B\cap\partial B_{f,e}$ with the property that
\[
\mathrm{dist}(\rho([0,b]),r([0,a])<\epsilon,
\]
then the set $\left\{U_s\rho((0,b))\right\}$ lies in $z(B_{f,e})$; we note this as it will be used shortly in the proof of Lemma~\ref{lemmafefinal}.
\end{remark}

The following lemma shows that we may choose $D=B_{f,e}$ and resolves the case of faces perpendicular to, and edges parallel to, directions in 
$U^{-2}_{s}\mathcal{M}^{-1}_{s}$.

\begin{lemma}
Let $\Omega\subset\mathbb{R}^{3}$ be a bounded convex polyhedral domain and suppose that the vector $e\in U^{-2}_{s}\mathcal{M}^{-1}_{s}$ is parallel to an edge or 
perpendicular to the normal of a face of $\Omega$. Let $z\in W^{1,\infty}(B_{f,e},\mathbb{R}^{3})$ be a map underlying a Young measure in 
$\mathcal{B}_{f,e}$ with support contained in $K$ a.e.~in $B_{f,e}$. Then $z(x)=U_{s}x$ for all $x\in B_{f,e}$.
\label{lemmafefinal}
\end{lemma}

\begin{proof}
Let $\mu\in\mathcal{B}_{f,e}$ with support a.e.~in $K$ and suppose that $z\in\,W^{1,\infty}(B_{f,e},\mathbb{R}^{3})$ is its underlying deformation. As in 
Lemma~\ref{finalpropmaxus}, we may cover $B_{f,e}$ with line segments in the direction of $e\in\,U^{-2}_{s}\mathcal{M}^{-1}_{s}$ whose endpoints lie on 
$\partial B\cap\partial B_{f,e}$. In particular, for any $x\in B_{f,e}$ let $t_x^-,t_x^+$ be as defined previously and $r_x(t)=x+te$, $t\in[t^{-}_{x},t^{+}_{x}]$ be a parametrization of the 
line segment through $x$, in the direction of $e$, with endpoints on $\partial B\cap\partial B_{f,e}$.

Let $E$ be the maximal subset of $B_{f,e}$ with the property that for every $x\in E$, $r_x(t)\in E$ for all $t\in(t^{-}_{x},t^{+}_{x})$ and 
$\left\{U_sr_x(t):t\in(t^{-}_{x},t^{+}_{x})\right\}\subset z(B_{f,e})$, i.e.~$E$ is the maximal subset of $B_{f,e}$ that can be covered by line segments in 
the direction of $e$, joining points on $\partial B\cap\partial B_{f,e}$ such that their image under the linear map $x\mapsto U_sx$ lies in $z(B_{f,e})$. Such 
a maximal set clearly exists.

We aim to show that $E$ is a non-empty, open and closed subset of $B_{f,e}$, where open and closed is meant in terms of the induced topology on $B_{f,e}$ as 
a subset of $\mathbb{R}^3$. Then, $B_{f,e}$ being connected, $E=B_{f,e}$; by Lemma~\ref{lemmamaximaldirectionsusinverse} and the continuity of $z$ we then 
deduce that $z(x)=U_sx$ for all $x\in B_{f,e}$.

Firstly, we can choose an appropriate closed set $C$ so that we can apply Lemma~\ref{distancelemma2}. For a face, we can choose $C=\{x\in\partial B:x\cdot n\leq -1+\tau\}$, 
and for an edge $C=\{x\in\partial B:x\cdot\frac{n_1+n_2}{|n_1+n_2|}\leq -1+\tau\}$, where $\tau>0$ is sufficiently small. Then Lemma~\ref{distancelemma2} says that there exists $\epsilon>0$ such 
that any line segment in the direction of $e$ with endpoints on $C$, say parametrized by $r(t)$, $t\in\left[t^{-},t^{+}\right]$, with 
$\mathrm{dist}(r(t),\,C)<\epsilon$ for all $t\in\left[t^{-},t^{+}\right]$ satisfies
\[\left\{U_{s}r(t)\,:\,t\in(t^{-},t^{+})\right\}\subset z(B_{f,e}).\]
In particular, $\left\{r(t): t\in(t^{-},t^{+})\right\}\subset E$ and $E$ is non-empty.

Let $x_0\in E$ and $r_{x_0}(t)$, $t\in[t^{-}_{x_0},t^{+}_{x_0}]$, be the parametrization of the corresponding line segment. By Remark~\ref{remark:selfcontact} following 
Lemma~\ref{distancelemma2}, there exists yet another $\epsilon>0$ such that whenever $\rho(t)$, $t\in[t^{-},t^{+}]$, is a line segment in the direction of 
$e$ with endpoints on $\partial B\cap\partial B_{f,e}$ with the property that
\[
\mathrm{dist}(\rho([t^{-},t^{+}]),r_{x_0}([t^{-}_{x_0},t^{+}_{x_0}]))<\epsilon,
\]
then $\left\{U_s\rho(t): t\in (t^{-},t^{+})\right\}\subset z(B_{f,e})$. In particular, this gives a neighbourhood of 
$\left\{r_{x_0}(t):t\in(t^{-}_{x_0},t^{+}_{x_0})\right\}$, say $D$, which is open in $B_{f,e}$ such that for any point $\tilde{x}\in D$, the image under 
$U_sx$ of the corresponding line segment parametrized by $r_{\tilde{x}}(t)$, $t\in (t^{-}_{\tilde{x}},t^{+}_{\tilde{x}})$ lies in $z(B_{f,e})$. 
But $E$ is the maximal subset of $B_{f,e}$ with this property and hence $D\subset E$. Trivially, \[x_0=r_{x_0}(0)\in D\subset E\] and $E$ is therefore open.

We note that, since $E$ is open and can be covered by line segments in the direction of $e$ with endpoints on $\partial B\cap\partial B_{f,e}$ 
whose image under $U_sx$ lies in $z(B_{f,e})$, by Lemma~\ref{lemmamaximaldirectionsusinverse} and the continuity of $z$, $z(x)=U_sx$ for all 
$x\in\overline{E}$. Note that, since $z$ is continuous up to the boundary of $B_{f,e}$, this is also true if $\overline{E}$ is viewed as the closure of 
$E$ in $\mathbb{R}^3$.

It is now easy to infer that $E$ is also closed in $B_{f,e}$. Let $x_k\in E$ be a sequence of points such that $x_k\rightarrow x\in B_{f,e}$. Let $r_{x}(t)$, 
$t\in[t^{-}_{x},t^{+}_{x}]$ be the parametrization of the line segment through $x$ and similarly $r_{x_k}(t)$ for the points $x_k$. The idea is to show 
that $z(r_x(t))=U_sr_{x}(t)$ for all $t\in(t^{-}_{x},t^{+}_{x})$; then surely $U_sr_{x}(t)\in z(B_{f,e})$ for all $t\in(t^{-}_{x},t^{+}_{x})$ and, by the 
maximality of $E$, $x\in E$.

Let $t_0\in(t^{-}_{x},t^{+}_{x})$ and consider the point $r_x(t_0)=x+t_0e$. Since $x_k+t_0e\to x+t_0e\in B_{f,e}$, it follows that $t_0\in(t_{x_k}^-,t_{x_k}^+)$ for $k$ sufficiently large 
and $r_{x_k}(t_0)\to r_x(t_0)$. Thus 
\[
z(r_x(t_0))=\lim_{k\to\infty}z(r_{x_k}(t_0))=\lim_{k\to\infty}U_sr_{x_k}(t_0)=U_sr_x(t_0).
\]
Hence $x\in E$ and $E$ is closed.\qed
\end{proof}

We can now easily deduce Lemma~\ref{prop:feqc} regarding the quasiconvexity of $W$ at $U_{s}$ at any face or edge of a domain admissible for $U_s$ as well as establish 
our main result, Theorem~\ref{theorem:main}.

\begin{proof}\hspace{-0.15cm} {\it of Lemma}~\ref{prop:feqc}
Suppose that $\mu\in\mathcal{B}_{f,e}$ is such that $\mathrm{supp}\,\mu_{x}\subset\,K$ a.e., otherwise the quasiconvexity follows trivially. 
Combining Lemmas ~\ref{lemmabarnu=us}, \ref{finalpropmaxus} and \ref{lemmafefinal} we deduce that, whenever $\Omega$ is admissible for $U_s$, $W$ is quasiconvex at $U_s$ on faces and edges.\qed
\end{proof}

\begin{proof}\hspace{-0.15cm} {\it of Theorem}~\ref{theorem:main}
The second part of the Theorem was established in Lemma~\ref{prop:corner}. As for the first part, it suffices to show that $W$ is quasiconvex at $U_{s}$ in the interior, 
on faces and edges; but this is immediate by Lemma~\ref{prop:interior} for the interior and Lemma~\ref{prop:feqc} for faces and edges.\qed
\end{proof}

In particular, we have shown that if the specimen is cut in such a way that its edges are parallel to vectors in $\mathcal{M}_{s}\cup U^{-2}_{s}\mathcal{M}^{-1}_{s}$, the 
austenite cannot nucleate anywhere in the interior, on faces or edges and we have shown how nucleation can occur at a corner for Seiner's specimen. We now turn our 
attention to the cubic-to-orthorhombic transition of CuAlNi, calculate the sets of maximal directions for each $s=1,\ldots\,,6$ and infer whether our simplified model and 
main result can indeed provide an explanation for the location of the nucleation points in Seiner's experiment.

\section{Maximal directions}
\label{sec:4.1}

In this section we identify the maximal directions for $U_{s}$ and $U^{-1}_{s}$ for the cubic-to-orthorhombic variants given in (\ref{eq:orthorhombicvariants}). These 
directions are dependent on the lattice parameters and, therefore, further assumptions must be made.

\paragraph{Assumptions on lattice parameters:}
\begin{itemize}
\item [(A1)] $\beta\leq 1\leq\gamma<\alpha$;
\item [(A2)] $2\alpha^{2}+\beta^{2}\geq 3$; equivalently $\vert U_{s}\vert ^{2}\geq 3+\gamma^{2}-\alpha^{2}$;
\item [(A3)] $\alpha^{2}\gamma^{2}+\alpha^{2}\beta^{2}+\beta^{2}\gamma^{2}\geq 3$; equivalently $\vert\mbox{cof}\,U_{s}\vert ^{2}\geq 3$;
\item [(A4)] $A-B>0$ where $A=\alpha ^{2}\gamma ^{2}-\beta ^{2}\left(\alpha ^{2}+\gamma ^{2}\right)/2>0$ and $B=\beta ^{2}\left(\alpha ^{2}-\gamma ^{2}\right)>0$.
\end{itemize}
We note that these assumptions are in accordance with the CuAlNi specimen of the experiment where $\alpha=1.06372$, $\beta=0.91542$ and $\gamma=1.02368$. Then, 
$\vert U_{s}\vert ^{2}-\gamma^{2}+\alpha^{2}=3.10099$, $\vert\mbox{cof}\,U_{s}\vert ^{2}=3.01206$ and $A-B=0.202513$.

For the remainder of this section we prove a series of results concerning the maximal directions for $U_{s}$ and $U^{-1}_{s}$; we warn the reader that the proofs of these 
rely on simple, but often long, calculations.

\begin{lemma}
Assume that the lattice parameters satisfy {\rm (A1)} and {\rm (A2)} and for a vector $e\in S^{2}$ write $e=\left(e_{1},e_{2},e_{3}\right)^{T}$. Then, for each 
$s=1,\ldots ,6$,
\begin{eqnarray*}
\mathcal{M}_{s}&=&\left\lbrace e\in S^{2}: \left(-1\right)^{s-1}e_{2}e_{3}\geq 0,\:\:\vert e_{1}\vert\leq\min\left\lbrace\vert e_{2}\vert ,\vert e_{3}\vert\right\rbrace 
\right\rbrace,\,\,\mbox{for $s=1,2$}\\
\mathcal{M}_{s}&=&\left\lbrace e\in S^{2}: \left(-1\right)^{s-1}e_{1}e_{3}\geq 0,\:\:\vert e_{2}\vert\leq\min\left\lbrace\vert e_{1}\vert ,\vert e_{3}\vert\right\rbrace 
\right\rbrace,\,\,\mbox{for $s=3,4$}\\
\mathcal{M}_{s}&=&\left\lbrace e\in S^{2}: \left(-1\right)^{s-1}e_{1}e_{2}\geq 0,\:\:\vert e_{3}\vert\leq\min\left\lbrace\vert e_{1}\vert ,\vert e_{2}\vert\right\rbrace 
\right\rbrace,\,\,\mbox{for $s=5,6$}.
\end{eqnarray*}
\label{lemmamaximalus}
\end{lemma}

\begin{proof}
 We first prove this for $s=1$. Writing out the expressions for $\vert U_{i}e\vert^{2}$ we get
\begin{eqnarray}
\vert U_{s}e\vert^{2}&=&\beta^{2}e^{2}_{1}+\frac{\alpha^{2}+\gamma^{2}}{2}\left(e^{2}_{2}+e^{2}_{3}\right)+\left(-1\right)^{s-1}\left(\alpha^{2}-\gamma^{2}\right)e_{2}
e_{3}\quad\mbox{for $s=1,2$},\nonumber\\
\vert U_{s}e\vert^{2}&=&\beta^{2}e^{2}_{2}+\frac{\alpha^{2}+\gamma^{2}}{2}\left(e^{2}_{1}+e^{2}_{3}\right)+\left(-1\right)^{s-1}\left(\alpha^{2}-\gamma^{2}\right)e_{1}
e_{3}\quad\mbox{for $s=3,4$},\nonumber\\
\vert U_{s}e\vert^{2}&=&\beta^{2}e^{2}_{3}+\frac{\alpha^{2}+\gamma^{2}}{2}\left(e^{2}_{1}+e^{2}_{2}\right)+\left(-1\right)^{s-1}\left(\alpha^{2}-\gamma^{2}\right)e_{1}
e_{2}\quad\mbox{for $s=5,6$.}\nonumber
\end{eqnarray}
We first show that $e\in\mathcal{M}_{1}$ as given in the statement of the lemma is necessary and sufficient for $\vert U_{1}e\vert=\max_{i}\vert U_{i}e\vert$. We deal with the condition $|U_1e|\geq 1$ later. 
Writing $N=\frac{\alpha^{2}+\gamma^{2}}{2}-\beta^{2}>0$ and $P=\alpha^{2}-\gamma^{2}>0$,
\begin{eqnarray}
\label{eq:U1-2}
\vert U_{1}e\vert^{2}-\vert U_{2}e\vert^{2}&=&2Pe_{2}e_{3},\\
\label{eq:U1-3}
\vert U_{1}e\vert^{2}-\vert U_{3}e\vert^{2}&=&-N(e^{2}_{1}-e^{2}_{2})+Pe_{3}(e_{2}-e_{1}),\\
\label{eq:U1-4}
\vert U_{1}e\vert^{2}-\vert U_{4}e\vert^{2}&=&-N(e^{2}_{1}-e^{2}_{2})+Pe_{3}(e_{2}+e_{1}),\\
\label{eq:U1-5}
\vert U_{1}e\vert^{2}-\vert U_{5}e\vert^{2}&=&-N(e^{2}_{1}-e^{2}_{3})+Pe_{2}(e_{3}-e_{1}),\\
\label{eq:U1-6}
\vert U_{1}e\vert^{2}-\vert U_{6}e\vert^{2}&=&-N(e^{2}_{1}-e^{2}_{3})+Pe_{2}(e_{3}+e_{1}).
\end{eqnarray}
Trivially, (\ref{eq:U1-2}) is non-negative if and only if $e_2e_3\geq0$, i.e. $\vert e_2e_3\vert =e_2e_3$. Then
\begin{eqnarray}
\vert U_{1}e\vert^{2}-\vert U_{3}e\vert^{2}&=&-N(\vert e_{1}\vert^2-\vert e_{2}\vert^2)+P\vert e_{3}e_{2}\vert -Pe_3e_{1},\nonumber\\
\vert U_{1}e\vert^{2}-\vert U_{4}e\vert^{2}&=&-N(\vert e_{1}\vert^2-\vert e_{2}\vert^2)+P\vert e_{3}e_{2}\vert +Pe_3e_{1}\nonumber
\end{eqnarray}
and the minimum between (\ref{eq:U1-3}) and (\ref{eq:U1-4}) is given by
\begin{equation}
\label{eq:U1-3U1-4min}
-N(\vert e_{1}\vert^2-\vert e_{2}\vert^2)+P\vert e_{3}e_{2}\vert -P\vert e_3e_{1}\vert=(\vert e_{1}\vert-\vert e_{2}\vert)\left[-N\vert e_1\vert-N\vert e_2\vert-P\vert e_3\vert\right].
\end{equation}
Simply by interchanging the roles of $e_2$ and $e_3$, the minimum between (\ref{eq:U1-5}) and (\ref{eq:U1-6}) is given by
\begin{equation}
\label{eq:U1-5U1-6min}
-N(\vert e_{1}\vert^2-\vert e_{3}\vert^2)+P\vert e_{2}e_{3}\vert -P\vert e_2e_{1}\vert=(\vert e_{1}\vert-\vert e_{3}\vert)\left[-N\vert e_1\vert-N\vert e_3\vert-P\vert e_2\vert\right].
\end{equation}
Then, (\ref{eq:U1-2})-(\ref{eq:U1-6}) are all non-negative if and only if $e_2e_3\geq0$ and both (\ref{eq:U1-3U1-4min}) and (\ref{eq:U1-5U1-6min}) are non-negative, i.e. $e_2e_3\geq0$ and $\vert e_1\vert\leq\min\left\{\vert e_2\vert,\vert e_3\vert\right\}$.

It now suffices to show that any vector $e$ as above satisfies $\vert U_{1}e\vert^{2}\geq 1$. Writing $e^{2}_{1}+e^{2}_{2}+e^{2}_{3}=1$,
\begin{eqnarray}
\vert U_{1}e\vert^{2}-1&=&\beta^{2}e^{2}_{1}+\frac{\alpha^{2}+\gamma^{2}}{2}\left(e^{2}_{2}+e^{2}_{3}\right)+\left(\alpha^{2}-\gamma^{2}\right)e_{2}e_{3}-1\nonumber\\
&=&\left(\beta^{2}-1\right)e^{2}_{1}+\left(\frac{\alpha^{2}+\gamma^{2}}{2}-1\right)\left(e^{2}_{2}+e^{2}_{3}\right)+\left(\alpha^{2}-\gamma^{2}\right)e_{2}e_{3}.\nonumber
\end{eqnarray}
But $\vert e_{1}\vert\leq\min\left\lbrace\vert e_{2}\vert ,\vert e_{3}\vert\right\rbrace$ and $e_2e_3\geq 0$ implies that $2e^{2}_{1}\leq e^{2}_{2}+e^{2}_{3}$ and $e^{2}_{1}\leq e_{2}e_{3}$. Noting that $\vert U_{1}\vert^{2}\geq 3+\gamma^{2}-\alpha^{2}$,
\begin{eqnarray}
\vert U_{1}e\vert^{2}-1&\geq &\left(\beta^{2}-1\right)e^{2}_{1}+2\left(\frac{\alpha^{2}+\gamma^{2}}{2}-1\right)e^{2}_{1}+\left(\alpha^{2}-\gamma^{2}\right)e^{2}_{1}\nonumber\\
&=&\left(\vert U_{1}\vert^{2}-3+\alpha^{2}-\gamma^{2}\right)e^{2}_{1}\geq0.\nonumber
\end{eqnarray}
For the remaining variants, the proof is almost identical. Also note that the result follows easily due to the symmetry relations between the martensitic variants 
(see Appendix A). Suppose we wish to calculate the maximal directions for $U_{s}$. There exists a rotation $Q=Q\left[\pi,a\right]$ about an axis 
$a\in\mathbb{R}^{3}$ such that $U_{s}=QU_{1}Q$. But then,
\begin{eqnarray}
e\in\mathcal{M}_{s}&\Leftrightarrow &\vert U_{s}e\vert =\max_{i}\lbrace\vert U_{i}e\vert ,\;1\rbrace\nonumber\\
&\Leftrightarrow &\vert QU_{1}Qe\vert=\max_{i}\lbrace\vert U_{i}e\vert ,\;1\rbrace\nonumber\\
&\Leftrightarrow &\vert U_{1}\left(Qe\right)\vert =\max_{j}\lbrace\vert U_{j}Qe\vert ,\;1\rbrace\Leftrightarrow Qe\in\mathcal{M}_{1}\nonumber
\end{eqnarray}
since for each $j=1,\ldots,6$ there exists a unique $i\in\left\{1,\ldots,6\right\}$ such that $QU_jQ=U_i$ and hence $\vert U_jQe\vert=\vert U_ie\vert$. For example, for $U_2$ we have that $Q=Q\left[\pi,i_{3}\right]$ and
\begin{equation}
e\in\mathcal{M}_{2}\Leftrightarrow Q_{17}e\in\mathcal{M}_{1}\Leftrightarrow \left(-e_{1},-e_{2},e_{3}\right)\in\mathcal{M}_{1}\Leftrightarrow e_{2}e_{3}\leq 0\:\:
\mbox{and}\:\:\vert e_{1}\vert\leq\min\lbrace\vert e_{2}\vert ,\;\vert e_{3}\vert\rbrace.\nonumber
\end{equation}\qed
\end{proof}
 
\begin{figure}[ht]
	\centering
	\includegraphics[scale=0.5]{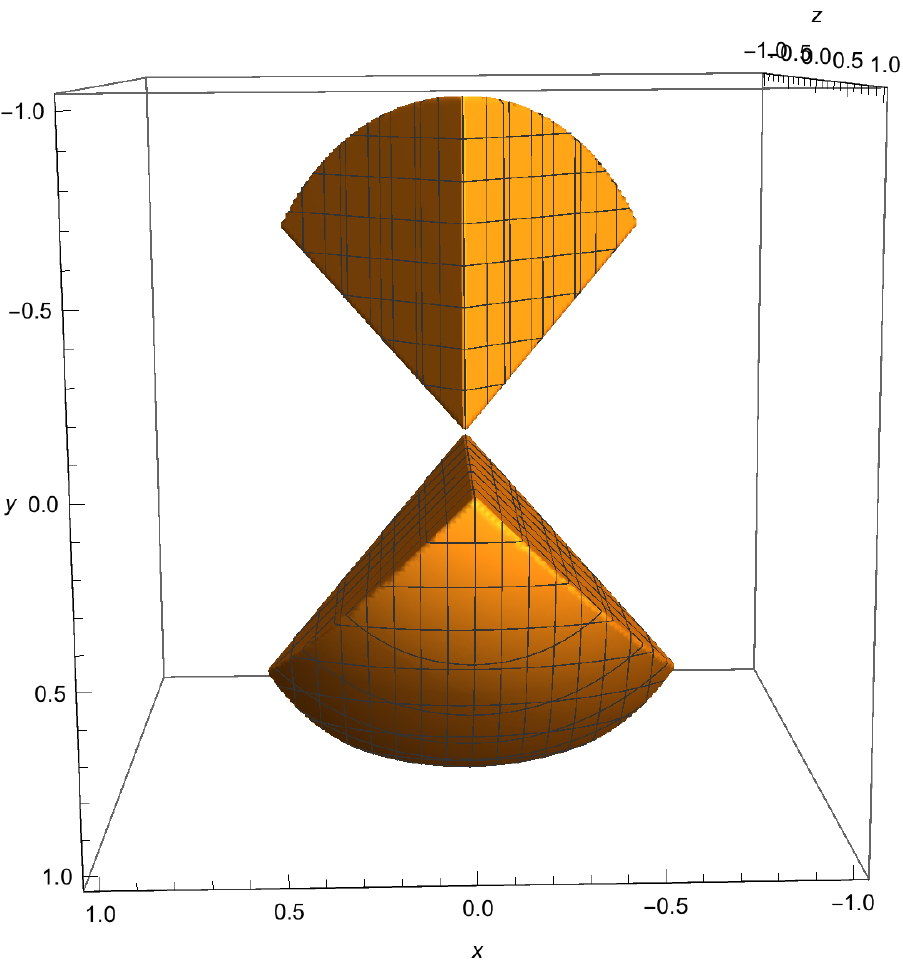} 
	\caption{All vectors $e\in\mathcal{M}_{1}$ calculated using the lattice parameters of the CuAlNi specimen in Seiner's experiment.}
\end{figure}

We now turn our attention to the problem of calculating the maximal directions for $U^{-1}_{s}$.
\begin{lemma}
Assume that the lattice parameters satisfy {\rm (A1)}, {\rm (A3)} and {\rm (A4)}. For a vector $e\in S^{2}$ write $e=\left(e_{1},e_{2},e_{3}\right)^{T}$. Then, for each $s=1,\ldots ,6$
\begin{eqnarray*}
\mathcal{M}^{-1}_{s}&=&\left\lbrace e\in S^{2}: \left(-1\right)^{s-1}e_{2}e_{3}< 0,\: \vert e_{1}\vert>\max\left\lbrace\vert e_{2}\vert ,\vert e_{3}\vert\right\rbrace
\right\rbrace\cup (1,0,0)^{T},\, s=1,2,\\
\mathcal{M}^{-1}_{s}&=&\left\lbrace e\in S^{2}: \left(-1\right)^{s-1}e_{1}e_{3}< 0,\: \vert e_{2}\vert>\max\left\lbrace\vert e_{1}\vert ,\vert e_{3}\vert\right\rbrace
\right\rbrace\cup (0,1,0)^{T},\, s=3,4,\\
\mathcal{M}^{-1}_{s}&=&\left\lbrace e\in S^{2}: \left(-1\right)^{s-1}e_{1}e_{2}< 0,\: \vert e_{3}\vert>\max\left\lbrace\vert e_{1}\vert ,\vert e_{2}\vert\right\rbrace
\right\rbrace\cup (0,0,1)^{T},\, s=5,6.
\end{eqnarray*}
\label{lemmamaximalusinverse}
\end{lemma}

\begin{proof}
We only prove this for $s=1$. The rest follows from the symmetry relations, since $\vert\mbox{cof}\,(QU_{1}Q)e\vert=\vert\mbox{cof}\,U_{1}
\left(Qe\right)\vert$. First note that $\lambda_{\max}(\mathrm{cof}\,U_1)=\alpha\gamma$ by (A1), so that $e_{\max}(\mathrm{cof}\,U_1)=(1,0,0)^T$. As for the remaining 
directions, note that
\begin{eqnarray}
\vert\mbox{cof}\;U_{s}e\vert^{2}&=&\alpha^{2}\gamma^{2}e^{2}_{1}+\beta^{2}\frac{\alpha^{2}+\gamma^{2}}{2}\left(e^{2}_{2}+e^{2}_{3}\right)+\left(-1\right)^{s-1}\beta^{2}
\left(\gamma^{2}-\alpha^{2}\right)e_{2}e_{3},\,s=1,2,\nonumber\\
\vert\mbox{cof}\;U_{s}e\vert^{2}&=&\alpha^{2}\gamma^{2}e^{2}_{2}+\beta^{2}\frac{\alpha^{2}+\gamma^{2}}{2}\left(e^{2}_{1}+e^{2}_{3}\right)+\left(-1\right)^{s-1}\beta^{2}
\left(\gamma^{2}-\alpha^{2}\right)e_{1}e_{3},\,s=3,4,\nonumber\\
\vert\mbox{cof}\;U_{s}e\vert^{2}&=&\alpha^{2}\gamma^{2}e^{2}_{3}+\beta^{2}\frac{\alpha^{2}+\gamma^{2}}{2}\left(e^{2}_{1}+e^{2}_{2}\right)+\left(-1\right)^{s-1}\beta^{2}
\left(\gamma^{2}-\alpha^{2}\right)e_{1}e_{2},\,s=5,6.\nonumber
\end{eqnarray}
First we show that $e\in\mathcal{M}^{-1}_{1}$, $e\neq e_{\max}(\mathrm{cof}\,U_s)$, with the above representation is necessary and sufficient for 
$\vert\mbox{cof}\,U_{1}e\vert>\max_{i\neq 1}\vert\mbox{cof}\,U_{i}e\vert$. We deal with the condition 
$\vert\mbox{cof}\; U_{s}e\vert> 1$ later. With $A>0$, $B>0$ as in the statement,
\begin{eqnarray}
\label{eq:cofU1-2}
\vert\mbox{cof}\;U_{1}e\vert^{2}-\vert\mbox{cof}\;U_{2}e\vert^{2}&=&-2Be_{2}e_{3}\\
\label{eq:cofU1-3}
\vert\mbox{cof}\;U_{1}e\vert^{2}-\vert\mbox{cof}\;U_{3}e\vert^{2}&=&A(e^{2}_{1}-e^{2}_{2})-Be_{3}(e_{2}-e_{1})\\
\label{eq:cofU1-4}
\vert\mbox{cof}\;U_{1}e\vert^{2}-\vert\mbox{cof}\;U_{4}e\vert^{2}&=&A(e^{2}_{1}-e^{2}_{2})-Be_{3}(e_{2}+e_{1})\\
\label{eq:cofU1-5}
\vert\mbox{cof}\;U_{1}e\vert^{2}-\vert\mbox{cof}\;U_{5}e\vert^{2}&=&A(e^{2}_{1}-e^{2}_{3})-Be_{2}(e_{3}-e_{1})\\
\label{eq:cofU1-6}
\vert\mbox{cof}\;U_{1}e\vert^{2}-\vert\mbox{cof}\;U_{6}e\vert^{2}&=&A(e^{2}_{1}-e^{2}_{3})-Be_{2}(e_{3}+e_{1}).
\end{eqnarray}
Trivially, (\ref{eq:cofU1-2}) is positive if and only if $e_{2}e_{3}< 0$; in particular, $\vert e_{2}e_{3}\vert=-e_{2}e_{3}$. Then,
\begin{eqnarray}
\vert\mbox{cof}\;U_{1}e\vert^{2}-\vert\mbox{cof}\;U_{3}e\vert^{2}&=&A(\vert e_1\vert^2-\vert e_2\vert^2)+B\vert e_3e_2\vert+Be_3e_1\vert,\nonumber\\
\vert\mbox{cof}\;U_{1}e\vert^{2}-\vert\mbox{cof}\;U_{4}e\vert^{2}&=&A(\vert e_1\vert^2-\vert e_2\vert^2)+B\vert e_3e_2\vert-Be_3e_1\vert\nonumber
\end{eqnarray}
and the minimum between (\ref{eq:cofU1-3}) and (\ref{eq:cofU1-4}) is given by
\begin{equation}
\label{eq:cofU1-3cofU1-4min}
A(\vert e_1\vert^2-\vert e_2\vert^2)+B\vert e_3e_2\vert+B\vert e_3e_1\vert=(\vert e_1\vert - \vert e_2\vert)\left[A\vert e_1\vert+A\vert e_2\vert-B\vert e_3\vert\right].
\end{equation}
Similarly, for (\ref{eq:cofU1-5}) and (\ref{eq:cofU1-6}), we need only interchange the roles of $e_2$ and $e_3$ and the minimum is given by
\begin{equation}
\label{eq:cofU1-5cofU1-6min}
A(\vert e_1\vert^2-\vert e_3\vert^2)+B\vert e_2e_3\vert+B\vert e_2e_1\vert=(\vert e_1\vert - \vert e_3\vert)\left[A\vert e_1\vert+A\vert e_3\vert-B\vert e_2\vert\right].
\end{equation}
Since $A>B$, if $\vert e_1\vert>\max\left\{\vert e_2\vert,\vert e_3\vert\right\}$, both (\ref{eq:cofU1-3cofU1-4min}) and (\ref{eq:cofU1-5cofU1-6min}) are positive and so 
are (\ref{eq:cofU1-3})-(\ref{eq:cofU1-6}). Conversely, suppose that (\ref{eq:cofU1-3})-(\ref{eq:cofU1-6}) are all positive. In particular,
\begin{eqnarray}
\label{eq:cofU1-3cofU1-4min0}
(\vert e_1\vert - \vert e_2\vert)\left[A\vert e_1\vert+A\vert e_2\vert-B\vert e_3\vert\right]&> &0\\
\label{eq:cofU1-5cofU1-6min0}
(\vert e_1\vert - \vert e_3\vert)\left[A\vert e_1\vert+A\vert e_3\vert-B\vert e_2\vert\right]&> &0.
\end{eqnarray}
Note that if $\vert e_1\vert>\vert e_3\vert$,
\[A\vert e_1\vert+A\vert e_2\vert-B\vert e_3\vert>0\]
since $A> B$ and $e\neq0$. Then (\ref{eq:cofU1-3cofU1-4min0}) says that $\vert e_1\vert>\vert e_2\vert$. Similarly, if $\vert e_1\vert>\vert e_2\vert$,
\[A\vert e_1\vert+A\vert e_3\vert-B\vert e_2\vert>0\]
and (\ref{eq:cofU1-5cofU1-6min0}) says that $\vert e_1\vert>\vert e_3\vert$, i.e.
\[\vert e_1\vert>\vert e_3\vert\,\Leftrightarrow\,\vert e_1\vert>\vert e_2\vert.\]
So, if $A\vert e_1\vert+A\vert e_2\vert-B\vert e_3\vert>0$ or $A\vert e_1\vert+A\vert e_3\vert-B\vert e_2\vert>0$, by the above argument, $\vert e_1\vert>\vert e_2\vert$ 
and $\vert e_1\vert>\vert e_3\vert$ and we need only examine the case
\begin{equation}
A\vert e_1\vert+A\vert e_2\vert-B\vert e_3\vert\leq0\:\:\:\mbox{and}\:\:\:A\vert e_1\vert+A\vert e_3\vert-B\vert e_2\vert\leq0.\nonumber
\end{equation}
Adding them up, $2A\vert e_1\vert+(A-B)(\vert e_2\vert+\vert e_3\vert)\leq0$ which is a contradiction since $A>B$.

To finish the proof, we show that for all vectors $e\in\mathcal{M}^{-1}_{1}$, $e\neq e_{\max}(\mathrm{cof}\,U_s)$, $\vert\mbox{cof}\;U_{1}e\vert^{2} -1> 0$. 
Writing $e^{2}_{1}+e^{2}_{2}+e^{2}_{3}=1$,
\begin{eqnarray}
\vert\mbox{cof}\;U_{1}e\vert^{2} -1&=&\alpha^{2}\gamma^{2}e^{2}_{1}+\beta^{2}\frac{\alpha^{2}+\gamma^{2}}{2}\left(e^{2}_{2}+e^{2}_{3}\right)+\beta^{2}\left(\gamma^{2}-\alpha^{2}\right)e_{2}e_{3}-e^{2}_{1}-e^{2}_{2}-e^{2}_{3}\nonumber\\
&=&\left(\alpha^{2}\gamma^{2}-1\right)e^{2}_{1}+\left(\beta^{2}\alpha^{2}+\beta^{2}\gamma^{2}-2\right)
\frac{e^{2}_{2}+e^{2}_{3}}{2}+\beta^{2}\left(\gamma^{2}-\alpha^{2}\right)e_{2}e_{3}\nonumber
\end{eqnarray}
However, $\vert e_{1}\vert>\max\lbrace\vert e_{2}\vert , \vert e_{3}\vert\rbrace$, so that $\vert e_{1}\vert ^{2}>(\vert e_{2}\vert ^{2}+\vert e_{3}\vert ^{2})/2$ and, 
since $e_{2}e_{3}< 0$,
\begin{eqnarray}
\vert\mbox{cof}\;U_{1}e\vert^{2} -1&\geq &\left(\alpha^{2}\gamma^{2}+\alpha^{2}\beta^{2}+\beta^{2}\gamma^{2}-3\right)\frac{e^{2}_{2}+e^{2}_{3}}{2}+\beta^{2}\left(\gamma^{2}-\alpha^{2}\right)e_{2}e_{3}\nonumber\\
&=&\left(\vert\mbox{cof}\;U_{1}\vert ^{2}-3\right)\frac{e^{2}_{2}+e^{2}_{3}}{2}+\beta^{2}\left(\alpha^{2}-\gamma^{2}\right)\vert e_{2}e_{3}\vert>0\nonumber
\end{eqnarray}
since $\vert\mathrm{cof}\;U_{1}\vert ^{2}\geq 3$.\qed
\end{proof}

\begin{figure}[ht]
	\centering
		\includegraphics[scale=0.5]{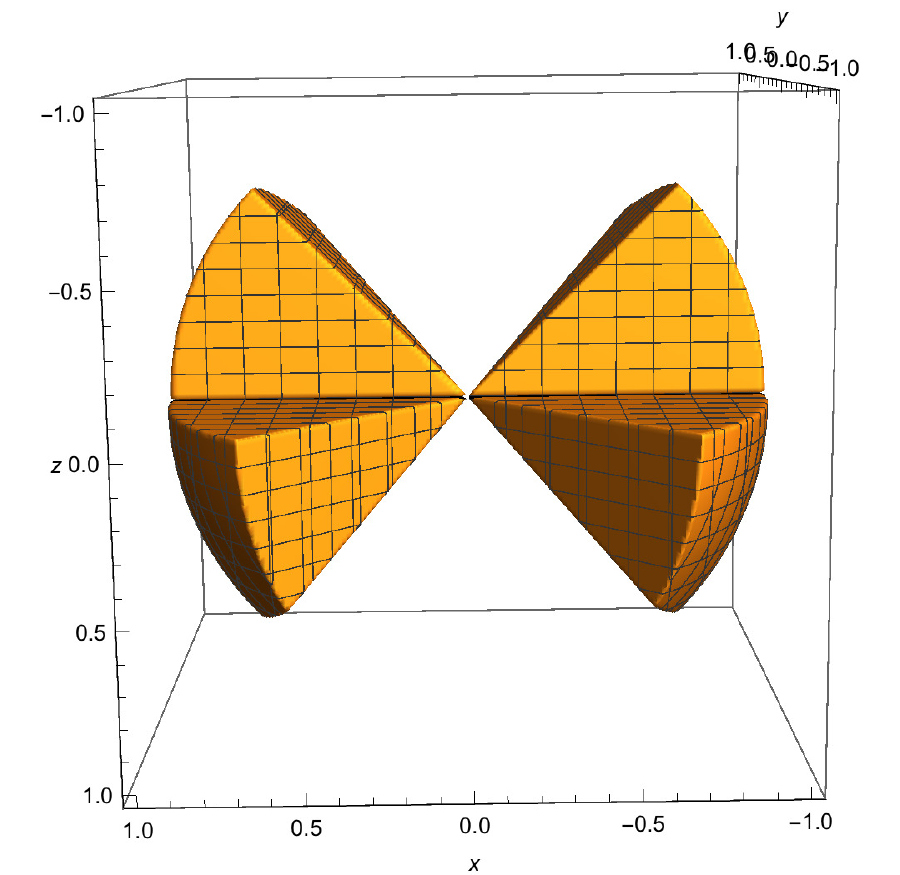} 
	\caption{All vectors $e\in\mathcal{M}^{-1}_{1}$ calculated using the lattice parameters of the CuAlNi specimen in Seiner's experiment.}
	\label{fig:coveringthesphereesinverses}
\end{figure}

It is natural to ask whether we are actually adding any new directions compared to the ones we had already obtained from $\mathcal{M}_{s}$. The answer is that all the 
directions in $U^{-2}_{s}\mathcal{M}^{-1}_{s}$ are in fact new. To show this, consider a vector $f=\left(f_{1},f_{2},f_{3}\right)\in\mathcal{M}^{-1}_{1}$, say, and 
suppose that $f$ is parallel to $U_1^2e$ for some unit vector $e=\left(e_{1},e_{2},e_{3}\right)^T$ in $U^{-2}_{1}\mathcal{M}^{-1}_{1}$. We claim that $e\notin\mathcal{M}_{1}$, where
\[\mathcal{M}_{1}=\left\lbrace e\in S^{2}: e_{2}e_{3}\geq 0,\:\:\vert e_{1}\vert\leq\min\left\lbrace\vert e_{2}\vert ,\vert e_{3}\vert\right\rbrace \right\rbrace.\]
This is because, for $\rho = |U^2_1e|$,
\[f=\frac{1}{\rho} U^{2}_{1}e=\frac{1}{\rho} \left(\beta^2 e_{1},\frac{\alpha^2+\gamma^2}{2}e_{2}+\frac{\alpha^2-\gamma^2}{2}e_{3},\frac{\alpha^2-\gamma^2}{2}e_{2}+\frac{\alpha^2+\gamma^2}{2}e_{3}\right)\]
must satisfy $f_{2}f_{3}< 0$ (being in $\mathcal{M}^{-1}_{1}$), i.e.
\begin{equation}
\frac{\alpha^{4}-\gamma^{4}}{4}\left(e^{2}_{2}+e^{2}_{3}\right)+\frac{\alpha^{4}+\gamma^{4}}{2}e_{2}e_{3}< 0
\label{notinA}
\end{equation}
However, the first term in the sum is non-negative and so is $\frac{\alpha^{4}+\gamma^{4}}{2}$. So, it must be the case that $e_{2}e_{3}<0$ and thus, 
$e\notin\mathcal{M}_{1}$, i.e.~$\mathcal{M}_{1}$ and $U^{-2}_{1}\mathcal{M}^{-1}_{1}$ are disjoint, so that we are genuinely adding new directions. On the other hand, it 
is clear that if $f$ is a unit vector parallel to $U^{2}_{1}e_{\max}(\mathrm{cof}\,U_1)$, then $f=e_{\max}(\mathrm{cof}\,U_1)\notin\mathcal{M}_{1}$.

An important issue is whether these directions can cover the unit sphere and hence exhaust all possible domains or, equivalently, whether we can deduce that 
for any bar-shaped specimen, nucleation can only occur at a corner. The answer is easily seen to be negative as the union of the sets  $\mathcal{M}_{s}$ and 
$U^{-2}_{s}\mathcal{M}^{-1}_{s}$ is a proper subset of the unit sphere. For example, consider $s=1$ and the unit vector $e=\frac{1}{\sqrt{2}}\left(0,\;1,\;-1\right)^T$. 
Since $e_{2}e_{3}=-1<0$, it becomes clear that $e\notin\mathcal{M}_{1}$ and on the other hand, $U^{2}_{1}e=\gamma^2 e$. But then, $\left[U^{2}_{1}e\right]_{1}=0$ and 
$U^{2}_{1}e\notin\mathcal{M}^{-1}_{1}$ or, equivalently, $e\notin U^{-2}_{1}\mathcal{M}^{-1}_{1}$. We note that writing out an explicit expression for all vectors outside 
$\mathcal{M}_{s}\bigcup U^{-2}_{s}\mathcal{M}^{-1}_{s}$ is a tedious task and, instead, we see these numerically for $s=1$ and the lattice parameters of Seiner's specimen 
in Fig.~\ref{fig:coveringthespherees} below.

\begin{figure}[ht]
	\centering
		\includegraphics[scale=0.5]{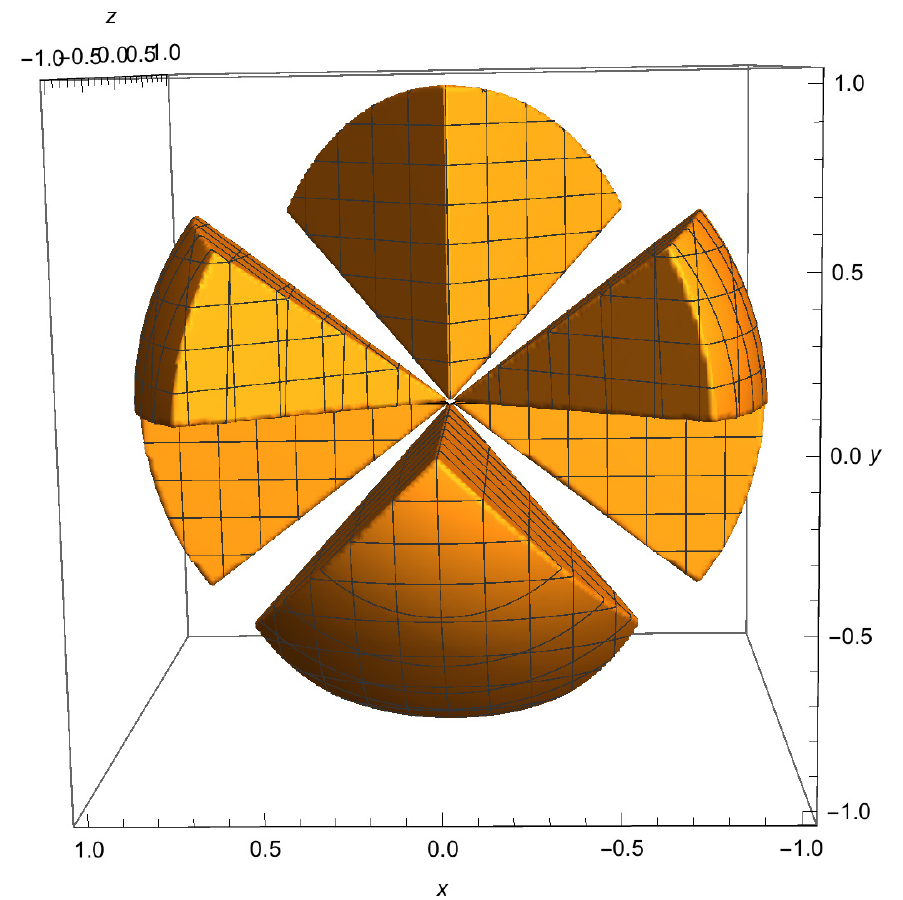} 
	\caption{All vectors $e\in\mathcal{M}_{1}\cup U^{-2}_{1}\mathcal{M}^{-1}_{1}$ calculated using the lattice parameters of the CuAlNi specimen in Seiner's experiment.}
	\label{fig:coveringthespherees}
\end{figure}

However, in the case of a face, this does not reveal much. The fact that the above directions do not cover the unit sphere does not necessarily imply that the normal vectors to these directions do not cover the sphere as well. Moreover, it would be more convenient to know all the possible normals to the faces that the method can decide on rather than the possible vectors lying on the face. The rest of this section aims to do precisely that. Let
\begin{eqnarray}
\mathcal{N}_{s}&=&\lbrace n\in S^{2}:\mbox{ there exists }e\in\mathcal{M}_{s}\mbox{ with }e\cdot n=0\rbrace\quad\mbox{and}\nonumber\\
\mathcal{N}^{-1}_{s}&=&\lbrace n\in S^{2}:\mbox{ there exists }e\in\mathcal{M}^{-1}_{s}\mbox{ with }e\cdot n=0\rbrace .\nonumber
\end{eqnarray}
The set we are then interested in is $\mathcal{N}_{s}\cup U^{2}_{s}\mathcal{N}^{-1}_{s}$. To see this, let $n\in\mathcal{N}_{s}\cup U^{2}_{s}\mathcal{N}^{-1}_{s}$.
\begin{itemize}
\item If $n\in\mathcal{N}_{s}$ there exists $e\in\mathcal{M}_{s}$ such that $e\cdot n=0$ and we can apply the method to any face with normal $n$ as there will exist a maximal direction $e$ which lies on the face.
\item On the other hand, if $n\in U^{2}_{s}\mathcal{N}^{-1}_{s}$ then there exists $m\in\mathcal{N}^{-1}_{s}$ such that $n=U^{2}_{s}m$. Since $m\in\mathcal{N}^{-1}_{s}$ there exists $e\in\mathcal{M}^{-1}_{s}$ such that $e\cdot m=0$. Once again, we can then apply the method to any face with normal $n$ as
\[U^{-2}_{s}e\cdot n=U^{-2}_{s}e\cdot U^{2}_{s}m=e\cdot m=0,\]
i.e.~$U^{-2}_{s}e$ lies on the face with normal $n$ and $e$ is a maximal direction for $U^{-1}_{s}$.
\end{itemize}

Let us characterize these sets and see whether these can exhaust the unit sphere, i.e.~whether our methods can be applied to all possible faces.

\begin{lemma}
Under the assumptions {\rm (A1)} and {\rm (A2)} on the lattice parameters and for each $s=1,\ldots\,,6$
\begin{eqnarray*}
\mathcal{N}_{s}&=&\left\lbrace n\in S^{2}: \left(-1\right)^{s-1}n_{2}n_{3}\leq 0\;\mbox{or}\;\vert n_{1}\vert\geq\vert n_{2}\vert +\vert n_{3}\vert\right\rbrace,\,\,
s=1,2\\ 
\mathcal{N}_{s}&=&\left\lbrace n\in S^{2}: \left(-1\right)^{s-1}n_{1}n_{3}\leq 0\;\mbox{or}\;\vert n_{2}\vert\geq\vert n_{1}\vert +\vert n_{3}\vert\right\rbrace,\,\,
s=3,4\\ 
\mathcal{N}_{s}&=&\left\lbrace n\in S^{2}: \left(-1\right)^{s-1}n_{1}n_{2}\leq 0\;\mbox{or}\;\vert n_{3}\vert\geq\vert n_{1}\vert +\vert n_{2}\vert\right\rbrace,\,\,
s=5,6.
\end{eqnarray*}
\label{lemmanormalstoa}
\end{lemma}

\begin{proof}
We only prove this for $s=1$. The rest follow easily due to symmetry. Let us first show that if a vector $n\in S^2$ has the above representation then there exists 
$e\in\mathcal{M}_{1}$ such that $e\cdot n=0$. Suppose that $n$ satisfies $n_{2}n_{3}\leq 0$. If $n_{2}=n_{3}=0$ then $n=\pm\left(1,0,0\right)^T$ and 
$e=\left(0,1,0\right)^T\in\mathcal{M}_{1}$ satisfies $e\cdot n=0$. So, assume that $\vert n_2\vert + \vert n_3\vert\neq0$ and let $e\in\mathcal{M}_{1}$ be the vector 
$\frac{1}{\rho}\left(0,\vert n_{3}\vert , \vert n_{2}\vert\right)$ where $\rho>0$ is a constant making $e$ of unit length. Then,
\begin{equation}
\rho e\cdot n=n_{2}\vert n_{3}\vert +n_{3}\vert n_{2}\vert =0\nonumber
\end{equation}
as either one of $n_2$, $n_3$ is zero or they have opposite sign. On the other hand, suppose that $\vert n_{1}\vert\geq\vert n_{2}\vert +\vert n_{3}\vert$. Clearly, we 
may additionally assume that $n_{2}n_{3}>0$ as otherwise it reduces to the previous case. Note that if $e=\left(e_{1},e_{2},e_{3}\right)^T$ satisfies $e\cdot n=0$ then
\begin{equation}
f\cdot\left(n_{1},-n_{2},-n_{3}\right)^T=0\nonumber
\end{equation}
for $f=\left(e_{1},-e_{2},-e_{3}\right)^T$ and if $e\in\mathcal{M}_{1}$ so is $f$. Hence, without loss of generality, we may assume that $n_{2}>0$ and $n_{3}>0$. Then, 
$\vert n_{2}\vert +\vert n_{3}\vert =n_{2}+n_{3}$ and $(\vert n_{2}+n_{3}\vert)/\vert n_{1}\vert\leq 1$ by assumption. So, let $e\in\mathcal{M}_1$ be the vector 
$\frac{1}{\rho}\left(-\left(n_{2}+n_{3}\right)/n_{1},1,1\right)^T$, where $\rho>0$ forces $\vert e\vert=1$, to get
\[\rho e\cdot n=-\left(n_{2}+n_{3}\right)+n_{2}+n_{3}=0\]
and sufficiency is established.

Conversely, to reach a contradiction, suppose that $n\cdot e=0$ for some $e\in\mathcal{M}_{1}$ but
\[n_{2}n_{3}>0\quad\mbox{and}\quad\vert n_{1}\vert <\vert n_{2}\vert +\vert n_{3}\vert.\]
Since $n_{2}n_{3}>0$ and $e_{2}e_{3}\geq 0$, it must be the case that $\left(n_{2}e_{2}\right)\left(n_{3}e_{3}\right)\geq 0$ and hence
\begin{equation}
\label{eq:lemmanormalstoa1}
\vert n_{2}e_{2}+n_{3}e_{3}\vert =\vert n_{2}e_{2}\vert +\vert n_{3}e_{3}\vert.
\end{equation}
If $n_{1}e_{1}+n_{2}e_{2}+n_{3}e_{3}=0$, by (\ref{eq:lemmanormalstoa1}),
\begin{equation}
\vert n_{1}e_{1}\vert =\vert n_{2}e_{2}\vert +\vert n_{3}e_{3}\vert\geq\min\lbrace\vert e_{2}\vert ,\vert e_{3}\vert\rbrace\left(\vert n_{2}\vert +\vert n_{3}\vert\right)
\geq\vert e_{1}\vert\left(\vert n_{2}\vert +\vert n_{3}\vert\right),\nonumber
\end{equation}
so that $\vert n_{1}\vert\geq\vert n_{2}\vert +\vert n_{3}\vert$, a contradiction proving necessity. Note that for the remaining variants
\begin{eqnarray}
n\in\mathcal{N}_{s}&\Leftrightarrow &\mbox{ there exists }e\in\mathcal{M}_{s}\mbox{ such that }e\cdot n=0\nonumber\\
&\Leftrightarrow &\mbox{ there exists }f\in\mathcal{M}_{1}\mbox{ such that }Qf\cdot n=0\mbox{ where $QU_{1}Q=U_{s}$}\nonumber\\
&\Leftrightarrow &Qn\in\mathcal{N}_{1}\nonumber
\end{eqnarray}
and the result follows easily.\qed
\end{proof}

\begin{figure}[ht]
	\centering
		\includegraphics[scale=0.5]{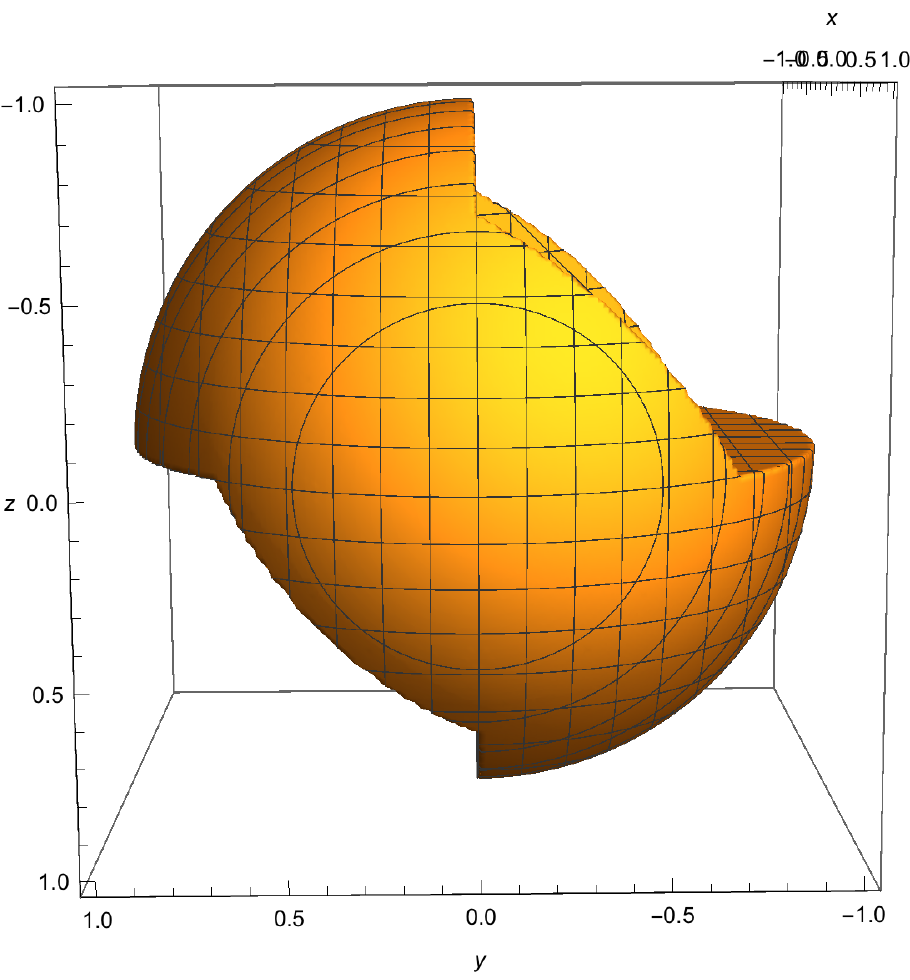} 
	\caption{All vectors $n\in\mathcal{N}_{1}$ calculated using the lattice parameters of the CuAlNi specimen in Seiner's experiment.}
	\label{fig:coveringthespherestraight}
\end{figure}

\begin{lemma}
Under the assumptions {\rm (A1)}, {\rm (A3)} and {\rm (A4)} on the lattice parameters
\begin{eqnarray*}
\mathcal{N}^{-1}_{s}&=&\left\lbrace n\in S^{2}: \left(-1\right)^{s-1}n_{2}n_{3}\leq 0\;\mbox{and}\;\vert n_{1}\vert<\vert n_{2}\vert +\vert n_{3}\vert\right\rbrace\\
&\cup&\left\lbrace n\in S^{2}: \left(-1\right)^{s-1}n_{2}n_{3}\geq 0\;\mbox{and}\;\vert n_{1}\vert<\max\lbrace\vert n_{2}\vert ,\;\vert n_{3}\vert\rbrace\right\rbrace
\cup\lbrace n\in S^{2}: n\cdot (1,0,0)^T=0\rbrace,\,\, s=1,2,\\
\mathcal{N}^{-1}_{s}&=&\left\lbrace n\in S^{2}: \left(-1\right)^{s-1}n_{1}n_{3}\leq 0\;\mbox{and}\;\vert n_{2}\vert<\vert n_{1}\vert +\vert n_{3}\vert\right\rbrace\\
&\cup&\left\lbrace n\in S^{2}: \left(-1\right)^{s-1}n_{1}n_{3}\geq 0\;\mbox{and}\;\vert n_{2}\vert<\max\lbrace\vert n_{1}\vert ,\;\vert n_{3}\vert\rbrace\right\rbrace 
\cup\lbrace n\in S^{2}: n\cdot (0,1,0)^T=0\rbrace,\,\, s=3,4,\\
\mathcal{N}^{-1}_{s}&=&\left\lbrace n\in S^{2}: \left(-1\right)^{s-1}n_{1}n_{2}\leq 0\;\mbox{and}\;\vert n_{3}\vert<\vert n_{1}\vert +\vert n_{2}\vert\right\rbrace\\
&\cup&\left\lbrace n\in S^{2}: \left(-1\right)^{s-1}n_{1}n_{2}\geq 0\;\mbox{and}\;\vert n_{3}\vert<\max\lbrace\vert n_{1}\vert ,\vert n_{2}\vert\rbrace\right\rbrace
\cup\lbrace n\in S^{2}: n\cdot (0,0,1)^T=0\rbrace,\,\, s=5,6.
\end{eqnarray*}
\label{lemmanormalstob}
\end{lemma}

\begin{proof}
We only show this for $s=1$ and the rest follows easily due to symmetry. Note that the plane $\lbrace n\in S^{2}: n\cdot (1,0,0)^T=0 \rbrace$ 
corresponds precisely to the set of vectors perpendicular to $e_{\max}(\mathrm{cof}\,U_1)$ and we only deal with the remaining maximal directions.

Assume first that $n\in S^2$ has the representation above; we wish to conclude that $e\cdot n=0$ for some $e\in\mathcal{M}^{-1}_{1}$. Note that $n_2$ and $n_3$ cannot 
both be zero and we may also assume that $n_1\neq0$ as otherwise
\[n\cdot (1,0,0)^T=0.\]
Next, suppose that either $n_2=0$ or $n_3=0$; without loss of generality assume that $n_3=0$. Then, from the above representations, $\vert n_1\vert<\vert n_2\vert$ so 
that the vector $e=\frac{1}{\rho}(n_2,-n_1,n_1)$ is an element of $\mathcal{M}^{-1}_{1}$ ($\rho>0$ forces $\vert e\vert=1$) and satisfies
\[\rho e\cdot n= n_2n_1-n_1n_2=0.\]
Hence, we may assume that $n_1n_2n_3\neq0$. Let us consider the case $n_{2}n_{3}< 0$ and $\vert n_{1}\vert<\vert n_{2}\vert +\vert n_{3}\vert $. Since $n_2n_3<0$ we infer 
that $\vert n_2-n_3\vert=\vert n_2\vert +\vert n_3\vert$, i.e.~$\vert n_2-n_3\vert>\vert n_1\vert$ and the vector $\frac{1}{\rho}(n_2-n_3,-n_1,n_1)^T$ is an 
element of $\mathcal{M}^{-1}_{1}$ satisfying
\[\frac{1}{\rho}(n_2-n_3,-n_1,n_1)^T\cdot n=0.\]
Next let $n_{2}n_{3}> 0$ and $\vert n_{1}\vert<\max\lbrace\vert n_{2}\vert,\vert n_{3}\vert\rbrace$. Without loss of generality, suppose that $\vert n_3\vert\geq\vert n_2\vert$ and let $\kappa\geq1$ be any number such that
\[\kappa\frac{\vert n_3\vert}{\vert n_1\vert}-\frac{\vert n_2\vert}{\vert n_1\vert}>\kappa.\]
The existence of $\kappa$ is trivial since $\vert n_3\vert>\vert n_1\vert$. With $\rho>0$ a normalizing factor, define the vector
\[e=\frac{1}{\rho}(\kappa n_3-n_2,n_1,-\kappa n_1)^T.\]
Clearly, $e\cdot n=0$ and we are left to show that $e\in\mathcal{M}^{-1}_{1}$. But $e_2e_3=-\kappa n^{2}_{1}<0$ since we are assuming that $n_1\neq0$; also, since 
$n_2n_3>0$ and $\kappa\geq 1$ we infer that
\begin{equation*}
\vert e_1\vert=\vert\kappa n_3- n_2\vert=\kappa\vert n_3\vert-\vert n_2\vert>\kappa\vert n_1\vert=\max\left\{\vert e_2\vert , \vert e_3\vert\right\}.
\end{equation*}

Conversely, assume that $n\in S^2$ and that there exists $e\in\mathcal{M}^{-1}_{1}$ such that $e\cdot n=0$. Then,
\begin{equation}
\label{eq:lemmanormalstob2}
\vert n_{1}e_{1}\vert =\vert n_{2}e_{2}+n_{3}e_{3}\vert
\end{equation}
and we distinguish between two cases depending on the sign of $n_{2}n_{3}$. Note that since $e\in\mathcal{M}^{-1}_{1}$, it must be the case that $e_{2}e_{3}< 0$.
If $n_{2}n_{3}\leq 0$, $\left(n_{2}e_{2}\right)\left(n_{3}e_{3}\right)\geq 0$ and $\vert n_{2}e_{2}+n_{3}e_{3}\vert=\vert n_{2}e_{2}\vert +\vert n_{3}e_{3}\vert$, i.e. by 
(\ref{eq:lemmanormalstob2}),
\begin{equation*}
\vert n_{1}e_{1}\vert =\vert n_{2}e_{2}\vert +\vert n_{3}e_{3}\vert\leq\max\lbrace\vert e_{2}\vert,\vert e_{3}\vert\rbrace\left(\vert n_{2}\vert +\vert n_{3}\vert\right)
<\vert e_{1}\vert\left(\vert n_{2}\vert +\vert n_{3}\vert\right)
\end{equation*}
since $e\in\mathcal{M}^{-1}_{1}$. Now $e_{1}\neq 0$ as otherwise, $e$ belonging to $\mathcal{M}^{-1}_{1}$, forces $e_{2}=e_{3}=0$. Therefore, $\vert n_{1}\vert < \vert n_{2}\vert +\vert n_{3}\vert$ and this case is finished.

On the other hand, if $n_{2}n_{3}\geq 0$ we get that $\left(n_{2}e_{2}\right)\left(n_{3}e_{3}\right)\leq 0$ and thus, $\vert n_{1}e_{1}\vert =\vert\vert n_{2}e_{2}\vert -\vert n_{3}e_{3}\vert\vert$. Then, by (\ref{eq:lemmanormalstob2}),
\begin{equation*}
\vert n_{1}e_{1}\vert =\vert\vert n_{2}e_{2}\vert -\vert n_{3}e_{3}\vert\vert\leq\max\lbrace\vert n_{2}e_{2}\vert,\vert n_{3}e_{3}\vert\rbrace <
\vert e_{1}\vert\max\lbrace\vert n_{2}\vert,\vert n_{3}\vert\rbrace
\end{equation*}
since $e\in\mathcal{M}^{-1}_{1}$. But $e_{1}\neq 0$ now says that $\vert n_{1}\vert < \max\lbrace\vert n_{2}\vert,\vert n_{3}\vert\rbrace$ and the proof is complete.\qed 
\end{proof}

\begin{figure}[ht]
	\centering
		\includegraphics[scale=0.5]{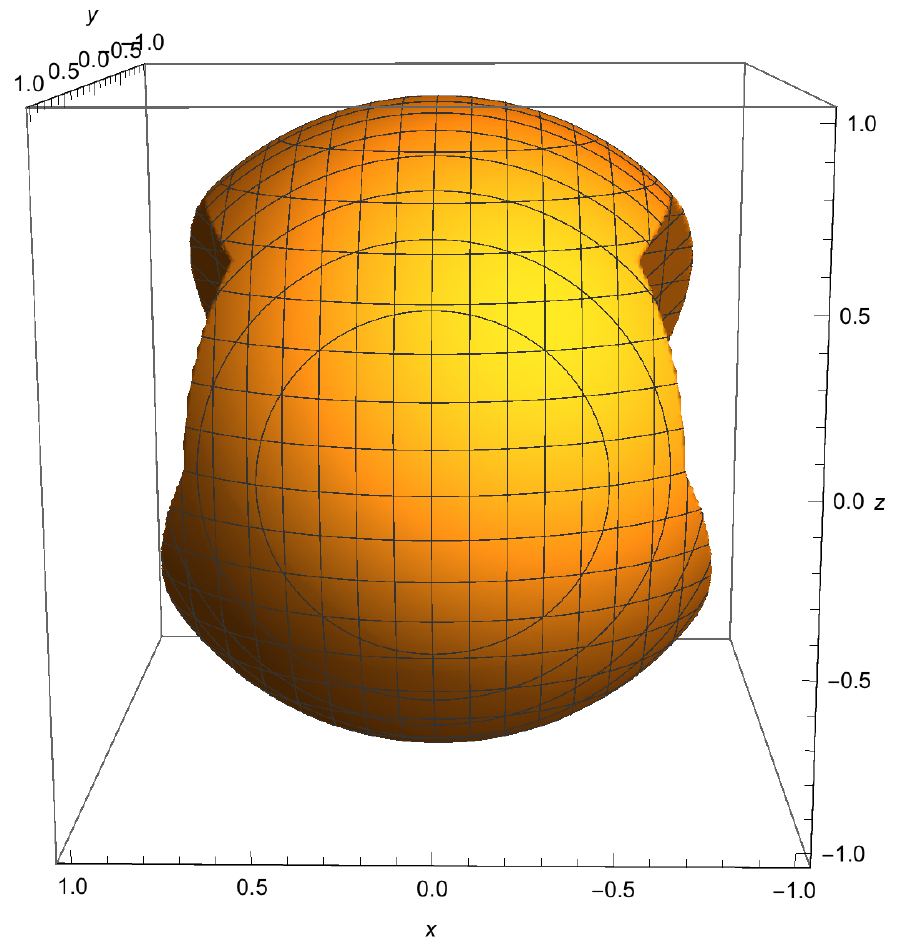} 
	\caption{All vectors $n\in\mathcal{N}^{-1}_{1}$ calculated using the lattice parameters of the CuAlNi specimen in Seiner's experiment.}
	\label{fig:coveringthesphereinverse}
\end{figure}

As described already, we are interested in whether the set $\mathcal{N}_{s}\cup U^{2}_{s}\mathcal{N}^{-1}_{s}$ can exhaust the entire unit sphere; the answer is next seen 
to be negative.

\begin{lemma}
Suppose that a unit vector $n$ satisfies $n\notin\mathcal{N}_{s}\cup\mathcal{N}^{-1}_{s}$. Then,
\[U^{2}_sn\notin\mathcal{N}_{s}\cup U^{2}_{s}\mathcal{N}^{-1}_{s}.\]
In particular, there exists $m\in S^{2}$ such that $m\notin\mathcal{N}_{s}\cup U^{2}_{s}\mathcal{N}^{-1}_{s}$.
\end{lemma}

\begin{proof}
We only prove this for $s=1$, the rest being identical. Let us first show that
\begin{equation}
\label{eq:normalsdonotexhaustsphere1}
 n\notin\mathcal{N}_{1}\Rightarrow U^{2}_{1}n\notin\mathcal{N}_{1}.
\end{equation}
Let $n\notin\mathcal{N}_{1}$ so that $n_{2}n_{3}>0$ and $\vert n_{1}\vert <\vert n_{2}\vert +\vert n_{3}\vert$. Then,
\begin{equation*}
U^{2}_{1}n=\left(\beta^{2} n_{1},\;\frac{\alpha^{2}+\gamma^{2}}{2}n_{2}+\frac{\alpha^{2}-\gamma^{2}}{2}n_{3},\;\frac{\alpha^{2}-\gamma^{2}}{2}n_{2}+\frac{\alpha^{2}+
\gamma^{2}}{2}n_{3}\right)
\end{equation*}
and we infer that, since $n_{2}n_{3}>0$ and $\alpha>\gamma>0$, $\left[U^{2}_{1}n\right]_{2}\left[U^{2}_{1}n\right]_{3}>0$. Moreover, as $n\notin\mathcal{N}_{1}$ and 
$\beta\leq\alpha$,
$$\vert\left[U^{2}_{1}n\right]_{1}\vert=\beta^{2}\vert n_{1}\vert<\alpha^{2}\left(\vert n_{2}\vert +\vert n_{3}\vert\right)=\alpha^{2}\vert n_{2}+n_{3}\vert\,\,
(\mbox{since $n_{2}n_{3}>0$}).$$
But then, since $\left[U^{2}_{1}n\right]_{2}\left[U^{2}_{1}n\right]_{3}>0$,
\begin{eqnarray}
\vert\left[U^{2}_{1}n\right]_{1}\vert&=&\vert\left(\frac{\alpha^{2}+\gamma^{2}}{2}n_{2}+\frac{\alpha^{2}-\gamma^{2}}{2}n_{3}\right)+\left(\frac{\alpha^{2}-\gamma^{2}}
{2}n_{2}+\frac{\alpha^{2}+\gamma^{2}}{2}n_{3}\right)\vert\nonumber\\
&=&\vert\left[U^{2}_{1}n\right]_{2}+\left[U^{2}_{1}n\right]_{3}\vert=\vert\left[U^{2}_{1}n\right]_{2}\vert +\vert\left[U^{2}_{1}n\right]_{3}\vert\nonumber
\end{eqnarray}
and (\ref{eq:normalsdonotexhaustsphere1}) is established. Then, $n\notin\mathcal{N}_{1}\cup\mathcal{N}^{-1}_{1}\Rightarrow U^{2}_{1}n\notin\mathcal{N}_{1}\cup 
U^{2}_{1}\mathcal{N}^{-1}_{1}$.
Finally, it is easy to check that whenever $n=\left(n_{1},n_{2},n_{3}\right)^T$ satisfies
\begin{equation}
n_{2}n_{3}>0\:\:\mbox{and}\:\:\max\lbrace\vert n_{2}\vert,\vert n_{3}\vert\rbrace\leq\vert n_{1}\vert <\vert n_{2}\vert +\vert n_{3}\vert ,\nonumber
\end{equation}
then $n\notin\mathcal{N}_{1}\cup\mathcal{N}^{-1}_{1}$, i.e.~$m=U^{2}_{1}n\notin\mathcal{N}_{1}\cup U^{2}_{1}\mathcal{N}^{-1}_{1}$.\qed
\end{proof}

\begin{figure}[ht]
	\centering
		\includegraphics[scale=0.5]{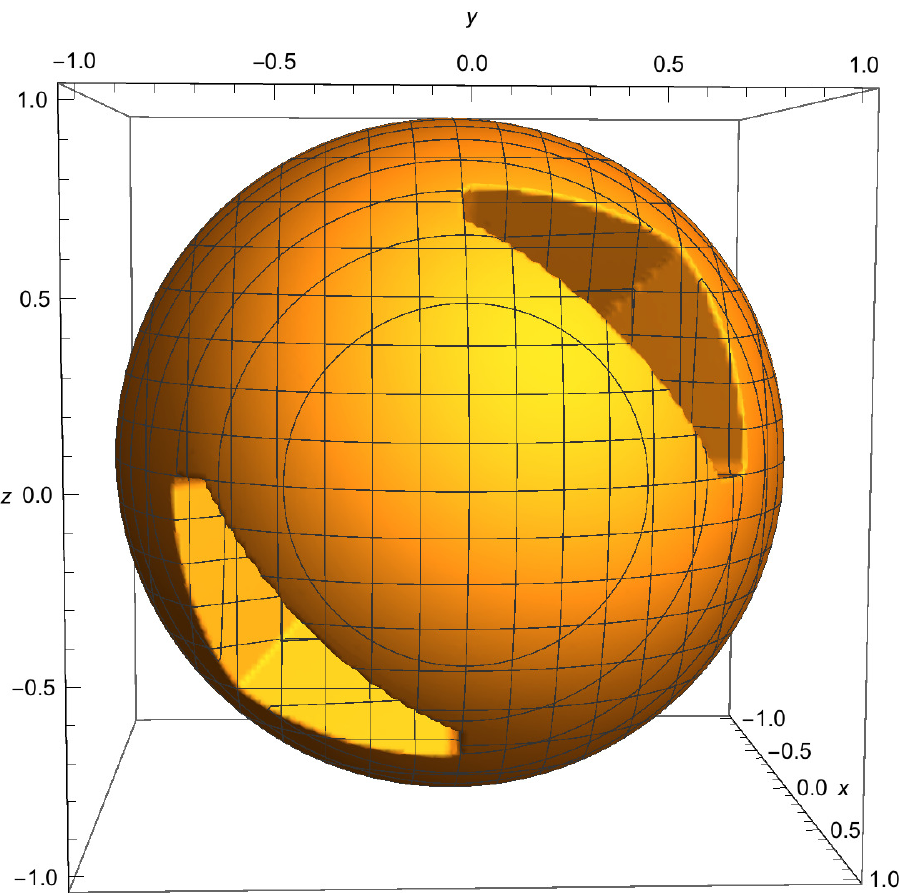} 
	\caption{All vectors $n\in\mathcal{N}_{1}\cup U^{2}_{1}\mathcal{N}^{-1}_{1}$ calculated using the lattice parameters of the CuAlNi specimen in Seiner's experiment.}
	\label{fig:coveringthesphere}
\end{figure}

The edges of the CuAlNi specimen in Seiner's experiment were oriented very nearly along $\left(1,0,0\right)^{T}$, $\left(0,1,0\right)^{T}$ and $\left(0,0,1\right)^{T}$ 
and our result becomes applicable for any $s=1,\ldots\,,6$. In particular, suppose that the mechanically stabilized variant of martensite is $U_{1}$.

As remarked already, the lattice parameters satisfy assumptions (A1)-(A4) and the representations of $\mathcal{M}_{1}$ and $\mathcal{M}^{-1}_{1}$ given in 
Lemma~\ref{lemmamaximalus} and Lemma~\ref{lemmamaximalusinverse} hold. In particular, the maximal directions for $U_{1}$ and $U^{-1}_{1}$ are
\begin{eqnarray*}
\mathcal{M}_{1}&=&\left\lbrace e\in S^{2}: e_{2}e_{3}\geq 0,\:\:\vert e_{1}\vert\leq\min\left\lbrace\vert e_{2}\vert ,\vert e_{3}\vert\right\rbrace \right\rbrace,\\
\mathcal{M}^{-1}_{1}&=&\left\lbrace e\in S^{2}:e_{2}e_{3}< 0,\:\:\vert e_{1}\vert>\max\left\lbrace\vert e_{2}\vert ,\vert e_{3}\vert\right\rbrace \right\rbrace\cup(1,0,0)^T.
\end{eqnarray*}
But the vectors $\left(0,1,0\right)^T$, $\left(0,0,1\right)^T$ belong to $\mathcal{M}_{1}$ and the vector $\left(1,0,0\right)^{T}$ is an element of $\mathcal{M}^{-1}_{1}$ and, being an eigenvector of $U_{1}$, it is also an element of $U^{-2}_{1}\mathcal{M}^{-1}_{1}$.
Therefore, one can cover any nucleation region $B_{f}$ on a face or $B_{e}$ at an edge and our main result applies.

\section{General remarks}
\label{sec:4.2}

It is worth mentioning that the set of maximal directions - as calculated in Lemma~\ref{lemmamaximalusinverse}, say for $s=1$, (similarly for the rest) - is the interior 
of the set
\[\left\lbrace e\in S^{2}: e_{2}e_{3}\leq 0,\: \vert e_{1}\vert\geq\max\left\lbrace\vert e_{2}\vert ,\vert e_{3}\vert\right\rbrace\right\rbrace,\]
where strict inequalities have been replaced with inequalities. In particular, we note that the above set contains the vector $e_{\max}(\mathrm{cof}\,U_1)$. Then, 
defining the set
\[\bar{\mathcal{M}}^{-1}_{s}:=\left\lbrace e\in S^{2}: \vert\mathrm{cof}\,U_se\vert=\max_{i}\left\{\vert\mathrm{cof}\,U_ie\vert,1\right\}\right\rbrace\]
it is possible to use a continuity argument and deduce the quasiconvexity conditions at any face with normal perpendicular to, or any edge in the direction of, a vector 
in $U^{-2}_{s}\bar{\mathcal{M}}^{-1}_{s}$. However, we have not been able to show that the set $\mathcal{M}^{-1}_{s}$ is the interior of $\bar{\mathcal{M}}^{-1}_{s}$ for a general transformation and 
without stringent assumptions on the lattice parameters.

Also, we note that for any $s=1,\ldots\,,6$ and appropriate lattice parameters, the set $\mathcal{M}_{s}\cup U^{-2}_{s}\mathcal{M}^{-1}_{s}$ contains several sets of 
three linearly independent directions and our methods apply to a variety of parallelepipeds with edges along these directions. 
However, for these lattice parameters, $\mathcal{M}_{s}\cup U^{-2}_{s}\mathcal{M}^{-1}_{s}$ does not exhaust the unit sphere. Hence our result 
leaves open the possibility that for differently cut specimens nucleation could occur at a face or an edge.

We stress that our analysis carries through even if we depart from the exact form of the energy wells of the cubic-to-orthorhombic transition. 
This exact form of the set $K$ is only relevant for the calculation of the maximal directions and the remainder of our analysis remains applicable for any number of martensitic variants 
and any transformation with cubic austenite. In fact, for our analysis to apply we need only require that
\[K=SO\left(3\right)\cup\bigcup^{N}_{i=1}SO\left(3\right)U_{i}\]
is such that the variants $U_{i}$ are symmetry related and satisfy the constraint on the determinant and the largest eigenvalue of the cofactor matrix. Of course, this is 
not necessarily the case for the construction at a corner.

We note that the argument of maximal directions reduces to measuring the length of lines; having fixed the endpoints of the deformed segments, the definition of a maximal 
direction ensures that its length cannot be greater than the length of the straight line joining the endpoints and, thus, it must precisely be that straight line. Similar ideas are 
also used in Sivaloganathan \& Spector~\cite{jay&spector} to show that for an incompressible cylinder under uniaxial 
extension, the unique minimizer is a homogeneous, isoaxial deformation (under additional constitutive hypotheses on the stored-energy function).
Note that it is also possible to use a similar method based on the cofactor matrix and areas of surfaces rather than lengths. We mention, however, that such a method allows 
one to deduce the quasiconvexity condition on faces with normals in $\mathcal{M}^{-1}_{s}$ and does not seem applicable to edges. In particular, since 
$\mathcal{M}^{-1}_{s}\subset\mathcal{N}_{s}$, we would add no new faces.

Also, as we alluded to in the introduction, there are connections between our work and that of Grabovsky \& Mengesha~\cite{grabovsky2009}. In particular, we mentioned 
that Grabovsky and Mengesha base their sufficiency proof on a decomposition lemma splitting arbitrary variations into a strong and a weak part which cannot lower the 
energy due to the (uniform) positivity of the second variation and the (uniform) quasiconvexity conditions, respectively.
In our analysis, we have restricted attention to localized variations corresponding to the localized nucleation of austenite. Nevertheless, if we depart from our 
simplified setting of the singular energy, one expects to be able to prove something stronger using the machinery of Grabovsky and Mengesha; in particular, one might be 
able to show that whenever $\nu$ is a $W^{1,\infty}$ gradient Young measure such that $\mathrm{d}\left(\delta_{U_s},\nu\right)$ is sufficiently small and 
$I\left(\nu\right)<I\left(\delta_{U_s}\right)$, then $\nu$ must necessarily involve nucleation at a corner. Here, the distance $\mathrm{d}\left(\cdot,\,\cdot\right)$, 
given by (\ref{eq:ymdistance}), comes from the metrization of the weak$\ast$ topology in $L^{\infty}_{w^{\ast}}\left(\Omega,\mathcal{M}\left(M^{3\times 3}\right)\right)$. 
This is a direction worth investigating further.

Concluding our final remarks, we mention that the same nucleation mechanism at a corner was also observed by Seiner for a CuAlNi specimen which was mechanically stabilized as a 
compound twin. We recall that, for typical lattice parameters, compound twins are also unable to form directly compatible interfaces with austenite so that the mechanical 
stabilization effect comes into play. We note that our methods, presented in this chapter, may be applicable to this case as well.
Nucleation at a corner also occurs for a homogeneously heated specimen consisting of a pure variant of martensite, but at a temperature significantly higher than for a corresponding specimen consisting of 
thermally induced martensite. Idealizing the latter by a simple laminate of martensite that is compatible with the austenite, one can intuitively understand this difference of temperatures. 
For the thermally induced martensite, nucleating a volume $V$ of austenite at a corner reduces the energy by $\delta V$ together with another term proportional to $V$ representing the energy of the twin 
interfaces in the laminate that are no longer present in the austenite. However, nucleating the same volume $V$ from a single variant using the construction in Lemma~\ref{prop:corner} again reduces the energy by $\delta V$, 
but there is an {\it increase} of energy required for the formation of the twin interfaces in the laminate interpolating between the austenite and the single variant. 
Thus it is energetically easier to nucleate the austenite in the thermally induced case. Of course the model considered in this paper ignores such interfacial energy contributions, 
and it would be interesting to be able to include them.

Incorporating interfacial energy would, however,  give rise to some nontrivial issues. For example, it is shown in Ball \& Crooks \cite{ballcrooks} that in a second gradient model of interfacial energy the austenite lies in a (shallow) potential well in $L^1$, rather than being unstable to the nucleation of an austenite-martensite interface as in our model. 

Another natural question concerns how our results are affected by a slight smoothing of the edges and corners of the specimen. A natural conjecture would be that the pure variant of martensite might then be stable, lying in a shallow potential well as in the case of small interfacial energy just described, rather than being unstable to nucleation of austenite at a corner as in our model.

Lastly, similar situations in which the incompatibility of gradients results in hysteresis are documented in other contexts, e.g.~Ball \& James~\cite{newBJ} or Ball, Chu \& 
James~\cite{ballchujames}. In the latter, though in a different way, the mathematical analysis argues that despite the existence of a state with lower energy than a certain 
martensitic variant, it is necessarily geometrically incompatible with it, giving rise to an energy barrier which keeps the specific martensitic variant stable.
\newpage

\section*{Appendix A : Symmetry relations between the cubic-to-orthorhombic variants}
\label{appendix1}

\noindent\textbf{Cubic symmetry group}

\noindent Below we give the 24 elements of the cubic symmetry group. The orthonormal vectors $\left\{i_{1},i_{2},i_{3}\right\}$ correspond to the cubic basis and $Q=Q\left[\phi,e\right]$ denotes a rotation by angle $\phi$ about the axis $e$.
\begin{eqnarray}
\mathbf{1}&=& i_{1}\otimes i_{1}+i_{2}\otimes i_{2}+i_{3}\otimes i_{3}\nonumber\\
Q_{1}&=&Q\left[2\pi /3,i_{1}+i_{2}+i_{3}\right]=i_{1}\otimes i_{3}+i_{2}\otimes i_{1}+i_{3}\otimes i_{2}\nonumber\\
Q_{2}&=&Q\left[-2\pi /3,i_{1}+i_{2}+i_{3}\right]=i_{1}\otimes i_{2}+i_{2}\otimes i_{3}+i_{3}\otimes i_{1}\nonumber\\
Q_{3}&=&Q\left[2\pi /3,-i_{1}+i_{2}+i_{3}\right]=-i_{1}\otimes i_{2}+i_{2}\otimes i_{3}-i_{3}\otimes i_{1}\nonumber\\
Q_{4}&=&Q\left[-2\pi /3,-i_{1}+i_{2}+i_{3}\right]=-i_{1}\otimes i_{3}-i_{2}\otimes i_{1}+i_{3}\otimes i_{2}\nonumber\\
Q_{5}&=&Q\left[2\pi /3,i_{1}-i_{2}+i_{3}\right]=-i_{1}\otimes i_{2}-i_{2}\otimes i_{3}+i_{3}\otimes i_{1}\nonumber\\
Q_{6}&=&Q\left[-2\pi /3,i_{1}-i_{2}+i_{3}\right]=i_{1}\otimes i_{3}-i_{2}\otimes i_{1}-i_{3}\otimes i_{2}\nonumber\\
Q_{7}&=&Q\left[2\pi /3,i_{1}+i_{2}-i_{3}\right]=i_{1}\otimes i_{2}-i_{2}\otimes i_{3}-i_{3}\otimes i_{1}\nonumber\\
Q_{8}&=&Q\left[-2\pi /3,i_{1}+i_{2}-i_{3}\right]=-i_{1}\otimes i_{3}+i_{2}\otimes i_{1}-i_{3}\otimes i_{2}\nonumber\\
Q_{9}&=&Q\left[\pi, i_{1}+i_{2}\right]=i_{1}\otimes i_{2}+i_{2}\otimes i_{1}-i_{3}\otimes i_{3}\nonumber\\
Q_{10}&=&Q\left[\pi, i_{1}-i_{2}\right]=-i_{1}\otimes i_{2}-i_{2}\otimes i_{1}-i_{3}\otimes i_{3}\nonumber\\
Q_{11}&=&Q\left[\pi, i_{1}+i_{3}\right]=i_{1}\otimes i_{3}-i_{2}\otimes i_{2}+i_{3}\otimes i_{1}\nonumber\\
Q_{12}&=&Q\left[\pi, i_{1}-i_{3}\right]=-i_{1}\otimes i_{3}-i_{2}\otimes i_{2}-i_{3}\otimes i_{1}\nonumber\\
Q_{13}&=&Q\left[\pi, i_{2}+i_{3}\right]=-i_{1}\otimes i_{1}+i_{2}\otimes i_{3}+i_{3}\otimes i_{2}\nonumber\\
Q_{14}&=&Q\left[\pi, i_{2}-i_{3}\right]=-i_{1}\otimes i_{1}-i_{2}\otimes i_{3}-i_{3}\otimes i_{2}\nonumber\\
Q_{15}&=&Q\left[\pi, i_{1}\right]=i_{1}\otimes i_{1}-i_{2}\otimes i_{2}-i_{3}\otimes i_{3}\nonumber\\
Q_{16}&=&Q\left[\pi, i_{2}\right]=-i_{1}\otimes i_{1}+i_{2}\otimes i_{2}-i_{3}\otimes i_{3}\nonumber\\
Q_{17}&=&Q\left[\pi, i_{3}\right]=-i_{1}\otimes i_{1}-i_{2}\otimes i_{2}+i_{3}\otimes i_{3}\nonumber\\
Q_{18}&=&Q\left[\pi /2, i_{1}\right]=i_{1}\otimes i_{1}-i_{2}\otimes i_{3}+i_{3}\otimes i_{2}\nonumber\\
Q_{19}&=&Q\left[-\pi /2, i_{1}\right]=i_{1}\otimes i_{1}+i_{2}\otimes i_{3}-i_{3}\otimes i_{2}\nonumber\\
Q_{20}&=&Q\left[\pi /2, i_{2}\right]=i_{1}\otimes i_{3}+i_{2}\otimes i_{2}-i_{3}\otimes i_{1}\nonumber\\
Q_{21}&=&Q\left[-\pi /2, i_{2}\right]=-i_{1}\otimes i_{3}+i_{2}\otimes i_{2}+i_{3}\otimes i_{1}\nonumber\\
Q_{22}&=&Q\left[\pi /2, i_{3}\right]=-i_{1}\otimes i_{2}+i_{2}\otimes i_{1}+i_{3}\otimes i_{3}\nonumber\\
Q_{23}&=&Q\left[-\pi /2, i_{3}\right]=i_{1}\otimes i_{2}-i_{2}\otimes i_{1}+i_{3}\otimes i_{3}\nonumber
\end{eqnarray}

The following relations between the cubic-to-orthorhombic variants can be found in Hane~\cite{41}.
\newpage

\begin{table}[ht]
\centering
\begin{tabular}{|c|c|c|c|c|c|c|c|c|c|c|c|c|}
\hline  & $\mathbf{1}$ & $Q_{1}$ & $Q_{2}$ & $Q_{3}$ & $Q_{4}$ & $Q_{5}$ & $Q_{6}$ & $Q_{7}$ & $Q_{8}$ & $Q_{9}$ & $Q_{10}$ & $Q_{11}$ \\ 
\hline $U_{1}$ &1 &3 &5 &6 &4 &5 &4 &6 &3 &4 &3 &6 \\ 
\hline $U_{2}$ &2 &4 &6 &5 &3 &6 &3 &5 &4 &3 &4 &5 \\ 
\hline $U_{3}$ &3 &5 &1 &2 &5 &2 &6 &1 &6 &2 &1 &3 \\ 
\hline $U_{4}$ &4 &6 &2 &1 &6 &1 &5 &2 &5 &1 &2 &4 \\ 
\hline $U_{5}$ &5 &1 &3 &3 &2 &4 &1 &4 &2 &5 &5 &2 \\ 
\hline $U_{6}$ &6 &2 &4 &4 &1 &3 &2 &3 &1 &6 &6 &1 \\ 
\hline  & $Q_{12}$ & $Q_{13}$ & $Q_{14}$ & $Q_{15}$ & $Q_{16}$ & $Q_{17}$ & $Q_{18}$ & $Q_{19}$ & $Q_{20}$ & $Q_{21}$ & $Q_{22}$ & $Q_{23}$ \\ 
\hline $U_{1}$ &5 &1 &1 &1 &2 &2 &2 &2 &5 &6 &4 &3 \\ 
\hline $U_{2}$ &6 &2 &2 &2 &1 &1 &1 &1 &6 &5 &3 &4 \\ 
\hline $U_{3}$ &3 &6 &5 &4 &3 &4 &6 &5 &4 &4 &1 &2 \\ 
\hline $U_{4}$ &4 &5 &6 &3 &4 &3 &5 &6 &3 &3 &2 &1 \\ 
\hline $U_{5}$ &1 &4 &3 &6 &6 &5 &3 &4 &2 &1 &6 &6 \\ 
\hline $U_{6}$ &2 &3 &4 &5 &5 &6 &4 &3 &1 &2 &5 &5 \\ 
\hline
\end{tabular}
\label{table}
\caption{Symmetry relations amongst the variants of a cubic-to-orthorhombic transformation, where the rotations of the cubic symmetry group are given above.}
\end{table}

As an example of how Table~1 is meant to be used, consider the variant $U_{3}$ and the rotation $Q_{14}$; the number corresponding to these is 5. Then, the table says that $Q_{14}U_{3}Q^{T}_{14}=U_{5}$.

\section*{Appendix B : Twin and habit plane elements for CuAlNi}
\label{appendix4}
In this Appendix we give the promised missing details of the proof of Lemma~\ref{prop:corner}. In order to do so, we need to summarize the results concerning the twin and habit plane shears and normals for all variant pairs $(l,s)$ in the cubic-to-orthorhombic transition of CuAlNi, i.e. the vectors $a$, $n$ (twin plane) and the vectors $b$, $m$ (habit plane) such that
\begin{eqnarray}
\label{eq:appendixb1}
 QU_l-U_s&=&a\otimes n,\\
\label{eq:appendixb2}
 U_s+\lambda a\otimes n&=&R\left(\mathbf{1}+b\otimes m\right).
\end{eqnarray}
We use the following notation:
\begin{eqnarray*}
 a^{I},\,n^{I}&:&\mbox{Type-I twin plane shear and normal respectively;}\\
 a^{II},\,n^{II}&:&\mbox{Type-II twin plane shear and normal respectively;}\\
 b^{+}_{1},\,m^{+}_{1}&:&\mbox{habit plane shear and normal respectively, using the Type-II twinning}\\
&&\mbox{elements $a^{II}$, $n^{II}$, $\kappa=+1$ as in~\cite{1}}\\
&&\mbox{and volume fraction $\lambda\in(0,1/2)$;}\\
 b^{-}_{1},\,m^{-}_{1}&:&\mbox{habit plane shear and normal respectively, using the Type-II twinning}\\
&&\mbox{elements $a^{II}$, $n^{II}$, $\kappa=-1$ as in~\cite{1}}\\
&&\mbox{and volume fraction $\lambda\in(0,1/2)$;}\\
 b^{+}_{2},\,m^{+}_{2}&:&\mbox{habit plane shear and normal respectively, using the Type-II twinning}\\
&&\mbox{elements $a^{II}$, $n^{II}$, $\kappa=+1$ as in~\cite{1}}\\
&&\mbox{and volume fraction $\lambda\in(1/2,1)$;}\\
 b^{-}_{2},\,m^{-}_{2}&:&\mbox{habit plane shear and normal respectively, using the Type-II twinning}\\
&&\mbox{elements $a^{II}$, $n^{II}$, $\kappa=-1$ as in~\cite{1}}\\
&&\mbox{and volume fraction $\lambda\in(1/2,1)$.}
\end{eqnarray*}
The Type-I and Type-II twinning elements can also be found in Hane~\cite{41}. The algebra for the habit plane elements of the austenite - Type-II twinned martensite interfaces becomes too involved for explicit analysis 
and we calculate the components numerically using \textit{Mathematica}.
The procedure involved in the numerical calculation simply verifies the hypotheses for the existence of a solution to the habit plane equation and calculates the solutions directly using the formulae from~\cite{1}. 
We note that the lattice parameters of Seiner's specimen are given by $\alpha=1.06372$, $\beta=0.91542$ and $\gamma=1.02368$. For these parameters and the Type-II twinning elements of any variant pair, $\lambda^{\ast}=0.300782\in(0,1/2)$ and
\[\delta =-2.37742<-2,\:\:\eta=0.0091991>0,\] 
where $\lambda^{\ast}$, $\delta$, $\eta$ are as in~\cite{1}, i.e.~for each variant pair there are four distinct solutions to the habit plane equation using the Type-II twinning elements as above.

In this Appendix, we do not present the twin plane elements for compound twins as, for our lattice parameters, these cannot form compatible interfaces with austenite (see Bhattacharya~\cite{45} for the appropriate conditions on the lattice parameters allowing for austenite - compound twinned martensite interfaces)
and hence cannot be used for the construction in Lemma~\ref{prop:corner} of the microstructure reducing the energy at a corner.
Also, we do not give the habit plane elements for austenite - Type-I twinned martensite interfaces as the Type-I twin planes belong to the family of crystallographically equivalent planes $\left\{110\right\}$ and the corresponding
twin plane normals are perpendicular to the edges of Seiner's specimen (see Table~2), i.e.~these cannot be used for our construction in Lemma~\ref{prop:corner} either.

Table~2 gives the Type-I twin plane normal $n^{I}$ and twin plane shear $a^{I}$ for the variant pair $(l,s)$. The components $u_1$, $u_2$, $u_3$ of the shears $a^{I}$ are given by
\[\left(\begin{array}{c} u_1\\ u_2\\ u_3\end{array}\right)=\frac{\sqrt{2}}{2\alpha^2\gamma^2 + \beta^2(\alpha^2 + \gamma^2)}\left(\begin{array}{c}
\frac{\alpha+\gamma}{2}\left(4\alpha\gamma\beta^2 - 2\alpha^2\gamma^2 - \beta^2(\alpha^2 + \gamma^2)\right)\\
\beta\left(\beta^2(\alpha^2 + \gamma^2) - 2\alpha^2\gamma^2\right)\\
\frac{\gamma-\alpha}{2}\left(4\alpha\gamma\beta^2 + 2\alpha^2\gamma^2 + \beta^2(\alpha^2 + \gamma^2)\right)                                                            
\end{array}\right).\]

Table~3 gives the Type-II twin plane normal $n^{II}$ and twin plane shear $a^{II}$ for the variant pair $(l,s)$. The components $t_1$, $t_2$ of the normals $n^{II}$ and $v_1$, $v_2$, $v_3$ of the shears $a^{II}$ are given by
\begin{eqnarray*}
 \left(\begin{array}{c} t_1\\ t_2\end{array}\right)&=&\frac{1}{\sqrt{8\beta^2(\beta^2-\alpha^2-\gamma^2)+6\alpha^4-4\alpha^2\gamma^2+6\gamma^2}}\left(\begin{array}{c}
2\beta^2-\alpha^2-\gamma^2\\                                                           
2(\gamma^2-\alpha^2)\end{array}\right),\\
\left(\begin{array}{c} v_1\\ v_2\\ v_3\end{array}\right)&=&\frac{\sqrt{8\beta^2(\beta^2-\alpha^2-\gamma^2)+6\alpha^4-4\alpha^2\gamma^2+6\gamma^2}}{\alpha^2+\gamma^2+2\beta^2}\left(\begin{array}{c}
\frac{\alpha + \gamma}{2}\\
-\beta\\                                                          
\frac{\alpha - \gamma}{2}\end{array}\right).
\end{eqnarray*}

\begin{table}[ht]
\centering
\begin{tabular}{|c|c|c|c|c|c|}
\hline  pair & $\sqrt{2}n^{I}$ & $a^{I}$ & pair & $\sqrt{2}n^{I}$ & $a^{I}$ \\
\hline $(1,3)$ & $\left(1,-1,0\right)$ & $\left(u_1,u_2,u_3\right)$ & $(3,1)$ & $\left(1,-1,0\right)$ & $\left(-u_2,-u_1,-u_3\right)$ \\
\hline $(1,4)$ & $\left(1,1,0\right)$ & $\left(u_1,-u_2,-u_3\right)$ & $(4,1)$ & $\left(1,1,0\right)$ & $\left(-u_2,u_1,u_3\right)$ \\
\hline $(1,5)$ & $\left(1,0,-1\right)$ & $\left(u_1,u_3,u_2\right)$ & $(5,1)$ & $\left(1,0,-1\right)$ & $\left(-u_2,-u_3,-u_1\right)$ \\ 
\hline $(1,6)$ & $\left(1,0,1\right)$ & $\left(u_1,-u_3,-u_2\right)$ & $(6,1)$ & $\left(1,0,1\right)$ & $\left(-u_2,u_3,u_1\right)$ \\
\hline $(2,3)$ & $\left(1,1,0\right)$ & $\left(u_1,-u_2,u_3\right)$ & $(3,2)$ & $\left(1,1,0\right)$ & $\left(-u_2,u_1,-u_3\right)$ \\
\hline $(2,4)$ & $\left(1,-1,0\right)$ & $\left(u_1,u_2,-u_3\right)$ & $(4,2)$ & $\left(1,-1,0\right)$ & $\left(-u_2,-u_1,u_3\right)$ \\
\hline $(2,5)$ & $\left(1,0,1\right)$ & $\left(u_1,u_3,-u_2\right)$ & $(5,2)$ & $\left(1,0,1\right)$ & $\left(-u_2,-u_3,u_1\right)$ \\
\hline $(2,6)$ & $\left(1,0,-1\right)$ & $\left(u_1,-u_3,u_2\right)$ & $(6,2)$ & $\left(1,0,-1\right)$ & $\left(-u_2,u_3,-u_1\right)$ \\
\hline $(3,5)$ & $\left(0,1,-1\right)$ & $\left(u_3,u_1,u_2\right)$ & $(5,3)$ & $\left(0,1,-1\right)$ & $\left(-u_3,-u_2,-u_1\right)$ \\
\hline $(3,6)$ & $\left(0,1,1\right)$ & $\left(-u_3,u_1,-u_2\right)$ & $(6,3)$ & $\left(0,1,1\right)$ & $\left(u_3,-u_2,u_1\right)$ \\
\hline $(4,5)$ & $\left(0,1,1\right)$ & $\left(u_3,u_1,-u_2\right)$ & $(5,4)$ & $\left(0,1,1\right)$ & $\left(-u_3,-u_2,u_1\right)$ \\ 
\hline $(4,6)$ & $\left(0,1,-1\right)$ & $\left(-u_3,-u_1,u_2\right)$ & $(6,4)$ & $\left(0,1,-1\right)$ & $\left(u_3,-u_2,-u_1\right)$ \\
\hline
\end{tabular}
\label{tableb1}
\caption{Components of the Type-I twin plane normal and shear for the different variant pairs in a cubic-to-orthorhombic transition~\cite{41}.}
\end{table}
\begin{table}[ht]
\centering
\begin{tabular}{|c|c|c|c|c|c|}
\hline  pair & $n^{II}$ & $a^{II}$ & pair & $n^{II}$ & $a^{II}$ \\
\hline $(1,3)$ & $\left(t_1,t_1,t_2\right)$ & $\left(v_1,v_2,v_3\right)$ & $(3,1)$ & $\left(t_1,t_1,t_2\right)$ & $\left(v_2,v_1,v_3\right)$ \\
\hline $(1,4)$ & $\left(-t_1,t_1,t_2\right)$ & $\left(-v_1,v_2,v_3\right)$ & $(4,1)$ & $\left(-t_1,t_1,t_2\right)$ & $\left(-v_2,v_1,v_3\right)$ \\
\hline $(1,5)$ & $\left(t_1,t_2,t_1\right)$ & $\left(v_1,v_3,v_2\right)$ & $(5,1)$ & $\left(t_1,t_2,t_1\right)$ & $\left(v_2,v_3,v_1\right)$ \\ 
\hline $(1,6)$ & $\left(-t_1,t_2,t_1\right)$ & $\left(-v_1,v_3,v_2\right)$ & $(6,1)$ & $\left(-t_1,t_2,t_1\right)$ & $\left(-v_2,v_3,v_1\right)$ \\
\hline $(2,3)$ & $\left(t_1,-t_1,t_2\right)$ & $\left(v_1,-v_2,v_3\right)$ & $(3,2)$ & $\left(t_1,-t_1,t_2\right)$ & $\left(v_2,-v_1,v_3\right)$ \\
\hline $(2,4)$ & $\left(t_1,t_1,-t_2\right)$ & $\left(v_1,v_2,-v_3\right)$ & $(4,2)$ & $\left(t_1,t_1,-t_2\right)$ & $\left(v_2,v_1,-v_3\right)$ \\
\hline $(2,5)$ & $\left(t_1,t_2,-t_1\right)$ & $\left(v_1,v_3,-v_2\right)$ & $(5,2)$ & $\left(t_1,t_2,-t_1\right)$ & $\left(v_2,v_3,-v_1\right)$ \\
\hline $(2,6)$ & $\left(t_1,-t_2,t_1\right)$ & $\left(v_1,-v_3,v_2\right)$ & $(6,2)$ & $\left(t_1,-t_2,t_1\right)$ & $\left(v_2,-v_3,v_1\right)$ \\
\hline $(3,5)$ & $\left(t_2,t_1,t_1\right)$ & $\left(v_3,v_1,v_2\right)$ & $(5,3)$ & $\left(t_2,t_1,t_1\right)$ & $\left(v_3,v_2,v_1\right)$ \\
\hline $(3,6)$ & $\left(t_2,-t_1,t_1\right)$ & $\left(v_3,-v_1,v_2\right)$ & $(6,3)$ & $\left(t_2,-t_1,t_1\right)$ & $\left(v_3,-v_2,v_1\right)$ \\
\hline $(4,5)$ & $\left(t_2,t_1,-t_1\right)$ & $\left(v_3,v_1,-v_2\right)$ & $(5,4)$ & $\left(t_2,t_1,-t_1\right)$ & $\left(v_3,v_2,-v_1\right)$ \\ 
\hline $(4,6)$ & $\left(-t_2,t_1,t_1\right)$ & $\left(-v_3,v_1,v_2\right)$ & $(6,4)$ & $\left(-t_2,t_1,t_1\right)$ & $\left(-v_3,v_2,v_1\right)$ \\
\hline
\end{tabular}
\label{tableb2}
\caption{Components of the Type-II twin plane normal and shear for the different variant pairs in a cubic-to-orthorhombic transition~\cite{41}.}
\end{table}
We note that for Seiner's CuAlNi specimen the components of the twinning elements become:
\begin{equation*}
\left(\begin{array}{c} u_1\\ u_2\\ u_3\end{array}\right)=\left(\begin{array}{c}
0.197977\\
-0.173644\\                                                          
0.00379754\end{array}\right),
\end{equation*}
\begin{equation*}
\begin{array}{cc}
\left(\begin{array}{c} t_1\\ t_2\end{array}\right)=\left(\begin{array}{c}
-0.688388\\                                                           
-0.228571\end{array}\right),&
\left(\begin{array}{c} v_1\\ v_2\\ v_3\end{array}\right)=\left(\begin{array}{c}
0.197977\\
-0.173644\\                                                          
0.00379754\end{array}\right).
\end{array}
\end{equation*}

For the habit plane elements, the calculations are purely numerical. In order to shorten the otherwise long tables that follow, let us write
\begin{equation*}
 \begin{array}{ccc}
  s_1=0.141221,& s_2=0.668151,& s_3=0.730501,\\
  s_4=0.261549,& s_5=0.727152,& s_6=0.634699,\\
  z_1=0.0244382,& z_2=0.0728267,& z_3=0.0575181,\\
  z_4=0.0123419,& z_5=0.0674388,& z_6=0.0671488.
 \end{array}
\end{equation*}

\begin{table}[ht]
\centering
\begin{tabular}{|c|c|c|c|c|c|}
\hline  pair & $m^{+}_{1}$ & $b^{+}_{1}$ & pair & $m^{+}_{1}$ & $b^{+}_{1}$ \\
\hline $(1,3)$ & $\left(-s_1, s_2, -s_3\right)$ & $\left(-z_1, -z_2, -z_3\right)$ & $(3,1)$ & $\left(-s_2, s_1, s_3\right)$ & $\left(z_2, z_1, z_3\right)$ \\
\hline $(1,4)$ & $\left(s_4, -s_5, -s_6\right)$ & $\left(z_4, z_5, -z_6\right)$ & $(4,1)$ & $\left(-s_5, s_4, s_6\right)$ & $\left(z_5, z_4, z_6\right)$ \\
\hline $(1,5)$ & $\left(s_1, s_3, -s_2\right)$ & $\left(z_1, z_3, z_2\right)$ & $(5,1)$ & $\left(-s_5, -s_6, -s_4\right)$ & $\left(z_5, -z_6, -z_4\right)$ \\ 
\hline $(1,6)$ & $\left(s_4, -s_6, -s_5\right)$ & $\left(z_4, -z_6, z_5\right)$ & $(6,1)$ & $\left(-s_2, -s_3, -s_1\right)$ & $\left(z_2, -z_3, -z_1\right)$ \\
\hline $(2,3)$ & $\left(s_4, -s_5, s_6\right)$ & $\left(z_4, z_5, z_6\right)$ & $(3,2)$ & $\left(-s_2, -s_1, s_3\right)$ & $\left(z_2, -z_1, z_3\right)$ \\
\hline $(2,4)$ & $\left(-s_1, s_2, s_3\right)$ & $\left(-z_1, -z_2, z_3\right)$ & $(4,2)$ & $\left(-s_5, -s_4, s_6\right)$ & $\left(z_5, -z_4, z_6\right)$ \\
\hline $(2,5)$ & $\left(s_4, s_6, -s_5\right)$ & $\left(z_4, z_6, z_5\right)$ & $(5,2)$ & $\left(-s_5, -s_6, s_4\right)$ & $\left(z_5, -z_6, z_4\right)$ \\
\hline $(2,6)$ & $\left(s_1, -s_3, -s_2\right)$ & $\left(z_1, -z_3, z_2\right)$ & $(6,2)$ & $\left(-s_2, -s_3, s_1\right)$ & $\left(z_2, -z_3, z_1\right)$ \\
\hline $(3,5)$ & $\left(s_3, s_1, -s_2\right)$ & $\left(z_3, z_1, z_2\right)$ & $(5,3)$ & $\left(s_3, -s_2, s_1\right)$ & $\left(z_3, z_2, z_1\right)$ \\
\hline $(3,6)$ & $\left(s_3, -s_1, -s_2\right)$ & $\left(z_3, -z_1, z_2\right)$ & $(6,3)$ & $\left(-s_6, s_5, -s_4\right)$ & $\left(-z_6, -z_5, -z_4\right)$ \\
\hline $(4,5)$ & $\left(s_6, s_4, -s_5\right)$ & $\left(z_6, z_4, z_5\right)$ & $(5,4)$ & $\left(s_3, -s_2, -s_1\right)$ & $\left(z_3, z_2, -z_1\right)$ \\ 
\hline $(4,6)$ & $\left(s_6, -s_4, -s_5\right)$ & $\left(z_6, -z_4, z_5\right)$ & $(6,4)$ & $\left(-s_6, s_5, s_4\right)$ & $\left(-z_6, -z_5, z_4\right)$ \\
\hline
\end{tabular}
\label{tableb3}
\caption{Components of the habit plane normal $m^{+}_{1}$ and shear $b^{+}_{1}$ for the different variant pairs in the cubic-to-orthorhombic transition of Seiner's CuAlNi specimen.}
\end{table}

\begin{table}[ht]
\centering
\begin{tabular}{|c|c|c|c|c|c|}
\hline  pair & $m^{-}_{1}$ & $b^{-}_{1}$ & pair & $m^{-}_{1}$ & $b^{-}_{1}$ \\
\hline $(1,3)$ & $\left(s_4, s_5, s_6\right)$ & $\left(z_4, -z_5, z_6\right)$ & $(3,1)$ & $\left(-s_5, -s_4, -s_6\right)$ & $\left(z_5, -z_4, -z_6\right)$ \\
\hline $(1,4)$ & $\left(-s_1, -s_2, s_3\right)$ & $\left(-z_1, z_2, z_3\right)$ & $(4,1)$ & $\left(-s_2, -s_1, -s_3\right)$ & $\left(z_2, -z_1, -z_3\right)$ \\
\hline $(1,5)$ & $\left(-s_4, -s_6, -s_5\right)$ & $\left(-z_4, -z_6, z_5\right)$ & $(5,1)$ & $\left(-s_2, s_3, s_1\right)$ & $\left(z_2, z_3, z_1\right)$ \\ 
\hline $(1,6)$ & $\left(-s_1, s_3, -s_2\right)$ & $\left(-z_1, z_3, z_2\right)$ & $(6,1)$ & $\left(-s_5, s_6, s_4\right)$ & $\left(z_5, z_6, z_4\right)$ \\
\hline $(2,3)$ & $\left(-s_1, -s_2, -s_3\right)$ & $\left(-z_1, z_2, -z_3\right)$ & $(3,2)$ & $\left(-s_5, s_4, -s_6\right)$ & $\left(z_5, z_4, -z_6\right)$ \\
\hline $(2,4)$ & $\left(s_4, s_5, -s_6\right)$ & $\left(z_4, -z_5, -z_6\right)$ & $(4,2)$ & $\left(-s_2, s_1, -s_3\right)$ & $\left(z_2, z_1, -z_3\right)$ \\
\hline $(2,5)$ & $\left(-s_1, -s_3, -s_2\right)$ & $\left(-z_1, -z_3, z_2\right)$ & $(5,2)$ & $\left(-s_2, s_3, -s_1\right)$ & $\left(z_2, z_3, -z_1\right)$ \\
\hline $(2,6)$ & $\left(-s_4, s_6, -s_5\right)$ & $\left(-z_4, z_6, z_5\right)$ & $(6,2)$ & $\left(-s_5, s_6, -s_4\right)$ & $\left(z_5, z_6, -z_4\right)$ \\
\hline $(3,5)$ & $\left(-s_6, -s_4, -s_5\right)$ & $\left(-z_6, -z_4, z_5\right)$ & $(5,3)$ & $\left(-s_6, -s_5, -s_4\right)$ & $\left(-z_6, z_5, -z_4\right)$ \\
\hline $(3,6)$ & $\left(-s_6, s_4, -s_5\right)$ & $\left(-z_6, z_4, z_5\right)$ & $(6,3)$ & $\left(s_3, s_2, s_1\right)$ & $\left(z_3, -z_2, z_1\right)$ \\
\hline $(4,5)$ & $\left(-s_3, -s_1, -s_2\right)$ & $\left(-z_3, -z_1, z_2\right)$ & $(5,4)$ & $\left(-s_6, -s_5, s_4\right)$ & $\left(-z_6, z_5, z_4\right)$ \\ 
\hline $(4,6)$ & $\left(-s_3, s_1, -s_2\right)$ & $\left(-z_3, z_1, z_2\right)$ & $(6,4)$ & $\left(s_3, s_2, -s_1\right)$ & $\left(z_3, -z_2, -z_1\right)$ \\
\hline
\end{tabular}
\label{tableb4}
\caption{Components of the habit plane normal $m^{-}_{1}$ and shear $b^{-}_{1}$ for the different variant pairs in the cubic-to-orthorhombic transition of Seiner's CuAlNi specimen.}
\end{table}

\begin{table}[ht]
\centering
\begin{tabular}{|c|c|c|c|c|c|}
\hline  pair & $m^{+}_{2}$ & $b^{+}_{2}$ & pair & $m^{+}_{2}$ & $b^{+}_{2}$ \\
\hline $(1,3)$ & $\left(-s_2, s_1, s_3\right)$ & $\left(z_2, z_1, z_3\right)$ & $(3,1)$ & $\left(s_4, s_5, s_6\right)$ & $\left(z_4, -z_5, z_6\right)$ \\
\hline $(1,4)$ & $\left(-s_5, s_4, s_6\right)$ & $\left(z_5, z_4, z_6\right)$ & $(4,1)$ & $\left(-s_1, -s_2, s_3\right)$ & $\left(-z_1, z_2, z_3\right)$ \\
\hline $(1,5)$ & $\left(-s_5, -s_6, -s_4\right)$ & $\left(z_5, -z_6, -z_4\right)$ & $(5,1)$ & $\left(s_1, s_3, -s_2\right)$ & $\left(z_1, z_3, z_2\right)$ \\ 
\hline $(1,6)$ & $\left(-s_2, -s_3, -s_1\right)$ & $\left(z_2, -z_3, -z_1\right)$ & $(6,1)$ & $\left(s_4, -s_6, -s_5\right)$ & $\left(z_4, -z_6, z_5\right)$ \\
\hline $(2,3)$ & $\left(-s_2, -s_1, s_3\right)$ & $\left(z_2, -z_1, z_3\right)$ & $(3,2)$ & $\left(s_4, -s_5, s_6\right)$ & $\left(z_4, z_5, z_6\right)$ \\
\hline $(2,4)$ & $\left(-s_5, -s_4, s_6\right)$ & $\left(z_5, -z_4, z_6\right)$ & $(4,2)$ & $\left(-s_1, s_2, s_3\right)$ & $\left(-z_1, -z_2, z_3\right)$ \\
\hline $(2,5)$ & $\left(-s_5, -s_6, s_4\right)$ & $\left(z_5, -z_6, z_4\right)$ & $(5,2)$ & $\left(s_1, s_3, s_2\right)$ & $\left(z_1, z_3, -z_2\right)$ \\
\hline $(2,6)$ & $\left(-s_2, -s_3, s_1\right)$ & $\left(z_2, -z_3, z_1\right)$ & $(6,2)$ & $\left(s_4, -s_6, s_5\right)$ & $\left(z_4, -z_6, -z_5\right)$ \\
\hline $(3,5)$ & $\left(s_3, -s_2, s_1\right)$ & $\left(z_3, z_2, z_1\right)$ & $(5,3)$ & $\left(s_3, s_1, -s_2\right)$ & $\left(z_3, z_1, z_2\right)$ \\
\hline $(3,6)$ & $\left(-s_6, s_5, -s_4\right)$ & $\left(-z_6, -z_5, -z_4\right)$ & $(6,3)$ & $\left(s_3, -s_1, -s_2\right)$ & $\left(z_3, -z_1, z_2\right)$ \\
\hline $(4,5)$ & $\left(s_3, -s_2, -s_1\right)$ & $\left(z_3, z_2, -z_1\right)$ & $(5,4)$ & $\left(s_6, s_4, -s_5\right)$ & $\left(z_6, z_4, z_5\right)$ \\ 
\hline $(4,6)$ & $\left(-s_6, s_5, s_4\right)$ & $\left(-z_6, -z_5, z_4\right)$ & $(6,4)$ & $\left(s_6, -s_4, -s_5\right)$ & $\left(z_6, -z_4, z_5\right)$ \\
\hline
\end{tabular}
\label{tableb5}
\caption{Components of the habit plane normal $m^{+}_{2}$ and shear $b^{+}_{2}$ for the different variant pairs in the cubic-to-orthorhombic transition of Seiner's CuAlNi specimen.}
\end{table}
\begin{table}[ht]
\centering
\begin{tabular}{|c|c|c|c|c|c|}
\hline  pair & $m^{-}_{2}$ & $b^{-}_{2}$ & pair & $m^{-}_{2}$ & $b^{-}_{2}$ \\
\hline $(1,3)$ & $\left(-s_5, -s_4, -s_6\right)$ & $\left(z_5, -z_4, -z_6\right)$ & $(3,1)$ & $\left(-s_1, s_2, -s_3\right)$ & $\left(-z_1, -z_2, -z_3\right)$ \\
\hline $(1,4)$ & $\left(-s_2, -s_1, -s_3\right)$ & $\left(z_2, -z_1, -z_3\right)$ & $(4,1)$ & $\left(s_4, -s_5, -s_6\right)$ & $\left(z_4, z_5, -z_6\right)$ \\
\hline $(1,5)$ & $\left(-s_2, s_3, s_1\right)$ & $\left(z_2, z_3, z_1\right)$ & $(5,1)$ & $\left(-s_4, -s_6, -s_5\right)$ & $\left(-z_4, -z_6, z_5\right)$ \\ 
\hline $(1,6)$ & $\left(-s_5, s_6, s_4\right)$ & $\left(z_5, z_6, z_4\right)$ & $(6,1)$ & $\left(-s_1, s_3, -s_2\right)$ & $\left(-z_1, z_3, z_2\right)$ \\
\hline $(2,3)$ & $\left(-s_5, s_4, -s_6\right)$ & $\left(z_5, z_4, -z_6\right)$ & $(3,2)$ & $\left(-s_1, -s_2, -s_3\right)$ & $\left(-z_1, z_2, -z_3\right)$ \\
\hline $(2,4)$ & $\left(-s_2, s_1, -s_3\right)$ & $\left(z_2, z_1, -z_3\right)$ & $(4,2)$ & $\left(s_4, s_5, -s_6\right)$ & $\left(z_4, -z_5, -z_6\right)$ \\
\hline $(2,5)$ & $\left(-s_2, s_3, -s_1\right)$ & $\left(z_2, z_3, -z_1\right)$ & $(5,2)$ & $\left(-s_4, -s_6, s_5\right)$ & $\left(-z_4, -z_6, -z_5\right)$ \\
\hline $(2,6)$ & $\left(-s_5, s_6, -s_4\right)$ & $\left(z_5, z_6, -z_4\right)$ & $(6,2)$ & $\left(-s_1, s_3, s_2\right)$ & $\left(-z_1, z_3, -z_2\right)$ \\
\hline $(3,5)$ & $\left(-s_6, -s_5, -s_4\right)$ & $\left(-z_6, z_5, -z_4\right)$ & $(5,3)$ & $\left(-s_6, -s_4, -s_5\right)$ & $\left(-z_6, -z_4, z_5\right)$ \\
\hline $(3,6)$ & $\left(s_3, s_2, s_1\right)$ & $\left(z_3, -z_2, z_1\right)$ & $(6,3)$ & $\left(-s_6, s_4, -s_5\right)$ & $\left(-z_6, z_4, z_5\right)$ \\
\hline $(4,5)$ & $\left(-s_6, -s_5, s_4\right)$ & $\left(-z_6, z_5, z_4\right)$ & $(5,4)$ & $\left(-s_3, -s_1, -s_2\right)$ & $\left(-z_3, -z_1, z_2\right)$ \\ 
\hline $(4,6)$ & $\left(s_3, s_2, -s_1\right)$ & $\left(z_3, -z_2, -z_1\right)$ & $(6,4)$ & $\left(-s_3, s_1, -s_2\right)$ & $\left(-z_3, z_1, z_2\right)$ \\
\hline
\end{tabular}
\label{tableb6}
\caption{Components of the habit plane normal $m^{-}_{2}$ and shear $b^{-}_{2}$ for the different variant pairs in the cubic-to-orthorhombic transition of Seiner's CuAlNi specimen.}
\end{table}
We note that the above components of the habit plane elements are only approximate. However, they have been verified by checking the relation
\[(U_s+\lambda n\otimes a)(U_s+\lambda a\otimes n)=(\mathbf{1}+m\otimes b)(\mathbf{1}+b\otimes m).\]

We can now verify that it is indeed possible to construct the microstructure at the corner that reduces the energy as in Lemma~\ref{prop:corner}. To see this we need to consider each possible value
of $s\in\left\{1,\ldots,6\right\}$ and check the corners at which the required microstructure can be constructed.

As in the proof of Lemma~\ref{prop:corner}, suppose that the coordinate system has been chosen in such a way that the edges of $\Omega$ are parallel to the axes and each corner of $\Omega$ lies in a different octant.
Let us write $x=(x_1,x_2,x_3)^T$ for the coordinates of the point $x\in\mathbb{R}^3$ in the standard basis of $\mathbb{R}^3$ and denote the octants as follows:
\begin{eqnarray*}
 O_1&:=&\left\{x\in\mathbb{R}^3:x_1>0,\,x_2>0,\,x_3>0\right\},\:\:O_5=-O_1\\
 O_2&:=&\left\{x\in\mathbb{R}^3:x_1<0,\,x_2>0,\,x_3>0\right\},\:\:O_6=-O_2\\
 O_3&:=&\left\{x\in\mathbb{R}^3:x_1>0,\,x_2<0,\,x_3>0\right\},\:\:O_7=-O_3\\
 O_4&:=&\left\{x\in\mathbb{R}^3:x_1>0,\,x_2>0,\,x_3<0\right\},\:\:O_8=-O_4.
\end{eqnarray*}
The results are summarized in Tables~8 and~9 for the habit plane solutions using $\lambda=\lambda^{\ast}\in(0,1/2)$ and $\lambda=1-\lambda^{\ast}\in(1/2,1)$ respectively and are explained below:
\begin{table}[ht]
\centering
\begin{tabular}{|c|c|c|c|c|c|c|c|}
\hline  \backslashbox{$s$}{$l$} & $1$ & $2$ & $3$ & $4$ & $5$ & $6$ & $O$ \\
\hline $1$ & - & - & $m^{-}_{1}$ & $-m^{+}_{1}$ & $m^{+}_{1}$ & $-m^{-}_{1}$ & $1,2,5,6$ \\
\hline $2$ & - & - & $m^{-}_{1}$ & $m^{+}_{1}$ & $m^{+}_{1}$ & $m^{-}_{1}$ & $3,4,7,8$ \\
\hline $3$ & $-m^{-}_{1}$ & $-m^{+}_{1}$ & - & - & $m^{-}_{1}$ & $m^{+}_{1}$ & $1,3,5,7$ \\ 
\hline $4$ & $m^{+}_{1}$ & $-m^{-}_{1}$ & - & - & $m^{-}_{1}$ & $-m^{+}_{1}$ & $2,4,6,8$ \\
\hline $5$ & $m^{-}_{1}$ & $-m^{+}_{1}$ & $m^{-}_{1}$ & $-m^{+}_{1}$ & - & - & $1,4,5,8$ \\
\hline $6$ & $m^{+}_{1}$ & $m^{-}_{1}$ & $m^{-}_{1}$ & $m^{+}_{1}$ & - & - & $2,3,6,7$ \\
\hline
\end{tabular}
\label{tableb7}
\caption{For each $s\in\left\{1,\ldots,6\right\}$, the table gives all the habit plane normals with $\lambda\in(0,1/2)$ that lie in the same octant as the Type-II twin normal for the possible values of $l$. The final column gives the octants in which the construction at the corner is possible.}
\end{table}
\begin{table}[ht]
\centering
\begin{tabular}{|c|c|c|c|c|c|c|c|}
\hline   \backslashbox{$s$}{$l$} & $1$ & $2$ & $3$ & $4$ & $5$ & $6$ & $O$ \\
\hline $1$ & - & - & $-m^{+}_{2}$ & $m^{-}_{2}$ & $m^{-}_{2}$ & $m^{+}_{2}$ & $1,2,5,6$ \\
\hline $2$ & - & - & $-m^{+}_{2}$ & $-m^{-}_{2}$ & $m^{-}_{2}$ & $-m^{+}_{2}$ & $3,4,7,8$ \\
\hline $3$ & $m^{-}_{2}$ & $m^{-}_{2}$ & - & - & $m^{-}_{2}$ & $m^{-}_{2}$ & $1,3,5,7$ \\ 
\hline $4$ & $-m^{+}_{2}$ & $m^{+}_{2}$ & - & - & $-m^{+}_{2}$ & $m^{+}_{2}$ & $2,4,6,8$ \\
\hline $5$ & $m^{+}_{2}$ & $m^{+}_{2}$ & $m^{-}_{2}$ & $m^{-}_{2}$ & - & - & $1,4,5,8$ \\
\hline $6$ & $-m^{-}_{2}$ & $m^{-}_{2}$ & $m^{+}_{2}$ & $-m^{+}_{2}$ & - & - & $2,3,6,7$ \\
\hline
\end{tabular}
\label{tableb8}
\caption{For each $s\in\left\{1,\ldots,6\right\}$, the table gives all the habit plane normals with $\lambda\in(1/2,1)$ that lie in the same octant as the Type-II twin normal for the possible values of $l$. The final column gives the octants in which the construction at the corner is possible.}
\end{table}

For example, let $s=1$. From Table~3 and the lattice parameters of Seiner's specimen, the Type-II twin normal is given by
\begin{equation*}
n=n^{II}=\left\{\begin{array}{rl}
(-0.688388,-0.688388,-0.228571)^T,& l=3\\
(0.688388,-0.688388,-0.228571)^T,& l=4\\
(-0.688388,-0.228571,-0.688388)^T,& l=5\\
(0.688388,-0.228571,-0.688388)^T,& l=6.
\end{array}\right.
\end{equation*}
For each $l\in\left\{3,\ldots,6\right\}$ and $\lambda\in(0,1/2)$, we see from Tables~4 and~5 that we may choose habit plane normal $m$ in the same octant as $n^{II}$ as follows:
\begin{equation*}
m=\left\{\begin{array}{rl}
m^{-}_{1},& l=3\\
-m^{+}_{1},& l=4\\
m^{+}_{1},& l=5\\
-m^{-}_{1},& l=6
\end{array}\right.
\end{equation*}
and similarly, for $\lambda\in(1/2,1)$ and Tables~6,~7,
\begin{equation*}
m=\left\{\begin{array}{rl}
-m^{+}_{2},& l=3\\
m^{-}_{2},& l=4\\
m^{-}_{2},& l=5\\
m^{+}_{2},& l=6.
\end{array}\right.
\end{equation*}

Note that we have multiplied some of the above habit plane normals by $-1$ implying that we also need to multiply the corresponding shears by $-1$ so that the tensor product $m\otimes b$ remains unaltered.
Then, for $l=3,5$ the Type-II twin normal lies in $O_5$ and we may choose a habit plane normal $m$ in the same octant such that neither $m$ nor $n^{II}$ are perpendicular to any of the edges of $\Omega$ (recall that the edges are along the principal cubic axes). 
Also, for $l=4,6$ the Type-II twin normal lies in $O_6$ and we may again choose an appropriate habit plane normal.

Clearly, by considering $-a^{II}$, $-n^{II}$ and respectively $-m$, $-b$, we may also choose twin and habit plane normals lying in the octants $O_1=-O_5$ and $O_2=-O_6$. The construction is now possible provided that the underlying deformation remains injective, i.e.~provided that
\[U^{-1}_{1}\tilde{b}\cdot n<0,\]
where $\tilde{b}=Rb$, $R\in SO(3)$ as in (\ref{eq:appendixb2}). Note that multiplying all the elements of the twin and habit planes by $-1$ leaves the above dot product unchanged. The rotation $R$ can be easily computed in each case by calculating
\[R=(U_1+\lambda a\otimes n)(\mathbf{1}+b\otimes m)^{-1}.\]
Then one finds that $U^{-1}_{1}\tilde{b}\cdot n=-0.0226521$ for any $l\in\left\{3,\ldots,6\right\}$ and any choice of the above habit plane elements.

The case $s=2,\ldots,6$ is identical to the above. The only non-trivial part is verifying the inequality
\[U^{-1}_{s}\tilde{b}\cdot n<0,\]
for the appropriate choice of habit plane normals. In fact, the value of the above dot product remains $-0.0226521$ for any $s$ and $l$, as long as we have chosen a habit plane normal that lies in the same quadrant as the Type-II twin normal.
This is surprising and leads one to conjecture that there must be some underlying structure for this to be true; nevertheless, algebraic complexity prevents us from unravelling this structure.

\def\cprime{$'$}


\end{document}